\documentclass{amsart}

\usepackage{amssymb,amsfonts,amsmath,amsthm,amsaddr}
\usepackage{enumitem}
\usepackage{color}
\usepackage[normalem]{ulem}
\usepackage{graphicx}
\usepackage{cancel}
\usepackage{algorithmic}
\usepackage{algorithm}

\newtheorem{theorem}{Theorem}[section]

\newtheorem{lemma}[theorem]{Lemma}
\newtheorem{proposition}[theorem]{Proposition}\date{\today}
\newtheorem{corollary}[theorem]{Corollary}
\newtheorem{remark}[theorem]{Remark}
\newtheorem{assumption}{Assumption}
\newtheorem{example}[theorem]{Example}

\newcommand{\E}{\ensuremath{\mathbb{E}}}
\newcommand{\Prob}{\ensuremath{\mathbb{P}}}
\newcommand{\Var}{\ensuremath{\mathrm{Var}}}

\newcommand{\Cov}{\ensuremath{\mathrm{Cov}}}
\newcommand{\ul}{\ensuremath{\lfloor t/\Delta\rfloor}}

\newcommand{\ulT}{\ensuremath{\lfloor T/\Delta\rfloor}}
\newcommand{\R}{\ensuremath{\mathbb{R}}}
\newcommand{\N}{\ensuremath{\mathbb{N}}}
\newcommand{\indicator}{\ensuremath{\mathbb{I}}}

\newcommand{\cA}{\mathcal{A}}

\newcommand{\cD}{\ensuremath{\mathcal{D}}}
\newcommand{\cF}{\mathcal{F}}

\newcommand{\cL}{\ensuremath{\mathcal{L}}}
\newcommand{\cO}{\ensuremath{\mathcal{O}}}
\newcommand{\cT}{\ensuremath{\mathcal{T}}}
\newcommand{\dd}{\,\mathrm{d}}
\newcommand{\dom}{\mathrm{dom}}
\newcommand{\tr}{\mathrm{tr}}
\newcommand{\diam}{\mathrm{diam}}
\newcommand{\inter}{\mathrm{int}}
\newcommand{\ver}{\mathrm{ver}}

\newtheorem*{lemma*}{Lemma} % temporary environments
\newtheorem*{remark*}{Remark} 
\newtheorem*{proposition*}{Proposition} 
\newtheorem*{theorem*}{Theorem}
\newtheorem*{corollary*}{Corollary}

\title[Nonparametric Inference for Noise Covariance Kernels]{
Nonparametric Inference for Noise Covariance Kernels in Parabolic SPDEs using Space-Time Infill-Asymptotics
}

\author{Andreas Petersson}
\address{Department of Mathematics, Linnaeus University, 351 95, V\"axj\"o, Sweden}

\author{Dennis Schroers}
\address{Institute of Finance and Statistics and Hausdorff Center for Mathematics, University of Bonn, Adenauerallee 24-42, 53113 Bonn, Germany}

\subjclass{60H15, 62M20, 62G05, 62G10, 35R60}

\keywords{Stochastic partial differential equations (SPDEs), noise covariance estimation, goodness-of-fit test, nonparametric statistics, space-time infill asymptotics, functional data analysis}

\begin{document}

	\begin{abstract}
		We develop an asymptotic limit theory for nonparametric estimation of the noise covariance kernel in linear parabolic stochastic partial differential equations (SPDEs) with additive colored noise, using space–time infill asymptotics. The method employs discretized infinite-dimensional realized covariations and  requires only mild regularity assumptions on the kernel to ensure consistent estimation and asymptotic normality of the estimator. On this basis, we construct omnibus goodness-of-fit tests for the noise covariance that are independent of the SPDE’s differential operator. Our framework accommodates a variety of spatial sampling schemes and allows for reliable inference even when spatial resolution is coarser than temporal resolution. 
	\end{abstract}
	\maketitle
	%	\tableofcontents
	
\section{Introduction and preliminaries}
\subsection{Introduction}
Parabolic stochastic partial differential equations (SPDEs) are often used in spatiotemporal modeling, where they frequently emerge as deterministic PDEs perturbed by additive random noise. This noise is sometimes chosen to be white in time and colored in space, with its spatial covariance specified through a tractable parametric form such as the Matérn class or a fractional inverse of the Laplacian. This choice is delicate, as misspecifications may propagate through subsequent analyzes and predictions. Yet, model validation and selection of noise covariance models is difficult, since research on nonparametric methods for their estimation is sparse. In this article, we show that such methods can be derived under space–time infill asymptotics and develop a framework for omnibus goodness-of-fit testing of noise covariance models.

We consider linear parabolic SPDEs with colored additive noise on a bounded Lipschitz domain $\cD \subset \R^d$  for $d \in \{1,2,3\}$:
\begin{equation}
	\label{eq:spde}
	\dd X_t + A X_t = \dd W_t, \quad t \ge 0,
\end{equation}
where $X_t$ takes values in the Hilbert space $H=L^2(\cD)$. Here $A$ denotes a possibly unknown elliptic (unbounded) operator and $W_t$ is a $Q$-Wiener process with spatial trace-class covariance $Q$ which is the object of interest throughout.
The precise conditions on the coefficients of the equation are described in Section \ref{Sec: Parabolic SPDEs on bounded domains}.

The main novelty of our work compared to existing literature is that we estimate $Q$---or equivalently the kernel $q$ such that $Qf(x)=\int_{\mathcal D} q(x,y)f(y)dy$ for all $f\in L^2(\mathcal D)$ and $x\in \mathcal D$--- 
globally in Hilbert-Schmidt norm (resp. $L^2(\mathcal D^2)$ norm) 
while $A$ might be unknown, which is often the case in practice. A canonical example are stochastic advection-diffusion equations for which 
$$A= -a \Delta_\cD  + \nabla \cdot \mathbf{b}+ c ,$$
where $\Delta_\cD$ and $ \nabla \cdot$ are respectively the Laplace and divergence operator and for which the diffusivity,  damping and velocity parameters $a,c>0$ and $\mathbf{b}\in \mathbb R^d$ are typically unknown a priori. 
In such cases, our theory enables omnibus goodness-of-fit testing for the noise covariance model prior to the estimation of $A$. Such tests are of practical importance,  since spatially colored noise is often considered in applications of SPDEs to space-time data analysis (c.f. \cite{Lingren2011}, \cite{lindgren2020}, \cite{Sigrist2012}, \cite{Sigrist2015} \cite{liu2022}).

As an estimator for the operator $Q$, we develop an asymptotic theory for the infinite-dimensional realized covariation
\begin{align}\label{eq:realized-variation}
	RV_t^{\Delta,O}:= \sum_{i=1}^{\ul} (O X_{i\Delta}-O X_{(i-1)\Delta})^{\otimes 2},\quad t\geq 0,
\end{align}
where $\Delta>0$ is a temporal resolution with which the data are observed,   $f^{\otimes 2}=\langle f,\cdot\rangle_{L^2(\cD)} f$ for all $f\in L^2(\cD)$ and $O X_{0},...,O X_{\Delta\ul}$ are approximations of the functional data  $ X_{0},..., X_{\Delta\ul}$ derived from discrete space-time data.  For instance,  if $\mathcal D=(0,1)$ and $t=1$  we might have discrete space-time data
$$X_{i\Delta}(j\delta),\quad j=0,1,...,\lfloor 1/\delta\rfloor, i=0,1,...,\lfloor 1/\Delta\rfloor,$$
where $\delta >0$ is a spatial resolution and $O$ is the operator that conducts a linear interpolation in space.   A precise description of the sampling schemes that we consider (perfectly observed functional data, spatial averages and pointwise spatial sampling) can be found in Section \ref{Sec: Sampling setting}.

We show under infill-asymptotics (spatial and temporal resolution converging to $0$) that $RV_t^{\Delta,O}/t$ consistently estimates $Q$ under mild regularity assumptions and provide rates of convergence that depend on the regularity of $Q$ and the sampling setting at hand. 
On that foundation, we establish central limit theorems which are the basis of asymptotic omnibus goodness-of-fit tests via the test statistic $T_n:= n\|RV_t^{\Delta,O}-Q\|^2_{\cL_2(H)}$ where $\|\cdot\|_{\cL_2(H)}$ denotes the Hilbert-Schmidt norm.  To be precise, under certain regularity conditions on $Q$, we can test for 
$$H_0: Q=Q_0\quad vs\quad  H_1: Q\neq Q_0,$$
where $Q_0$ is a hypothesized noise covariance
as well as the joint hypothesis
\begin{align*}
	H_0: Q =Q_{\theta}  \text{ for some }\theta \in \Theta \quad \text{ vs. }\quad H_1: Q \neq Q_{\theta}  \text{ for all }\theta\in \Theta,
\end{align*}
for a parametric class of covariances $\{Q_{\theta}:\theta\in \Theta\}$ and a set of parameters $\Theta$.

We demonstrate the finite sample performance of the proposed  tests and of  $RV_t^{O,\Delta}/t$ as  an estimator of $Q$ in a simulation study in the context of Mat{\'e}rn  and  inverse fractional Laplacian covariance classes. Both our theory as well as the results from the simulation study show that due to the regularity induced by the analyticity of $A$ and the trace-class property of $Q$ reliable estimation of the latter is possible even when spatial resolution is much coarser than temporal resolution.

To the best of our knowledge the approach to goodness-of-fit testing of noise covariance models via infinite-dimensional realized covariations is new. 
We believe that the techniques and methods in this article can be used as an ansatz for statistical noise analyses in the context of more general SPDEs governed by a leading second-order differential operator and lead to various applications.   
The proofs of our limit theory make use of tools from polynomial approximation theory in a new context,  which we conjecture  to be generalizable in various ways and serve as a fundament to treat noise kernel estimation for more general dynamics including  anisotropic stochastic advection-diffusions,  nonlinear drifts or stochastic volatility. 

Our results on the asymptotic behaviour of realized covariations complement the existing statistical theory for SPDEs. In the majority of works on statistics for SPDEs the focus is on the estimation of finite-dimensional functionals of the noise covariance. In these settings, $Q$ is typically not trace-class and the solution process has the same regularity as the stochastic heat equation with space-time white noise. Examples include \cite{Hildebrandt2021}, \cite{hildebrandt2023}, and \cite{Uchida2020}, who used space-time infill asymptotics to estimate volatility ($Q=\sigma I$) and other parameters of linear or semilinear SPDEs. \cite{Bibinger2019} studied high-frequency temporal samples at finitely many spatial points for one-dimensional linear SPDEs with $Q=I$ and time-varying deterministic volatility. \cite{Chong2020} and \cite{ChongDalang2020} extended this to stochastic volatility kernels, while \cite{Cialenco2020} considered discrete space-time observations with possibly multiplicative noise.
Related contributions on parameter estimation for the operator $A$ in non–trace-class settings include the approaches via local measurements (\cite{altmeyer2021}, \cite{altmeyer2024}, \cite{altmeyer2023}, \cite{Reiss23}), small-diffusivity (\cite{Gaudlitz2023}),  and the spectral approach (see \cite{Cialenco2018} for a survey).
Global estimation of trace-class $Q$ in operator norms as in this article is considered in \cite{Benth2022}, \cite{BSV2022}, \cite{Schroers2024}, and \cite{Schroers2024b}. These works assume that the differential operator $ A$ is known a priori and that continuous observations in space are available (or at least, well-approximated). Our results show that these conditions are usually not needed when $ A$ is elliptic.
Our theory can also be contrasted with the classical limit theory for finite-dimensional semimartingales in the context of convergence rates. To be precise, when $X$ takes values in $\mathbb{R}^m$ and $A,Q$ are matrices, $X$ is an Ornstein–Uhlenbeck process, and it is well known that  the variance of realized covariation vanishes at rate $\sqrt{\Delta}$ and the bias at rate $\Delta$ (as for more general It{\^o} semimartingales, c.f. \cite[Theorem  5.4.2]{JacodProtter2012}). In our framework, by contrast, the variance (at least when $O=I$) still vanishes at $\sqrt{\Delta}$, but we prove that the bias may decay arbitrarily slowly.

The rest of this article is organized as follows: We first provide some necessary preliminaries in Section \ref{Sec: Preliminary Notation}. Section \ref{Sec: Identifiability} discusses the general identifiability of $Q$ via realized covariations under perfect spatial observations.
Section \ref{Sec: Sampling setting} provides a discussion on the spatial sampling settings that we consider, while Section \ref{Sec: Asymptotic Theory} contains our asymptotic theory. Section \ref{Sec: Universal Goodness of fit test} describes the goodness-of-fit test for noise-covariance kernels, while we provide finite sample-evidence for our methods in a simulation study in Section \ref{Sec: Simulation Study}.  Formal proofs of our asymptotic results can be found in Section \ref{Sec. Proofs} of the Appendix.

\subsection{Preliminaries and notation}\label{Sec: Preliminary Notation}

Let $(H, \langle \cdot , \cdot \rangle_{H})$ and $(U, \langle \cdot , \cdot \rangle_{U})$ be real separable Hilbert spaces. The space of bounded linear operators from $H$ to $U$ is represented by $\cL(H,U)$, the space of Hilbert-Schmidt operators by $\cL_2(H,U)$ and the space of trace-class operators by $\cL_1(H,U)$. The $\cL(H)$ and $\cL_2(H)$ shorthand notations are used when $U=H$. We write $U \hookrightarrow H$ when the embedding $U \subset H$ is continuous and $U \cong H$ if $H \hookrightarrow U \hookrightarrow H$. 

The space $\cL_2(H,U)$ is itself a separable Hilbert space with inner product given by $\langle \Gamma_1 , \Gamma_2 \rangle_{\cL_2(H,U)} = \sum_{j = 1}^{\infty} \langle \Gamma_1 e_j , \Gamma_2 e_j \rangle_U$ for $\Gamma_1, \Gamma_2 \in \cL_2(H,U)$ and an arbitrary orthonormal basis $(e_j)_{j = 1}^\infty$ of $H$. This space is also an operator ideal.

We identify $\cL_2(H,U)$ and $U \otimes H$ (the Hilbert tensor product). The tensor $u \otimes v$ is viewed as an element of $\cL(H,U)$ by the relation $(u \otimes v) w = \langle v , w \rangle_{H} u$ for $v,w \in H$ and $u \in U$ so that $\| u \otimes v \|_{\cL_2(H,U)} = \| u \|_{U} \| v \|_{H}$. Also, $\langle \Gamma , u \otimes v \rangle_{\cL_2(H,U)} = \langle \Gamma v , u \rangle_{U}$ for $\Gamma \in \cL_2(H,U)$. If $V$ and $G$ are additional real and separable Hilbert spaces, then $\Gamma_1 u \otimes \Gamma_2 v = \Gamma_1 (u \otimes v) \Gamma_2^*$ for $u \in U$, $v \in H$, $\Gamma_1 \in \cL(U,V)$ and $\Gamma_2 \in \cL(H,G)$. The adjoint of $\Gamma_1$ is denoted by $\Gamma_1^*$. Note that $\|\Gamma_1\|_{\cL_2(U,V)} = \|\Gamma_1^*\|_{\cL_2(V,U)}$. 

The class $\cL_1(H,U)$ is a Banach space under the norm $\| \Gamma \|_{\cL_1(H,U)} = \tr(\Gamma^* \Gamma)^{1/2})$ and also an operator ideal. For an additional real and separable Hilbert space $V$ and any $\Gamma_1 \in \cL_2(H,U)$, $\Gamma_2 \in \cL_2(V,H)$ we have $\Gamma_1 \Gamma_2 \in \cL_1(V,U)$ with $\| \Gamma_1 \Gamma_2 \|_{\cL_1(V,U)} \le \| \Gamma_1\|_{\cL_2(H,U)} \| \Gamma_2\|_{\cL_2(V,H)}$. In particular $\| \Gamma_1 \|_{\cL_2(H,U)}^2 = \| \Gamma_1^* \Gamma_1\|_{\cL_1(H)}$.

Given a bounded domain $\cD \subset \R^d$ with Lipschitz boundary, $d \in \{1,2,3\}$, we let $H^r := W^{r,2}(\cD)$, $r \in [0,\infty)$ denote the (fractional) Sobolev space of order $r$. For non-integer $r$, it is equipped with the Sobolev--Slobodeckij norm
\begin{equation}
	\label{eq:sobolev-slobodeckij}
	\|u\|^2_{H^r} = \|u\|_{H^m}^2 + \sum_{|\alpha| = m} \int_{\cD \times \cD} \frac{|D^\alpha u(x) - D^\alpha u(y)|^2}{|x-y|^{d+2\sigma}}\dd x \dd y, \quad u \in H^r,
\end{equation}
$r = m + \sigma$, $m \in \N_0, \sigma \in (0,1)$, where $\|\cdot\|_{H^m}^2$ is the usual integer order Sobolev norm.

This paper employs generic constants $C$. These vary from occurrence to occurrence and are independent of any parameter of interest, such as the number of SPDE observations. We write $a \lesssim b$ for $a,b \in \R$ to signify $a \le C b$ for such a generic constant $C$.

\section{Consistency of Realized Variations under perfect spatial observations}\label{Sec: Identifiability}
This section considers the identifiability of 
$Q$ under perfect spatial observations, providing a foundation for the asymptotic theory in Section \ref{Sec: Asymptotic Theory}.

Let $\cD \subset \R^d, d \in \{1,2,3\}$, be a bounded domain with Lipschitz boundary and let $H = L^2(\cD)$ be the Hilbert space of square-integrable functions on $\cD$. We let $(\Omega, \cA, (\cF_t)_{t \ge 0}, P)$ be a complete probability space equipped with a normal filtration $(\cF_t)_{t \ge 0}$. 

If $-A$ is the generator of a strongly continuous semigroup $(S(t))_{t\geq 0}$ in $L^2(\cD)$ and $X_0 \in L^2(\Omega,H)$ is $\cF_0$-measurable, the SPDE \eqref{eq:spde} has a mild solution $(X_t)_{t\ge 0}$, i.e., $X$ is given by the stochastic Volterra process
\begin{equation}\label{eq:mild}
	X_t=S(t)X_0+\int_0^t S(t-s) dW_s,\quad t \ge 0.
\end{equation}
In general, for such processes $RV_T^{\Delta,I}/T$ is not always a consistent estimator of $Q$ (c.f. Example 4(ii) in \cite{BSV2022}). However, the subsequent 
mild Assumption is sufficient  to guarantee identifiability of  $Q$ via $ RV_T^{\Delta,I}$.  
\begin{assumption}\label{ass: minimal Assumption on A}
	The operator $A$ is closed, densely defined, self-adjoint, positive definite on $H$ and has a compact inverse. Moreover, $Q$ is trace-class and the initial condition satisfies $X_0 \in L^4(\Omega,H)$.
\end{assumption}
Under Assumption \ref{ass: minimal Assumption on A}, the operator $-A$ generates an analytic semigroup $S = (S(t))_{t \ge 0}$, characterized by the spectral representation $S(t) e_j = e^{-\lambda_j t} e_j$ for $j \in \mathbb{N},\ t \ge 0$. In particular, this ensures the existence of a mild solution. Moreover, the analytic structure of the semigroup exerts a regularizing effect on the dynamics of $X$, which is sufficient to establish the following essential identifiability result.
\begin{theorem}\label{thm: Identifiability}
	If Assumption \ref{ass: minimal Assumption on A} holds, we have as $\Delta \downarrow 0$ that
	\begin{equation*}
		\frac{RV_T^{\Delta,I}}T \to Q \quad \text{ in }L^2(\Omega, \cL_2(H)).
	\end{equation*}
\end{theorem}

Under additional regularity conditions on $Q$ and $A$, we derive convergence rates and a central limit theorem for the realized covariation, allowing also for discrete spatial sampling.

\section{Sampling setting}\label{Sec: Sampling setting}

This section presents the discrete spatial sampling schemes we examine. Precisely,
for a sequence $(i\Delta)_{i = 0, \ldots, \ulT}$ of temporal observation points 
we consider three spatial observation settings. To each of them we define an \textit{observation operator} $O$ which maps $X$ from $H$ into either $H$ (Setting~\ref{obs:cont} below) or a discrete subspace thereof (Settings~\ref{obs:average} and \ref{obs:pointwise}). 

\begin{enumerate}[label=(\Alph*)]
\item (Continuous sampling) \label{obs:cont} We observe $X_{i\Delta}, i = 1, \ldots, \ulT$ as $\Delta  \to 0$. 
Here, we set $O:=I$, the identity operator on $H$.
\item (Sampling of local averages) \label{obs:average} We observe 
\begin{equation*}
	\frac{1}{|T_h|} \int_{T_h} X_{i\Delta}(x) \dd x, \quad T_h \in \cT_h, h := \max_{T_h} \diam(T_h) \in (0,1], i = 1, \ldots, \ulT,
\end{equation*}
for a non-degenerate family $(\cT_h)_{h \in (0,1]}$ of subdivisions of $\cD$, as $\Delta  \to 0$, $h \to 0$. 
We set $O := P_h$, the orthogonal projection from $H$ to the space of piecewise constant functions on $T_h$. This means we can write
\begin{equation*}
	O X_{i\Delta} = P_h X_{i\Delta} = \sum_{T_h \in \cT_h} \frac{\indicator_{T_h}}{|T_h|} \int_{T_h} X_{i\Delta}(x) \dd x, 
\end{equation*}
where $\indicator_{T_h}(x) = 1$ if $x \in T_h$ and $0$ otherwise.
\item (Pointwise sampling) \label{obs:pointwise} We assume $\cD$ is polygonal and observe 
\begin{equation*}
	X_{i\Delta}(x_j), \quad x_j \in \ver(\cT_h), h := \max_{T_h} \diam(T_h) \in (0,1], i=1,...,\ulT,
\end{equation*}
for a non-degenerate family $(\cT_h)_{h \in (0,1]}$ of triangular subdivisions of $\cD$, as $\Delta  \to 0$, $h \to 0$. 
Here $\ver(\cT_h)$ is the set of vertices (or nodes) of $\cT_h$. We set $O := I_h$, the piecewise linear interpolant which takes a continuous function on $\cD$ to its piecewise linear interpolation on $\cT_h$. This means we can write
\begin{equation*}
	O X_{i\Delta} = I_h X_{i\Delta} = \sum_{x_j \in \ver(\cT_h)} X_{i\Delta}(x_j) \phi^j_{h},
\end{equation*}
where $\phi^j_{h}$ is the nodal basis function associated to $x_j$, i.e., the unique piecewise linear function on $\cT_h$ for which $\phi^j_{h}(x_i) = \delta_{ij}$ for all $x_i \in \ver(\cT_h)$. The random variables $X_{t_i}(x_j)$ and $I_h X_{t_i}$ should be understood in terms of cylindrical Gaussian random variables, and are (under  Assumption~\ref{ass:Ih} below) well-defined by Proposition~\ref{prop:cylindrical-pointwise} and Lemma~\ref{lem:obs-hs}. For technical reasons related to the proof of Lemma~\ref{lem:obs-hs}, we also assume the subdivision is quasiuniform when $d=1$. This is a mild condition that simplifies the analysis but is not essential to the result.
\end{enumerate}
We refer to these observation settings by specifying the value of the observation operator $O$. That is to say, when we write $O = I$, we refer to the observation setting~\ref{obs:cont} and so on.

Some remarks are in order.	

\begin{remark}[On the assumptions on the subdivisions]
By a subdivision $\cT_h$, we refer to a finite collection $\{T_h^j\}$ of closed subsets of $\R^d$ with nonempty interiors such that $\inter(T_h^i) \cap \inter(T_h^j) = \emptyset$ for $i \neq j$ and $\bar{\cD} = \cup_{j} T_h^j$. It is called triangular if each $T_h \in \cT_h$ is an interval, triangle or tetrahedron, depending on the dimension of the underlying euclidean space. Note that any set $\{x_j\}$ of points in $\bar{\cD}$ gives rise to a triangular subdivision as long as $\{x_j\}$ contains all vertices of $\partial \cD$ and there is at least one vertex in $\cD$. We call $(\cT_h)_{h \in (0,1]}$ non-degenerate if there exists a $C > 0$ such that for all $T_h \in \cT_h$ and all $h \in (0,1]$, $\diam(B_{T_h}) \ge C \diam(T_h)$, where $B_{T_h}$ is the largest ball in $T_h$ such that $T_h$ is star-shaped  with respect to $B_{T_h}$ (that is, the closed convex hull of $\{x\} \cup B_{T_h}$ is a subset of $T_h$). We call $(\cT_h)_{h \in (0,1]}$ quasiuniform if the ratio of the largest to the smallest element diameter is uniformly bounded across all $h \in (0,1]$. Uniform meshes of points on a cube $\cD = [0,1]^d$ with decreasing mesh size $h$ is a particular instance of such a family.  
\end{remark}

\begin{remark}[Occurrence of the different sampling schemes] The three considered sampling schemes are important in different theoretical contexts and applications:
\begin{itemize}
	\item[(A)] The continuous spatial sampling scheme (A) is not met in practice where observations in both time and space are discrete. However, it is still practially important when we have dense data in space that are sampled with noise. In this case we might approximate the data $X_{i\Delta},i=1,...,\ulT$ in a way in which the discretization error in space gets negligible and the asymptotics for the continuous observation scheme apply.
	\item[(B)] The sampling scheme of local averages occurs in various applications,
	for instance, when temperature is measured as an average temperature within a demarcated area or in the bond market, where log-bond prices are local averages of instantaneous forward rates, which are in econometric contexts sometimes modelled as an equation of type \eqref{eq:spde} (c.f. \cite{Cont2005}).
	\item[(C)] The pointwise sampling scheme (C) is commonly considered in the literature on statistics for SPDEs (c.f. e.g. \cite{Bibinger2019} or \cite{Chong2020}). It corresponds to sampling at particular points in space.
\end{itemize}
\end{remark}
\begin{remark}[On the kernel form of the realized covariation]
It is instructive to consider the kernel form of the estimator $\mathrm{RV}_T^{\Delta,O}/T$ for $q$ in the discrete observation settings. Let us write $\Delta_i X(x_j) := X_{i\Delta}(x_j) - X_{(i-1)\Delta}(x_j)$ for the observed increment at a spatial location $x_j$, and $\overline{\Delta_i X}_{T_h} := |T_h|^{-1}\int_{T_h} (X_{i\Delta}(z)-X_{(i-1)\Delta}(z))\dd z$ for the increment averaged over a set $T_h \in \cT_h$. In the case of pointwise sampling ($O=I_h$), the estimator for $q(x,y)$ has a piecewise bilinear kernel constructed by interpolating the sample covariance matrix of the pointwise increments:
\begin{equation} 
	\label{eq:pointwise-RV-kernel-form}
	(x,y) \mapsto \frac{1}{T}\sum_{x_j, x_k \in \ver(\cT_h)} \left( \sum_{i=1}^{\ulT} \Delta_i X(x_j) \Delta_i X(x_k) \right) \phi_h^j(x) \phi_h^k(y). \end{equation}
When sampling local averages ($O=P_h$), the estimator for $q(x,y)$ has instead a piecewise constant kernel on the grid $\cT_h \times \cT_h$, given by
\begin{equation} 
	\label{eq:spatav-RV-kernel-form}
	(x,y) \mapsto \frac{1}{T}\sum_{T_h^k, T_h^j \in \cT_h} \left( \sum_{i=1}^{\ulT} \overline{\Delta_i X}_{T_h^k} \overline{\Delta_i X}_{T_h^j} \right) \indicator_{T_h^k}(x) \indicator_{T_h^j}(y). 
\end{equation}
In both settings, the estimator translates the empirical covariance structure from the observed discrete data into a functional form.
\end{remark}

In the perfect observation setting, no spatial discretization error has to be accounted for. This aspect becomes relevant in settings (B) and (C).
We therefore now state two important lemmata on the spatial approximation adequacy under the different observations schemes (B) and (C). The proofs can be found in Section \ref{Sec: Proofs of Section 3}. The first lemma is a small extension of a well-known theorem of polynomial approximation theory.
\begin{lemma}
\label{lem:obs-error}
For a constant $C < \infty$ that is independent of $h \in (0,1]$,
\begin{equation*}
	\| O - I \|_{\cL(H^r,H)} \le 
	\begin{cases}
		C h^{\min(r,1)} \text{ if } r \ge 0 \text{ and } O = P_h, \\
		C h^{\min(r,2)} \text{ if } r > d/2 \text{ and } O = I_h.
	\end{cases}
\end{equation*}
\end{lemma}
Since we consider cylindrical Gaussian random variables, we also need to establish a Hilbert--Schmidt bound on $I_h$ that is uniform in $h$.
\begin{lemma}
\label{lem:obs-hs}
For all $r > d/2$, there exists an $\epsilon > 0$ and a constant $C < \infty$ that is independent of $h \in (0,1]$, such that $\| I_h \|_{\cL_2(H^r,H^\epsilon)} \le C$.
\end{lemma}

With observation schemes established, we next consider the asymptotic theory.

\section{Asymptotic Theory}\label{Sec: Asymptotic Theory}
In this section, we establish convergence rates, central limit theorems, and the required regularity conditions. We begin with a discussion of the latter.

\subsection{Regularity Conditions}\label{Sec: Parabolic SPDEs on bounded domains}
To derive rates of convergence and a characterization of the asymptotic distribution of the realized covariation \eqref{eq:realized-variation} under the different sampling schemes presented in Section \ref{Sec: Sampling setting}, we impose additional regularity assumptions, which are the subject of this subsection.

To formulate these conditions, we also need the notion of fractional powers of densely defined and positive definite operators $A$ with compact inverse.
By the spectral theorem we obtain for such operators a sequence $(\lambda_j)_{j=1}^\infty$ of positive non-decreasing eigenvalues of $A$ with $\lim_{j \to \infty} \lambda_j = \infty$, along with an orthonormal eigenbasis $(e_j)_{j=1}^\infty$ in $H$. Negative fractional powers of $A$ are bounded, positive fractional powers are defined by
\begin{equation*}
	A^{\frac{r}{2}} v := \sum_{j = 1}^{\infty} \lambda_j^{\frac{r}{2}} \langle v , e_j\rangle e_j, \quad v \in \dot{H}^r := \dom(A^{\frac{r}{2}}) = \Big\{ v \in H : \sum^\infty_{j=1} \lambda_j^r |\langle v, e_j\rangle|^2 < \infty \Big\}.
\end{equation*}
The space $\dot{H}^r$ is a Hilbert space with respect to the graph inner product $\langle \cdot, \cdot \rangle_{\dot{H}^r} = \langle A^{r/2} \cdot, A^{r/2} \cdot \rangle$.
For $r<0$, we define $\dot{H}^r$ to be the completion of $H$ under the same graph norm. Then, the embedding $\dot{H}^r \hookrightarrow \dot{H}^s$, $r > s$ is dense and compact, and $A^{r/2}$ extends to an isometric isomorphism $A^{r/2}  \colon \dot{H}^{s} \to \dot{H}^{s-r}$, $r,s \in \R$.

We are now in the position to state the the regularity assumptions which we  impose to derive rates of convergence and a central limit theorem for the realized covariation estimator. 

We impose a sequence of increasingly stronger assumptions, that apply for $O = I$, $O = P_h$ and $O = I_h$, respectively. 
\begin{assumption}\label{ass:rate} Assumption \ref{ass: minimal Assumption on A} is fulfilled. Additionally, for some $\gamma, \iota \ge 0$, 
	\begin{enumerate}[label=(\Alph*)]
		\item \label{ass:rate:q} $\| A^{\gamma/2} Q \|_{\cL_2(H)} < \infty$ and 
		\item \label{ass:rate:initial} $\| A^{\iota/2} X_0 \|_{L^4(\Omega,H)}  < \infty$.
	\end{enumerate}
\end{assumption}
In order to make use of Lemma~\ref{lem:obs-error} for $O=P_h$, a Sobolev regularity condition on $A$ is required, namely $\dot{H}^r\hookrightarrow H^r$ for $r \in [0,1]$, which is why we impose the following.
\begin{assumption}\label{ass:Ph} Assumption \ref{ass:rate} is fulfilled. Additionally, $\dot{H}^1 \hookrightarrow H^1$.
\end{assumption}
For $O=I_h$, we also impose a corresponding, stronger, assumption that yield higher spatial convergence rates. We also need higher regularity to allow for pointwise observations, but due to Lemma~\ref{lem:obs-hs}, we may impose this in operator norm, rather than the stronger Hilbert--Schmidt norm.
\begin{assumption}\label{ass:Ih} Assumption \ref{ass:rate} is fulfilled. Additionally, for some $\beta \ge 0$, $\eta, \zeta > \max(d/2-1,0)$,
	\begin{enumerate}[label=(\Alph*)]
		\item \label{ass:Ih:q:1} $\| A^{(\beta +d-1+\epsilon)/4} Q^{\frac 12} \|_{\cL(H)} < \infty$ for some $\epsilon > 0$, 
		\item \label{ass:Ih:q:2} $\| A^{\zeta/2} Q A^{\eta/2} \|_{\cL(H)} < \infty$, and
		\item \label{ass:Ih:ell-reg} $\dot{H}^2 \hookrightarrow H^2$.
	\end{enumerate}
\end{assumption}
We discuss these assumptions in the remainder of this subsection. We first consider the regularity conditions on the elliptic operator $A$ and then the noise covariance $Q$.

\subsubsection{Regularity conditions on the operator $A$}
\label{Sec: elliptic regularity discussion}

The conditions $\dot{H}^1 \hookrightarrow H^1$ from Assumption~\ref{ass:Ph} and the stronger $\dot{H}^2 \hookrightarrow H^2$ from Assumption~\ref{ass:Ih} can be seen as Sobolev regularity conditions on $A$. The latter is often called the shift property or elliptic regularity. We first describe the canonical example of $A$ as a symmetric elliptic differential operator, considering two different boundary conditions.

\begin{example}
	\label{ex:elliptic}
	Let $\mathcal{D} \subset \mathbb{R}^d$ be a convex polygon (or polyhedron). Let $a_{i,j} \in C^1(\bar{\mathcal{D}})$ with $a_{i,j} = a_{j,i}$, and suppose there exists a constant $\lambda_0 > 0$ such that for all $y \in \mathbb{R}^d$ and $x \in \mathcal{D}$,
	\begin{equation*}
		\sum_{i,j=1}^d a_{i,j}(x) y_i y_j \ge \lambda_0 |y|^2.
	\end{equation*}
	Let $c \in L^\infty(\mathcal{D})$ be a non-negative function. In the case of Neumann boundary conditions, assume further that $c \ge c_0 > 0$ on $\mathcal{D}$.
	Define the operator $A$ by
	\begin{equation}
		\label{eq:elliptic-def}
		(Au)(x) = -\sum_{i,j=1}^d D^i\big(a_{i,j}(x) D^j u(x)\big) + c(x) u(x), \quad x \in \mathcal{D},
	\end{equation}
	where $D^j u$ denotes the weak derivative of $u$ with respect to $x_j$. Then, $A$ is an unbounded operator whose domain is
	\begin{equation*}
		\dot{H}^2 = \{ u \in H^2(\mathcal{D}) \cap H^1_0(\mathcal{D}) \}
	\end{equation*}
	when it is equipped with zero Dirichlet boundary conditions, and
	\begin{equation*}
		\dot{H}^2 = \left\{ u \in H^2(\mathcal{D}) : \sum_{j=1}^d a_{i,j}(x) D^j u(x) n_i(x) = 0 \text{ on } \partial\mathcal{D} \right\}
	\end{equation*}
	when it is equipped with zero Neumann boundary conditions. Here $n = (n_1, \dots, n_d)$ is the outward unit normal vector on $\partial\mathcal{D}$.
\end{example}
The operator $A$ in this example is a densely defined, closed, and positive definite operator on the Hilbert space \( H = L^2(\mathcal{D}) \). It possesses a compact inverse and clearly $\dot{H}^2 \hookrightarrow H^2$. The operator may also be defined through a bilinear form $a(\cdot, \cdot)$ on a subset of $H^1 \times H^1$. The regularity of $a_{i,j}$ can then be relaxed and $\mathcal{D}$ need not be convex. This construction directly yields the equivalence $\dot{H}^1 \cong H^1$. However, with the $C^1$ coefficients and convex polygonal domain of Example~\ref{ex:elliptic}, we also have the equivalence $\dot{H}^r \cong H^r$ for $r \in [0, \theta)$, where
\begin{equation}
	\label{eq:sobolev_id_theta_def}
	\theta = \begin{cases}
		\frac{1}{2} & \text{for Dirichlet boundary conditions}, \\
		\frac{3}{2} & \text{for Neumann boundary conditions}.
	\end{cases}
\end{equation}
This allows us to translate the abstract regularity conditions on $Q$ into more concrete and verifiable smoothness conditions on its kernel $q$. We refer to \cite{Yagi2010}, in particular to \cite[Theorems~16.9, 16.13]{Yagi2010}, for the full details.

Under further restrictions on the coefficients and the domain, it is possible to include first-order advection terms in the operator $A$ by transforming it into a self-adjoint operator on a weighted Hilbert space, similar to the setting of \cite{Bibinger2020}.
\begin{example}
	\label{ex:advection-diffusion}
	Let $\mathcal{D}$ be as in Example~\ref{ex:elliptic} and let
	\[ (Au)(x) = -a(x) \Delta_\cD u(x) + \mathbf{b}(x) \cdot \nabla u(x) + c(x) u(x), \]
	with $\Delta_\cD$ and $\nabla$ the Laplace operator and divergence operator and
	$a \in C^2(\bar{\mathcal{D}})$ bounded from below by some $\lambda_0 > 0$, $\mathbf{b} \in (C^1(\bar{\mathcal{D}}))^d$, and with $c$ fulfilling the same conditions as in Example~\ref{ex:elliptic}. We further assume that
	\[ \mathbf{V}(x) = \frac{\mathbf{b}(x) + \nabla a(x)}{a(x)} \]
	satisfies
	\( \partial V_j / \partial x_k = \partial V_k / \partial x_j\) for all \(j,k = 1, \dots, d\). This is for instance true if $d=1$ or if $a$ is constant and $\mathbf{b}$ a conservative vector field.
	This condition ensures that $\nabla \phi = -\mathbf{V}$, where $\phi$ is for an arbitrary fixed point $x_0 \in \mathcal{D}$ given by
	\[ \phi(x) = -\int_0^1 \mathbf{V}(x_0 + t(x - x_0)) \cdot (x-x_0) \, dt. \]
	Now, we define the weight function $w(x) = \exp(\phi(x))$ and introduce the weighted space $H_w = L^2(\mathcal{D}, w(x)\dd x)$ with inner product $\langle u, v \rangle_w = \int_{\mathcal{D}} u(x)v(x) w(x) \dd x$. Using $\nabla w = w \nabla\phi = -w\mathbf{V}$, a direct calculation shows that for smooth $u$,
	\[ -\nabla \cdot (w(x)a(x) \nabla u(x)) + w(x)c(x)u(x) = w(x) (Au)(x). \]
	From this, it follows that $A$ is a symmetric operator in $H_w$. 
\end{example}
From the fact that the weighted operator $w A$ is of the same form as Example~\ref{ex:elliptic}, it follows that when we equip $A$ with the boundary conditions of this example, it is densely defined, closed, and positive definite on $H_w$ with a compact inverse. Its domain is again given by either of the spaces of Example~\ref{ex:elliptic}, depending on the boundary condition. The weight function $w(x)=\exp(\phi(x))$ is continuous and strictly positive, so that the weighted and unweighted spaces of Examples \ref{ex:elliptic} and \ref{ex:advection-diffusion} fulfill $H_w \cong H$. Since $\dot{H}^r$ is an interpolation space (see \cite[Chapter~1]{LionsMagenes1972}), not only does Assumption~\ref{ass:Ih}\ref{ass:Ih:ell-reg}) hold, but the identification of \eqref{eq:sobolev_id_theta_def} remains valid for the same range of $r$ as the operator of Example~\ref{ex:elliptic}. The operator thus fits into our framework, provided we work in $H_w$ instead of $H$. 

Let \(M_w\) be the bounded, invertible operator on \(H\) corresponding to multiplication by the weight function \(w(x)\). The inner products are related by \(\langle u, v \rangle_w = \langle u, M_w v \rangle_H\). Consequently, the covariance operator of the noise process in the space \(H_w\) (denoted $Q_w$) is related to the original covariance \(Q\) in \(H\) by the identity \(Q_w = Q M_w\). The realized covariation estimator itself must also be defined with respect to the appropriate inner product. In the weighted space, it takes the form
\[
\mathrm{RV}_T^w := \sum_{i=1}^{\ulT} (O X_{i\Delta}-O X_{(i-1)\Delta})^{\otimes_w 2},
\]
where \((u \otimes_w v)z = \langle z,v \rangle_w u\). As a direct consequence of the definitions, we have
\[
\| \mathrm{RV}_T^{\Delta,O}/T - Q \|_{\mathcal{L}_2(H)} = \| M^{-1/2}_w (\mathrm{RV}_T^w/T - Q_w) M^{-1/2}_w\|_{\mathcal{L}_2(H_w)}.
\]
Due to equivalence of $\|\cdot \|_H$ and $\|\cdot \|_{H_w}$, our convergence results in $H = L^2(\cD)$ below remain valid for the non-self-adjoint operator of this example. The equivalence of domains for fractional powers of weighted an unweighted operators $A$ ensures that the regularity conditions on \(Q_w\) have the same interpretation in terms of function smoothness as those on \(Q\).

\subsubsection{Regularity conditions on the noise covariance}\label{Sec: Discussion of regularity conditions of Q}
We now discuss the regularity conditions on $Q$ from Assumptions~\ref{ass:rate} and~\ref{ass:Ih}. The subsequent examples are formulated on the level of the equivalent covariance kernel $q$ rather than the operator $Q$. For this, recall that since $Q$ is trace-class, it can be written as a kernel operator
\begin{equation*}
	\cD \ni x \mapsto Q f(x)= \int_{\cD} q(x,y) f(y)dy,
\end{equation*}
for a kernel $q\in L^2(\cD \times \cD, \mathbb R)$. While our assumptions provide a general framework, their conditions may appear abstract. The examples below illustrate how the conditions on $Q$ imposed by Assumptions \ref{ass:rate} and \ref{ass:Ih} can be verified by connecting them to more familiar regularity properties of the kernel $q$. For the full details of these techniques and references to the literature, we refer to \cite{KovacsLangPetersson2023}.

Throughout this discussion, we consider the operator $A$ as defined in Example \ref{ex:elliptic}, where the operator-induced spaces $\dot{H}^r$ are equivalent to the classical fractional Sobolev spaces $H^r$ for a range of $r$ determined by the boundary conditions (cf.\ \eqref{eq:sobolev_id_theta_def}). This equivalence allows us to translate the operator-theoretic conditions into conditions on the smoothness of the kernel $q$.

\begin{example}[Hölder continuous kernels]
	To verify a Hilbert--Schmidt condition such as $\| A^{\gamma/2} Q \|_{\cL_2(H)} < \infty$ of Assumption \ref{ass:rate} directly, it is natural to consider Hölder continuity of the kernel $q$. Due to the equivalence $\dot{H}^\gamma \cong H^\gamma$, this condition is for sufficiently small $\gamma$ equivalent to $\|Q\|_{\cL_2(H,H^\gamma)} < \infty$. To see how Hölder continuity implies this, we can examine the squared norm
	\begin{align*}
		\|Q\|_{\cL_2(H,H^\gamma)}^2 &= \sum_{j=1}^\infty \|Q e_j\|_{H^\gamma}^2,
	\end{align*}
	where $(e_j)_{j=1}^\infty$ is an orthonormal basis of $H$. For $\gamma  < 1$, the Sobolev--Slobodeckij norm \eqref{eq:sobolev-slobodeckij} gives
	\begin{align*}
		\sum_{j=1}^\infty \|Q e_j\|_{H^\gamma}^2 &= \sum_{j=1}^\infty \left( \|Qe_j\|_H^2 + \int_{\cD \times \cD} \frac{|Qe_j(x) - Qe_j(y)|^2}{|x-y|^{d+2\gamma}} \dd x \dd y \right) \\
		&= \|Q\|_{\cL_2(H)}^2 + \int_{\cD \times \cD} \frac{\sum_{j=1}^\infty |\langle q(x,\cdot)-q(y,\cdot), e_j \rangle|^2}{|x-y|^{d+2\gamma}} \dd x \dd y \\
		&= \|Q\|_{\cL_2(H)}^2 + \int_{\cD \times \cD} \frac{\|q(x,\cdot)-q(y,\cdot)\|_H^2}{|x-y|^{d+2\gamma}} \dd x \dd y.
	\end{align*}
	If we assume that for a.e. $x \in \cD$, the function $q(x, \cdot)$ is $\sigma$-Hölder continuous with a uniformly bounded Hölder constant, then the numerator is bounded by $C^2|x-y|^{2\sigma}$, and the integral converges if $\sigma > \gamma$. Higher-order regularity of $q$ (i.e., differentiability and satisfying boundary conditions) yields analogous bounds for larger $\gamma$. 
\end{example}

Assumption~\ref{ass:Ih} includes conditions in the operator as well as (through Assumption~\ref{ass:rate}) the Hilbert–Schmidt norm. While a Hilbert–Schmidt property implies the corresponding operator norm bound, the reverse is not generally true. However, if $Q$ maps $H$ to a sufficiently smooth Sobolev space $H^s$, the embedding from $H^s$ into a lower-order space $H^r$ is Hilbert–Schmidt if $s-r > d/2$. This connection is key for analyzing the next class of kernels.

\begin{example}[Stationary kernels]
	\label{ex:stationary-kernels}
	For a stationary kernel $q(x,y) = q(x-y)$, its regularity properties are naturally described by the decay of its Fourier transform $\hat{q}$. The connection to the regularity of $Q$ is established via Sobolev spaces, whose norms for functions on $\R^d$ are conveniently defined in Fourier space. The space $H^s(\R^d)$ for $s \in \R$ consists of tempered distributions $u$ for which the norm
	\begin{equation*}
		\|u\|_{H^s(\R^d)}^2 = \frac{1}{(2\pi)^{d/2}} \int_{\R^d} (1+|\xi|^2)^s |\hat{u}(\xi)|^2 \dd\xi < \infty.
	\end{equation*}
	Since $\cD$ has a Lipschitz boundary, the Sobolev space $H^s$ can be defined via restrictions of functions in $H^s(\R^d)$, with its norm equivalent to the Sobolev--Slobodeckij norm \eqref{eq:sobolev-slobodeckij} for $s \ge 0$.
	
	A decay rate of $\hat{q}(\xi) \lesssim (1+|\xi|^2)^{-\sigma}$ for some $\sigma>d/2$ implies the operator property $Q \in \cL(H,H^{2\sigma})$, which is proven by extending functions from $\cD$ to $\R^d$. For conditions involving the square root, $Q^{1/2}$, the key fact is that $Q^{1/2}(H)$ is isometrically isomorphic to the reproducing kernel Hilbert space $H_q$. The same decay condition on $\hat{q}$ ensures the continuous embedding $H_q \hookrightarrow H^\sigma$ (see, e.g., \cite[Corollary 10.48]{Wendland04}). This gives the mapping property $Q^{1/2} \in \cL(H, H^\sigma)$. This method, combined with Sobolev embeddings, can be used to verify the conditions on $Q$ and $Q^{1/2}$ in our assumptions.
	
	The Matérn class of kernels is a canonical example. A Matérn kernel with paramaters $\sigma,\nu, \rho >0$ is specified by the formula
	\begin{equation}\label{eq: Matern kernel}
		q(x,y)=\sigma^2 \frac {2^{1-\nu}}{\Gamma(\nu)}\left(\sqrt{2\nu}\frac  {|x-y|}{\rho}\right)^{\nu}K_{\nu}\left(\sqrt{2\nu}\frac {|x-y|}{\rho}\right),\quad x,y \in \cD
	\end{equation}
	where $\Gamma$ is the gamma function and $\mathcal K_{\nu}$ is a modified Bessel function of the second kind.
	Its spectral density satisfies $\hat{q}(\xi) \eqsim (1+|\xi|^2)^{-(\nu+d/2)}$for some $\nu > 0$, which corresponds to a Sobolev regularity of $\sigma = \nu+d/2$. Under Neumann boundary conditions ($\theta=3/2$), one can show that Assumption~\ref{ass:rate}\ref{ass:rate:q} holds for all $\gamma < \min(3/2, 2\nu+d/2)$, while Assumption~\ref{ass:Ih}\ref{ass:Ih:q:1} holds for $\beta < 2\min(\nu+1/2, 2-d/2)$ and~\ref{ass:Ih}\ref{ass:Ih:q:2} for any $\zeta, \eta < 3/2$ such that $\zeta+\eta < 2\nu+d$.
\end{example}

\begin{example}[Kernels commuting with $A$]
	\label{ex:kernels-commuting-with-A}
	A commonly considered scenario is when $A$ and $Q$ share the same orthonormal basis of eigenfunctions $(e_k)_{k=1}^\infty$. These are then related to the kernel via the expansion
	\begin{equation*}
		q(x,y) = \sum_{k=1}^{\infty} \mu_k e_k(x) e_k(y),
	\end{equation*}
	where the series converges uniformly. The conditions in Assumption~\ref{ass:rate} and Assumption~\ref{ass:Ih} can then be expressed directly in terms of the eigenvalue sequence $(\mu_k)_{k=1}^\infty$ of $Q$ and $(\lambda_k)_{k=1}^\infty$ of $A$.
	
	The Hilbert--Schmidt condition in  Assumption~\ref{ass:rate} becomes a summability condition
	\begin{equation*}
		\| A^{\gamma/2} Q \|_{\cL_2(H)}^2 = \sum_{k=1}^\infty \| A^{\gamma/2} Q e_k \|^2 = \sum_{k=1}^\infty \| \mu_k \lambda_k^{\gamma/2} e_k \|^2 = \sum_{k=1}^\infty \lambda_k^\gamma \mu_k^2 < \infty,
	\end{equation*}
	while the operator norm conditions in  Assumption~\ref{ass:Ih} become supremum conditions
	\begin{align*}
		\| A^{(\beta+d+\epsilon-1)/4} Q^{1/2} \|_{\cL(H)} & = \sup_k \lambda_k^{(\beta+d+\epsilon-1)/4} \mu_k^{1/2} < \infty, \\
		\| A^{\zeta/2} Q A^{\eta/2} \|_{\cL(H)} &= \sup_k \lambda_k^{(\zeta+\eta)/2} \mu_k < \infty.
	\end{align*}
	A canonical example is the fractional inverse Dirichlet Laplacian covariance, where $Q = -\Delta_{\cD}^{-\sigma}$ for some $\sigma>d/2$. This implies $\mu_k = \lambda_k^{-\sigma}$ and the summability condition becomes $\sum_k \lambda_k^\gamma (\lambda_k^{-\sigma})^2 = \sum_k \lambda_k^{\gamma - 2\sigma} < \infty$. By Weyl's law, $\lambda_k \simeq k^{2/d}$, so this sum converges if $\gamma < 2\sigma-d/2$, while the supremum conditions are fulfilled for $\beta \le 2\sigma + 1 - d - \epsilon$ and all $\eta, \zeta$ such that $\eta + \zeta \le 2\sigma$
\end{example}

We close this subsection with a remark on the relation between Assumption \ref{ass:rate}\ref{ass:rate:q} and \ref{ass:Ih}\ref{ass:Ih:q:2}.

\begin{remark}
	With the condition $\| A^{\zeta/2} Q A^{\eta/2} \|_{\cL(H)} = \| Q \|_{\cL(\dot{H}^{-\eta},\dot{H}^\zeta)} < \infty$ it is implicitly assumed that $Q$ extends to an operator on $\dot{H}^{-\eta}$. Since $Q$ is symmetric, the condition is equivalent to $Q$ extending to an operator on $\dot{H}^{-\zeta}$ satisfying $\| Q \|_{\cL(\dot{H}^{-\zeta},\dot{H}^{\eta})} = \| A^{\frac \eta 2} Q A^{\frac \zeta 2} \|_{\cL(H)} < \infty$. A corresponding one-sided condition is stronger in the sense that if $\|A^{\gamma/2} Q\|_{\cL(H)} < \infty$, then by operator interpolation, $\|A^{\zeta/2} Q A^{\eta/2}\|_{\cL(H)} < \infty$ for all $\zeta, \eta \ge 0$ with $\zeta+\eta = \gamma$. The converse, however, is not generally true.
\end{remark}

We now turn to the asymptotic theory of the  realized covariation \eqref{eq:realized-variation}. We start with rates of convergence in Section \ref{Sec: Rates of Convergence} and then discuss central limit theorems in Section \ref{Sec: Asymptotic Normality}.

\subsection{Convergence rates}\label{Sec: Rates of Convergence}

The subsequent theorem contains the rates of convergence for the realized covariation as an estimator of the covariance $Q$.

In the theorem and its proof the bias and the variance of the estimator $RV_T^{\Delta,O}/T$ for the infinite-dimensional covariance $Q$ are defined by
\begin{align*}
	Bias\left(\frac{RV_T^{\Delta,O}}T\right) &= \mathbb E\left[\frac{RV_T^{\Delta,O}}T\right]-Q, \quad \text{and}\\ Var\left(\frac{RV_T^{\Delta,O}}T\right) &=  \mathbb E\left[\left\|\frac{RV_T^{\Delta,O}}T-\mathbb E\left[\frac{RV_T^{\Delta,O}}T\right]\right\|_{\cL_2(H)}^2\right].
\end{align*}
\begin{theorem}
	\label{thm:RV-convergence}
	The following quantification of the bias and variance holds true.
	\begin{enumerate}[label=(\Alph*)]
		\item If $O = I$ and Assumption~\ref{ass:rate} is fulfilled, then
		\begin{align*}
			\left\| \mathrm{Bias}\left(\frac{RV^{\Delta,I}_T}{T}\right)\right\|_{\cL_2(H)}^2 &= \cO(\Delta^{\min(\gamma,2)} + T^{-2} \Delta^{\min(2\iota,2)}), \quad \text{and}\\
			\mathrm{Var}\left(\frac{RV^{\Delta,I}_T}{T}\right) &= \cO(T^{-1}\Delta + T^{-2} \Delta^{\min(2\iota,2)}).
		\end{align*}
		
		\item If $O = P_h$ and Assumption~\ref{ass:Ph} is fulfilled, then
		\begin{align*}
			\left\| \mathrm{Bias}\left(\frac{RV^{\Delta,P_h}_T}{T}\right)\right\|_{\cL_2(H)}^2 &= \cO(h^{\min(2\gamma,2)}+\Delta^{\min(\gamma,2)} + T^{-2} \Delta^{\min(2\iota,2)}), \quad \text{and}\\
			\mathrm{Var}\left(\frac{RV^{\Delta,P_h}_T}{T}\right) &= \cO(T^{-1}\Delta + T^{-2} \Delta^{\min(2\iota,2)}).
		\end{align*}
		
		\item Let $O = I_h$ and suppose that Assumption~\ref{ass:Ih} is fulfilled and that Assumption~\ref{ass:rate}\ref{ass:rate:initial} holds with $\iota > d/2$.  Let $\psi \in (d/2,1+\zeta)$ and write $\tilde \epsilon = \max(\psi-\zeta,0)+\max(d/2-\eta+\epsilon,0)$. Then, for sufficiently small $\epsilon> 0$,
		\begin{align*}
			\left\| \mathrm{Bias}\left(\frac{RV^{\Delta,I_h}_T}{T}\right)\right\|_{\cL_2(H)}^2 &= \cO(T^{-2}(h^{\min(4\iota,8)} + \Delta^{\min(2\iota,2)}) \\
			&\hspace{4em}+ \Delta^{-\tilde \epsilon} h^{\min(2\psi,4)} + \Delta^{\min(\gamma,2)}), \quad \text{and}\\
			\mathrm{Var}\left(\frac{RV^{\Delta,I_h}_T}{T}\right) &= \cO(T^{-1} \Delta^{\min(1,\beta)} + T^{-2}h^{\min(4\iota,8)} ).
		\end{align*}
	\end{enumerate}
\end{theorem}

\begin{remark}
    Note that we included large $T$ behaviour into the rates in Theorem \ref{thm:RV-convergence}. While this is not of direct use for pure infill-asymptotics, it is interesting to observe that the variance vanishes as $T\to\infty$, while the Bias (apart from the parts that concern the initial condition) does not disappear with $T\to \infty$. This leaves potential to estimate the bias when $T\to \infty$ can be assumed, a route that might be investigated in future research.
\end{remark}

Using Theorem \ref{thm:RV-convergence}, we obtain a bound on the root mean square error (RMSE) of the realized covariation in the different sampling settings.

\begin{theorem}
	\label{thm:RV-convergence-corollary}
	Suppose that $T$ is fixed and that Assumption~\ref{ass:rate} holds with $\iota \ge \gamma/2$.
	\begin{enumerate}[label=(\Alph*)]
		\item If $O=I$, the RMSE fulfills
		\[
		\Big\| \frac{RV^{\Delta,I}_T}{T} - Q \Big\|_{L^2(\Omega,\cL_2(H))} = \cO(\Delta^{\min(\gamma,1)/2}).
		\]
		\item If $O=P_h$, Assumption~\ref{ass:Ph} is satisfied, and we couple the spatial and temporal resolutions such that $h = \cO(\Delta^{1/2})$, the RMSE fulfills
		\[
		\Big\| \frac{RV^{\Delta,P_h}_T}{T} - Q \Big\|_{L^2(\Omega,\cL_2(H))} = \cO(\Delta^{\min(\gamma,1)/2}).
		\]
		\item Suppose that $O = I_h$, that Assumption~\ref{ass:Ih} is satisfied and and that $\iota > d/2$ in Assumption~\ref{ass:rate}\ref{ass:rate:initial}. Let $\psi < \iota$ and $\tilde\epsilon$ be as in Theorem~\ref{thm:RV-convergence}.
		If we couple the spatial and temporal resolutions by $h = \cO(\Delta^{(\min(\gamma,1,\beta)+\tilde\epsilon)/\min(2\psi,4)})$, which is (with the choices $\eta=\zeta= (d+\epsilon+\beta-1)/2, \psi:= (d+1+\beta)/2$ and assuming $\beta = \gamma \le 1$) in particular the case if 
		$$h= \begin{cases}
			\mathcal O\left(\Delta^{\frac{3}{2d+2}}\right) & d\leq 2\\
			\mathcal O\left(\Delta^{\frac {3}8}\right) & d= 3,
		\end{cases}
		$$
		then for $O=I_h$, the RMSE fulfills
		\[
		\Big\| \frac{RV^{\Delta,I_h}_T}{T} - Q \Big\|_{L^2(\Omega,\cL_2(H))} = \cO(\Delta^{\min(\gamma,\beta,1)/2}).
		\]
	\end{enumerate}
\end{theorem}

\begin{remark}
	The mild solution $X$ is not necessarily a continuous It\^o semimartingale in $H$. When the regularity condition $\gamma \ge 1$ from Assumption~\ref{ass:rate} is not met, the convergence rates in Theorem~\ref{thm:RV-convergence-corollary} are slower than the canonical $\mathcal{O}(\Delta^{1/2})$ rate (c.f. \cite[Theorem  5.4.2]{JacodProtter2012}). Proposition~\ref{prop:lower-bound} reveals that these slower rates are, up to logarithmic factors, optimal. A standard sufficient condition for $X$ to be a strong solution, and thus a semimartingale, is that $\|A^{1/2}Q^{1/2}\|_{\cL_2(H)} < \infty$ (see \cite[Proposition 6.18]{DPZ2014}). Our framework, however, relies on the potentially much weaker conditions of Assumption~\ref{ass: minimal Assumption on A} and Assumption~\ref{ass:rate}.
\end{remark}

\begin{remark}
	The rates in Theorem~\ref{thm:RV-convergence-corollary} can be interpreted more concretely in a practical scenario of $X_0=0$, the elliptic operator $A$ from Example~\ref{ex:elliptic}, and the widely-used Matérn covariance kernel from Example~\ref{ex:stationary-kernels}. In this setting, the convergence rate is determined by an interplay between operator's boundary conditions, the kernel's smoothness, and the spatial sampling scheme.
	
	For continuous sampling or local averages ($O=I$ or $O=P_h$), the achievable regularity $\gamma$ is dictated by the compatibility between the kernel and the operator. With Dirichlet boundary conditions, the equivalence in \eqref{eq:sobolev_id_theta_def} limits us to $\gamma < 1/2$, yielding a rate slightly worse than $\Delta^{1/4}$. In contrast, Neumann boundary conditions allow for $\gamma=1$ and an optimal rate of $\cO(\Delta^{1/2})$ provided the Matérn smoothness parameter $\nu$ fulfills $\nu \ge 1/2 - d/4$.
	
	The pointwise sampling setting ($O=I_h$) is more intricate, due to the interaction term $\Delta^{-\tilde\epsilon} h^{\min(2\psi,4)}$ in the bias and a variance term dependent on $\beta$. Pointwise sampling can still achieve the same optimal temporal rates as the other schemes and require fewer spatial observations. For a Matérn kernel under Neumann conditions, we can set $\tilde\epsilon=0$ by choosing, for example, $\eta = d/2+\epsilon$ and $\psi = \zeta = \min(2\nu+d/2-\epsilon, 3/2)$. With this choice, and by setting $\beta=1$ (which is permissible), the temporal convergence rate becomes $\cO(\Delta^{\min(\gamma,1)/2})$, matching the other schemes. For a Matérn kernel with smoothness $\nu > 1-d/4$, we have $\gamma \ge 1$, so the coupling $h = \cO(\Delta^{\min(\gamma,1)/\min(2\psi,4)})$ becomes $h =  \cO(\Delta^{1/3})$. The general-purpose coupling can (depending on $d$) thus be quite conservative.
	
	These restrictions on parameters and the overall complexity arise because the Matérn kernel does not inherently satisfy the boundary conditions imposed by $A$. In contrast, for a kernel that commutes with $A$, such as the fractional inverse Laplacian $Q = A^{-\sigma}$ from Example~\ref{ex:kernels-commuting-with-A}, the analysis simplifies considerably. If $\sigma > \max(d/2,1/2 + d/4)$, the optimal temporal convergence rate of $\cO(\Delta^{1/2})$ is achieved for all three sampling schemes. The required coupling from Theorem~\ref{thm:RV-convergence-corollary} then simplifies, in the worst case scenario as $\sigma \searrow d/2$, to $h = \cO(\Delta^{1/d})$.
\end{remark}

The next proposition shows that, up to a logarithmic factor, the temporal rate of convergence established in Theorem~\ref{thm:RV-convergence} cannot be improved.
\begin{proposition}
	\label{prop:lower-bound}
	For any $\gamma \in (1/2, 1)$ and fixed $T>0$, there exist operators $A$ and $Q$ satisfying Assumption~\ref{ass:rate} such that for a constant $C>0$ and all sufficiently small $\Delta > 0$, 
	$$
	\Big\| \frac{RV^{\Delta,I}_T}{T} - Q \Big\|_{L^2(\Omega,\cL_2(H))} \ge \left\| \mathbb{E}\left[\frac{RV^{\Delta,I}_T}{T}\right] - Q \right\|_{\mathcal{L}_2(H)} \ge C \frac{\Delta^{\gamma/2}}{|\log \Delta|}. $$
\end{proposition}

In the next subsection, we characterize the asymptotic distribution of the realized covariation.
\subsection{Asymptotic normality}\label{Sec: Asymptotic Normality}
At this point, we turn to asymptotic normality of the operator-valued realized variation.

\begin{theorem}
	\label{thm: CLT}
	Suppose that $T$ is fixed and that Assumption~\ref{ass:rate} is fulfilled with $2\iota \ge \gamma \ge 1$. Define $\Gamma: \cL_2(H)\to \cL_2(H)$ by $$\Gamma B:= Q(B+B^*)Q.$$
	Then, the following claims on convergence in distribution with respect to the Hilbert-Schmidt norm  hold true.
	\begin{enumerate}[label=(\Alph*)]
		\item If $O=I$, then
		\begin{align*}
			\Delta^{-\frac 12}\left(\frac{\mathrm{RV}^{\Delta,I}_T}{T} - Q \right) \overset{d}{\longrightarrow} N(0,\Gamma)\qquad \text{ as } \Delta\to 0.
		\end{align*}
		\item If $O=P_h$, suppose also that Assumption~\ref{ass:Ph} is fulfilled, and couple the spatial and temporal resolutions by $h=o(\Delta^{1/2})$. Then
		\begin{align*}
			\Delta^{-\frac 12}\left(\frac{\mathrm{RV}^{\Delta,P_h}_{T}}{T} - Q \right) \overset{d}{\longrightarrow} N(0,\Gamma)\qquad \text{ as } \Delta\to 0.
		\end{align*}
		\item If $O=I_h$, suppose also that Assumption~\ref{ass:Ih} is fulfilled with $\beta \ge 1$ and that $\iota > d/2$ in Assumption~\ref{ass:rate}\ref{ass:rate:initial}. Let $\psi$ and $\tilde\epsilon$ be as in Theorem~\ref{thm:RV-convergence}. If we couple the spatial and temporal resolutions by $h = \cO(\Delta^{(1+\tilde\epsilon)/\min(2\psi,4)})$ which is satisfied, for instance, if 
		$$h= \begin{cases}
			\mathcal O\left(\Delta^{\frac{3}{2d+2}}\right) & d\leq 2\\
			\mathcal O\left(\Delta^{\frac {3}8}\right) & d= 3,
		\end{cases}
		$$
		then
		\begin{align*}
			\Delta^{-\frac 12}\left(\frac{\mathrm{RV}^{\Delta,I_h}_{T}}{T} - Q \right) \overset{d}{\longrightarrow} N(0,\Gamma)\qquad \text{ as } \Delta\to 0.
		\end{align*}
	\end{enumerate}
\end{theorem}

The central limit theorem can be used to derive goodness-of-fit tests for covariance models. This is the subject of the next section.

\section{Goodness-of-fit tests for the noise covariance\label{Sec: Universal Goodness of fit test}}

We now develop goodness-of-fit tests for the covariance $Q$ of the driving Wiener process. Precisely, we develop asymptotic tests for the hypothesis
\begin{align}\label{eq: universal test fixed}
	H_0: Q =Q_0\quad \text{ vs. }\quad H_1: Q \neq Q_0,
\end{align}
as well as a conservative test for the joint hypothesis
\begin{align}\label{eq: universal test estimated}
	H_0: Q =Q_{\theta}  \text{ for some }\theta \in \Theta \quad \text{ vs. }\quad H_1: Q \neq Q_{\theta}  \text{ for all }\theta\in \Theta .
\end{align}

\subsection{Test for fixed noise covariances}
For the test \eqref{eq: universal test fixed},  the CLT, Theorem \ref{thm: CLT} in the previous section, immediately yields with $t_{\Delta}:= \ulT \Delta$ a test statistic
\begin{equation}\label{eq: test statistic in the regular case}
	T_{\Delta}:=\Delta^{-1}\left\|\frac{RV_O^{t_{\Delta}}}{t_{\Delta}}-Q_0\right\|_{\cL_2(H)}^2  ,  
\end{equation}
which, under $H_0$ has asymptotically the same distribution as the random variable $\sum_{i=1}^{\infty} \Lambda_i^2 \zeta_i^2$, where $(\Lambda_i)_{i\in \mathbb N}$ are in decreasing order the eigenvalues of $\Gamma$ as defined in Theorem \ref{thm: CLT} and $(\zeta_i)_{i\in \mathbb N}$ is a sequence of independent standard normal random variables.  Indeed, it is simple to derive the eigenvalues of $\Gamma$ from the eigenvalues of $Q$
\begin{lemma}
	Let $Q= \sum_{i=1}^{\infty} \mu_i f_i$, then 
	$$\Gamma =Q(\cdot+\cdot^*)Q= \sum_{i=1}^{\infty} \sum_{j\leq i} 2\mu_i\mu_j F_{i,j}^{\otimes 2},$$
	with $F_{i,j}=(f_i\otimes f_j + f_j\otimes f_i)/\sqrt{ 2}$ when $j<i$ and $F_{i,i}= e_i^{\otimes 2}$, which form an orthonormal system in $\cL_{HS}(H)$.
\end{lemma}
In particular, we have that
the test statistics $T_{\Delta}$ has asymptotically the same distribution as the random variable 
\begin{equation}\label{eq: Quadratic form}
	V:=\sum_{i=1}^{\infty}\sum_{j\leq i} 2\mu_i\mu_j \zeta_{i,j}^2,
\end{equation}
for an independent family of standard normal random variables $(\zeta_{i,j})_{i,j=1,...,\infty, j<i}$. If we truncate the sum, this is a generalized $\chi^2$ random variable, of which we can calculate exceedance probabilities explicitly. Due to the fast eigenvalue decay,  these probabilities are close to the ones of $V$, when enough terms in the sum are added.
This guarantees feasibility of the test outlined in the subsequent theorem. 
\begin{theorem}[Goodness-of-fit test for noise covariances]\label{thm: Universal Goodness of fit}
	Let $(\mu_k)_{k\in \mathbb N}$ be the eigenvalues of $Q_0$, w.l.o.g, in decreasing order and let $V$ be defined as in \eqref{eq: Quadratic form}.    
	Let $q_{1-\alpha}$ denote the $1-\alpha$ quantile of $V$.
	Then, the test-function
	$$\varphi_{\Delta}=\indicator_{T_{\Delta}\geq q_{1-\alpha}}$$
	defines a test that is asymptotically of level $\alpha$.
\end{theorem}

The pseudocode for the implementation of the test described in Theorem \ref{thm: Universal Goodness of fit} can be found in Algorithm~\ref{Alg: Pseudocode for fixed Q GoF} for the case of spatially discrete observations. To implement this, we must discretize the Hilbert-Schmidt norm. In light of \eqref{eq:pointwise-RV-kernel-form} and \eqref{eq:spatav-RV-kernel-form}, a computationally convenient choice is to approximate the null hypothesis kernel $q_0$ by 
\begin{equation}
	\label{eq:discrete-null-kernel}
	q_0^h(x,y) = \sum_{j,k=1}^{N_h} \mathbf{C}^{0}_{jk} \psi_j(x)\psi_k(y)
\end{equation}
and compute the Hilbert-Schmidt norm by a simple matrix calculation. Here $(\psi_j)_{j=1}^{N_h}$ are given by $\psi_j = \phi_h^j$  for $O=I_h$ and by $\psi_j = \indicator_{T_h^j}$ for $O=P_h$. In the former case, we simply set $\mathbf{C}^{0}_{jk} = q_0(x_j, x_k)$ while in the latter, we take representative points (e.g., centroids) $c_j \in T_h^j$ and set $\mathbf{C}^{0}_{jk} = q_0(c_j, c_k)$. Then, the squared Hilbert-Schmidt norm of the difference is
\begin{equation*}
	%\label{eq:hs-norm-matrix-form}
	\tr\left( \mathbf{G} (\mathbf{C}^{RV} - \mathbf{C}^{0})^T \mathbf{G} (\mathbf{C}^{RV} - \mathbf{C}^{0}) \right),
\end{equation*}
where $\mathbf{G}$ is the Gram matrix with entries $G_{jk} = \langle \psi_j, \psi_k \rangle$ while $\mathbf{C}^{RV}_{jk}$ is, in the notation of \eqref{eq:pointwise-RV-kernel-form} and \eqref{eq:spatav-RV-kernel-form}, $\frac{1}{T} \sum_{i=1}^{\ulT} \Delta_i X(x_j) \Delta_i X(x_k)$ or $\frac{1}{T} \sum_{i=1}^{\ulT} \overline{\Delta_i X}_{T_h^j} \overline{\Delta_i X}_{T_h^k}$. 

\begin{algorithm}[H]
	\caption{GoF Test for Fixed Covariance \(Q_0\)}
	\label{Alg: Pseudocode for fixed Q GoF}
	\begin{algorithmic}[1]
		\REQUIRE Observed increments on a spatial grid $\cT_h$, temporal step size $\Delta$, hypothesized kernel $q_0$, significance level $\alpha$.
		\STATE Compute the coefficient matrix $\mathbf{C}^{RV}$ from data and $\mathbf{C}^{0}$ from $q_0$.
		\STATE Compute the Gram matrix $G$ for the basis functions of the observation scheme.
		\STATE Compute the test statistic:
		$
		T_{\Delta} = \frac{1}{\Delta} \, \tr\left( \mathbf{G} (\mathbf{C}^{RV} - \mathbf{C}^{0})^T \mathbf{G} (\mathbf{C}^{RV} - \mathbf{C}^{0}) \right).
		$
		\STATE Approximate the eigenvalues $\{\mu_k\}_{k=1}^{N_h}$ of $Q_0$ by the eigenvalues of $\mathbf{G} \mathbf{C}^{0}$.
		\STATE Compute 
		$p = \mathbb{P}\left(\sum_{j=1}^{N_h} \sum_{k \le j} 2 \mu_j \mu_k \zeta_{jk}^2 > T_{\Delta}\right)$ 
		via Davis or Imhof’s method. 
		\STATE Set $\varphi_{\Delta} = 1$ if $p < \alpha$, else $0$.
		\RETURN $\varphi_{\Delta}$.
	\end{algorithmic}
\end{algorithm}

\begin{remark}
	The accuracy of the test implementation depends on how well the discrete approximations represent their exact counterparts. For greater power, one could approximate $q_0$ on a finer spatial grid than that used for the observations and/or use numerical quadrature. However, for simplicity, we have only presented the case where both approximations use the same basis, since under mild regularity Assumptions on $Q_0$, this does not alter the limiting behaviour of the test provided
	$h$  decreases at an appropriate rate.
	
	For instance,  in the setting $O=I_h$ and with $Q_0$ of either Matérn class or equal to $-\Delta_{\mathcal{D}}^{-\nu-d/2}$ with smoothness parameter $\nu>1/2-d/4$, this choice will not decrease the accuracy asymptotically as long as $h = o(\Delta^{1/2})$. This is because the kernel approximation \eqref{eq:discrete-null-kernel} corresponds to replacing $Q_0$ with $Q_0^h := I_h Q_0^{1/2} (I_h Q_0^{1/2})^*$ and it is not hard to see (cf.\ Examples~\ref{ex:stationary-kernels} and \ref{ex:kernels-commuting-with-A}) that 
	\[
	\| Q_0^h - Q_0\|_{\cL_2(H} \le \| I- I_h \|_{\cL(H^{2 \nu + d/2-\epsilon},H)} \| I + I_h \|_{\cL_2(H^{d/2+\epsilon},H)} \| Q_0 \|_{\cL(H^{d/2+\epsilon},H^{2 \nu + d/2-\epsilon})}
	\]
	for arbitrary $\epsilon>0$, and this quantity tends to $0$ faster than $\Delta^{1/2}$. A similar argument, combined with Weyl's inequality, controls the error in the approximation of the eigenvalues of $Q_0$.
\end{remark}
\subsection{Test for Parametric Families of Noise Covariances}

We now turn to the problem of testing whether the noise covariance belongs to a parametric class of covariance kernels, such as the Mat{\'e}rn  class or fractional powers of the inverse Laplacian. 
Specifically, we consider testing the hypothesis in \eqref{eq: universal test estimated}. To this end, we introduce the test statistic
\begin{equation}\label{eq: test statistic in the estimated parameter case}
	T_{\Delta}^*:=\min_{\theta\in \Theta}\Delta^{-1}\left\|\frac{RV_O^{t_{\Delta}}}{t_{\Delta}}-Q_{\theta}\right\|_{\cL_2(H)}^2  . 
\end{equation}
This leads to a conservative test as described in the next theorem.
\begin{theorem}[Goodness-of-fit test for parametric noise covariances models]\label{thm: Universal Goodness of fit- estimated parameters}
	We need 
	\begin{itemize}
		\item (Identifiability) The map $\theta \mapsto Q_{\theta}$ is injective;
		\item $\theta \mapsto Q_{\theta}$ is continuous with respect ot the trace norm;
		\item There almost surely exists a $\theta_{\Delta}^*$ such that the minimum in formula \eqref{eq: test statistic in the estimated parameter case} is attained. 
	\end{itemize}
	Let  $(\mu_k^*)_{k\in \mathbb N}$ be the eigenvalues of $Q_{\theta^*_{\Delta}}$, w.l.o.g, in decreasing order and let $V^*=\sum_{i=1}^{\infty}\sum_{j \leq i}2\mu_i^* \mu_j^*$. Then  $q_{1-\alpha}^*$ denotes the $1-\alpha$ quantile of $V^*$.
	Then, the test-function
	$$\varphi_{\Delta}=\indicator_{T_{\Delta}^*\geq q_{1-\alpha}^*}$$
	defines a test that is asymptotically at most of level $\alpha$.
\end{theorem}

The test described in Theorem~\ref{thm: Universal Goodness of fit- estimated parameters} is implemented by discretizing the problem and employing numerical optimization to compute the minimizer in \eqref{eq: test statistic in the estimated parameter case}. Once the minimizer is obtained, the procedure follows analogously to Algorithm~\ref{Alg: Pseudocode for fixed Q GoF}. The corresponding pseudocode is presented in Algorithm~\ref{Alg: Pseudocode for parametric Q GoF- estimated paramaters}.

\begin{algorithm}[H]\label{Alg: Pseudocode for parametric Q GoF- estimated paramaters}
	\caption{GoF Test for Parametric Covariance Family \(\{Q_\theta\}_{\theta\in\Theta}\)}
	\begin{algorithmic}[1]
		\REQUIRE Observed increments on a spatial grid $\cT_h$, temporal step size $\Delta$, parametric family of kernels $\theta\mapsto q_\theta$, significance level $\alpha$.
		\STATE Compute the coefficient matrix $\mathbf{C}^{RV}$ from data and the Gram matrix $G$.
		\STATE Compute the test statistic by solving the minimization problem:
		\[ T_{\Delta}^* = \min_{\theta \in \Theta} \frac{1}{\Delta} \tr\left( \mathbf{G} (\mathbf{C}^{RV} - \mathbf{C}^{\theta})^T \mathbf{G} (\mathbf{C}^{RV} - \mathbf{C}^{\theta}) \right), \]
		where $\mathbf{C}^{\theta}$ is the coefficient matrix for the kernel $q_\theta$.
		\STATE Let $\theta_{\Delta}^*$ be the parameter value that attains the minimum. Proceed with Steps~3--4 of Algorithm~\ref{Alg: Pseudocode for fixed Q GoF} using $\mathbf{C}^0 := \mathbf{C}^{\theta_{\Delta}^*}$ and $T_{\Delta} := T_{\Delta}^*$.
	\end{algorithmic}
\end{algorithm}

The assumptions ensuring the validity of the test in Theorem~\ref{thm: Universal Goodness of fit- estimated parameters} must be verified separately for each parametric family of covariance kernels. In two important cases--the class of inverse fractional Laplacians and the Mat{\'e}rn class--these assumptions can be checked with relatively little effort. This is illustrated in the following two examples.

\begin{example}[Verification of assumptions for the Matérn covariance kernel]  
	Recall the definition of the Mat{\'e}rn covariance kernel on $L^2(\mathcal{D} \times \mathcal{D})$ defined in \eqref{eq: Matern kernel}. When we restrict the parameter set to a compact subset  
	$$ (\sigma^2, \nu, \rho) = \theta \in 
	\Theta = [\sigma_{\min}^2, \sigma_{\max}^2] \times [\nu_{\min}, \nu_{\max}] \times [\rho_{\min}, \rho_{\max}],
	$$
	with strictly positive lower bounds, all assumptions of Theorem \ref{thm: Universal Goodness of fit- estimated parameters} can be verified.
	
	Injectivity of the mapping follows directly from the kernel’s spectral density characterization (cf. \cite[p.\ 31]{Stein1999}).
	
	Furthermore, the Mat{\'e}rn kernel depends continuously on $\theta$ uniformly over $\mathcal{D} \times \mathcal{D}$ since $\Theta$ is compact and the kernel is smooth with respect to its parameters. Because $\mathcal{D}$ is bounded, continuity with respect to the $L^2(\mathcal{D} \times \mathcal{D})$-norm follows immediately, and hence the map $\theta \mapsto Q_{\theta}$ is continuous in the Hilbert--Schmidt norm.
	Moreover, since $tr(Q_{\theta})= \sigma^2\leq \sigma_{max}^2$, we find by dominated convergence that convergence also holds in trace norm.
	
	Finally, the compactness of $\Theta$ and continuity of $\theta \mapsto Q_\theta$ in the Hilbert--Schmidt norm guarantee the existence of a minimizer in formula \eqref{eq: test statistic in the estimated parameter case} by Weierstra{\ss}' extreme value theorem.
	\end{example}

\begin{example}[Verification of assumptions for the inverse Laplacian operator \(\Delta^{-\nu}\)]  
	Let $-\Delta = \sum_{i=1}^{\infty} \lambda_i e_i^{\otimes 2}$ be the negative Laplacian with Dirichlet boundary conditions, such that $\lambda_j \equiv j^2$.
	Then for $\nu>\frac 12$, the inverse fractional Laplacian operator $\Delta^{-\nu}$ on $L^2(\mathcal{D})$ is a symmetric, positive trace-class operator. When restricting the parameter set to a compact subset  
	$$ \nu=\theta \in
	\Theta \subseteq [\nu_{\min}, \nu_{\max}],
	$$
	with $\nu_{\min} > \frac{d}{2}$, all assumptions of Theorem \ref{thm: Universal Goodness of fit- estimated parameters} can be verified.
	
	First, injectivity of the mapping follows directly from the spectral characterization of $\Delta^{-\nu}$, as distinct values of $\nu$ produce distinct eigenvalue sequences.
	
	Next, continuity in trace norm follows since for $\nu, \rho > \frac{d}{2}$ we have via the mean value theorem the estimate
	$$
	\|\Delta_{\mathcal{D}}^{-\nu} - \Delta_{\mathcal{D}}^{-\rho}\|_{\mathcal{L}_1(H)} \lesssim \sum_{j=1}^\infty |j^{-2\nu} - j^{-2\rho}| \lesssim |\nu - \rho| \sum_{j=1}^\infty (\log j) j^{-2 \nu_{\min}}.
	$$
	Since $\nu_{\min} > \frac{d}{2}$, the sum on the right converges, yielding continuity of the map $\theta = \nu \mapsto \Delta_{\mathcal{D}}^{-\nu} = Q_\theta$ in the trace-norm topology.
	
	Together with the compactness of 
	$\Theta$, this guarantees the existence of a minimizer in formula \eqref{eq: test statistic in the estimated parameter case} by the Weierstra{\ss}' extreme value theorem.	
\end{example}

The finite-sample performance of both tests is analyzed in a simulation study for the Mat{\'e}rn and inverse fractional Laplacian covariance classes   in the next section.

\section{Simulation Study}\label{Sec: Simulation Study}

In this section, we illustrate our results with numerical simulations. For computational cost reasons, we only consider the case that $d = 1$ with $\cD = (0,1)$. We consider a fixed observation window $[0,T]$ with $T = 1$ and set the initial value $X_0 = 0$. We consider two different choices for $A$ and $q$ along with two different methods to simulate $X$.

\subsection{Root mean square errors}
First, we illustrate the results of Theorem~\ref{thm:RV-convergence-corollary}, starting with the case that $q$ is a Mat\'ern covariance kernel (see Example~\ref{ex:stationary-kernels}). We set $A = 1-\tfrac{1}{20}\Delta_\cD$, with  zero Neumann boundary conditions. In this case, $A$ and the operator $Q$ with kernel $q$ do not commute.
\begin{figure}[h]
	\centering
	\begin{minipage}{0.48\textwidth}
		\centering
		\includegraphics[width=\textwidth]{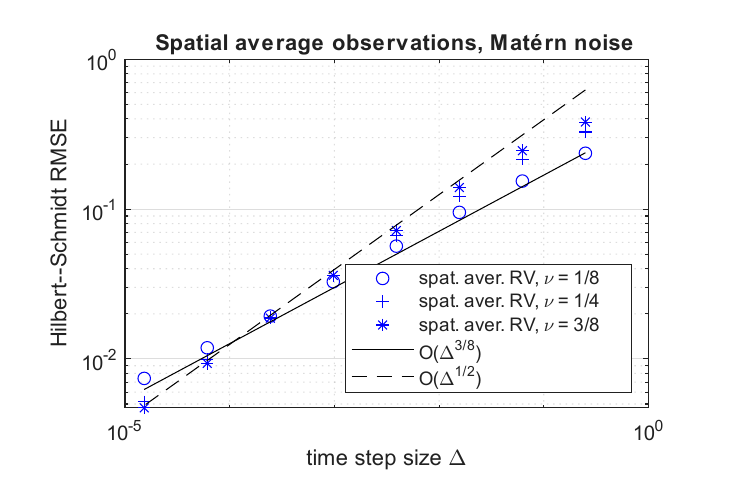}
	\end{minipage}\hfill
	\begin{minipage}{0.48\textwidth}
		\centering
		\includegraphics[width=\textwidth]{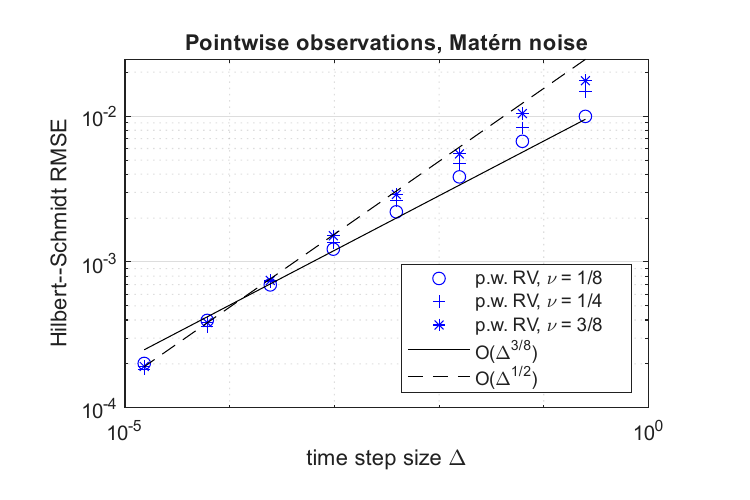}
	\end{minipage}
	\caption{Approximation of the error of Theorem~\ref{thm:RV-convergence-corollary}, with $h = \sqrt{\Delta}$ and $q$ a Mat\'ern covariance kernel.}
	\label{fig:matern}
\end{figure}
We first consider sampling of local averages, i.e., $O = P_h$. As discussed in Example~\ref{ex:stationary-kernels}, for the Mat\'ern kernel and our choice of operator $A$, Assumption~\ref{ass:Ph} is satisfied with regularity parameter $\gamma < \min(3/2, 2\nu+1/2)$. Since we set $X_0 = 0$ and couple the resolutions by $h = \sqrt{\Delta}$, Theorem~\ref{thm:RV-convergence-corollary}(B) predicts an RMSE of order $\cO(\Delta^{\min(\gamma,1)/2})$. For a Mat\'ern smoothness parameter $\nu > 1/4$, we have $\gamma > 1$, which yields an optimal rate of $\cO(\Delta^{1/2})$. For $\nu \le 1/4$, the rate is $\cO(\Delta^{\gamma/2})$ for any $\gamma < 2\nu+1/2$, so we expect a rate approaching $\cO(\Delta^{\nu+1/4})$. These theoretical rates are consistent with the numerical simulation results shown in Figure~\ref{fig:matern} (left).

Next, we consider pointwise sampling ($O=I_h$). Following the discussion in Example~\ref{ex:stationary-kernels}, we see that we can satisfy the conditions of Assumption~\ref{ass:Ih} with parameters $\gamma$ and $\beta$ such that, under the coupling $h = \sqrt{\Delta}$, the predicted rate in Theorem~\ref{thm:RV-convergence-corollary} is the same as for local averages. This is again what we observe in Figure~\ref{fig:matern} (right).

In these simulations, the SPDE solution for the Matérn case has been replaced by an approximation based on the backward Euler method in time and piecewise linear finite elements in space, with maximal time step size $2^{-18}$ and spatial resolution corresponding to the observation resolution. The noise was interpolated and sampled using the circulant embedding method. We refer to \cite{LordPetersson25} for details on this SPDE approximation. In the error calculation, the true kernel $q$ was approximated by interpolating it on a fine grid with $2^9+1$ nodes. The error itself was approximated by a Monte Carlo method based on $120$ samples for each setting. The simulations were performed for smoothness parameters $\nu = 1/8, 1/4, 3/8$, and observation parameters $\Delta = 2^{-2}, 2^{-4}, \ldots, 2^{-16}$ with $h = \sqrt{\Delta}$.
\begin{figure}[h]
	\centering
	\begin{minipage}{0.48\textwidth}
		\centering
		\includegraphics[width=\textwidth]{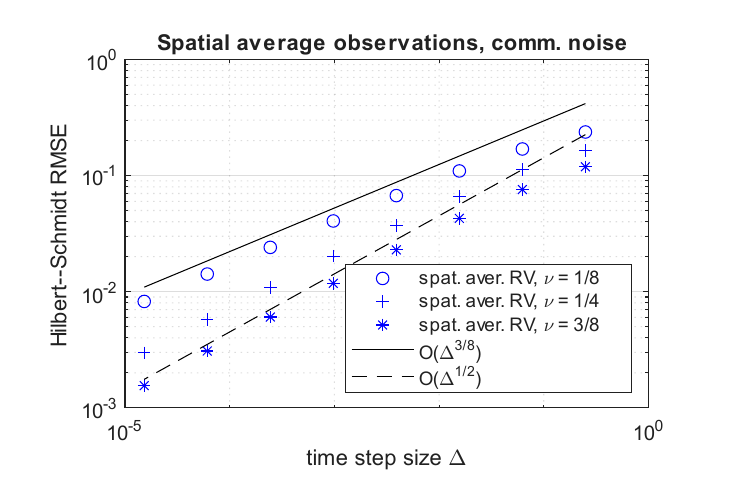}
	\end{minipage}\hfill
	\begin{minipage}{0.48\textwidth}
		\centering
		\includegraphics[width=\textwidth]{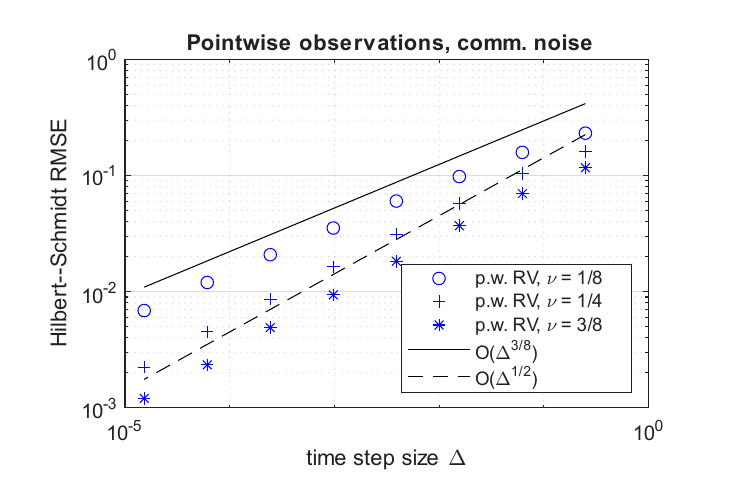}
	\end{minipage}
	\caption{Approximation of the error of Theorem~\ref{thm:RV-convergence-corollary}, with $h = \sqrt{\Delta}$ and $q$ defined in terms of the eigenbasis of $A$.}
	\label{fig:commutative}
\end{figure}
In our second example, we let $A$ be the negative Laplacian with zero Dirichlet boundary conditions and define $Q$ via its eigenfunctions $(e_j)_{j=1}^\infty$ and eigenvalues $(\lambda_j)_{j=1}^\infty$, which are shared with $A$. We set
\begin{equation}
	\label{eq:num-q-comm-def-corrected}
	q(x,y)=\sum_{j=1}^\infty \lambda_j^{-\nu-1/2} e_j(x) e_j(y) ,\quad x,y \in \cD,
\end{equation}
where $\nu > 0$ is a smoothness parameter. This is a specific instance of the kernels discussed in Example~\ref{ex:kernels-commuting-with-A}, with $Q=(-\Delta_{\cD})^{-(\nu+1/2)}$. Also here, Assumption~\ref{ass:rate}\ref{ass:rate:q} is satisfied for $\gamma <  2\nu+1/2$. Consequently, Theorem~\ref{thm:RV-convergence-corollary} predicts the same convergence rates as in the Matérn case.
Since $A$ and $Q$ commute, the SPDE solution can be simulated directly from its spectral expansion, which we truncate at $N=2^9$. We choose observation parameters $\Delta$ and $h$, smoothness parameters $\nu$, and the number of Monte Carlo samples as before. The results are plotted in Figure~\ref{fig:commutative}. Again, the observed errors display the expected rates of convergence for both local average and pointwise sampling.

\subsection{Hypothesis tests}
To demonstrate the performance of the goodness-of-fit tests developed in Section \ref{Sec: Universal Goodness of fit test}, we conducted a series of numerical simulations. We analyze the tests for both a fixed null hypothesis (Theorem~\ref{thm: Universal Goodness of fit}) and a parametric family of hypotheses (Theorem~\ref{thm: Universal Goodness of fit- estimated parameters}). All experiments are performed under the pointwise observation scheme $O = I_h$.

The simulation setup is similar to that of the previous subsection. We consider the Mat\'ern framework from Example~\ref{ex:stationary-kernels} and the commutative framework from Example~\ref{ex:kernels-commuting-with-A}. For the Mat\'ern case, the true solution $X$ is simulated with high resolutions $\Delta = 2^{-16}$ and $h = 2^{-8}$, while for the commutative case, we use $N=300$ eigenpairs. The test statistics are implemented according to Algorithms~\ref{Alg: Pseudocode for fixed Q GoF} and~\ref{Alg: Pseudocode for parametric Q GoF- estimated paramaters}, but with the Hilbert-Schmidt norm and the null kernel approximated on a fine spatial grid with mesh size $h=2^{-8}$. All rejection rates are estimated using $10^4$ independent Monte Carlo samples at a significance level of $\alpha=0.05$.

\begin{figure}[h]
	\centering
	\includegraphics[width=1.0\linewidth]{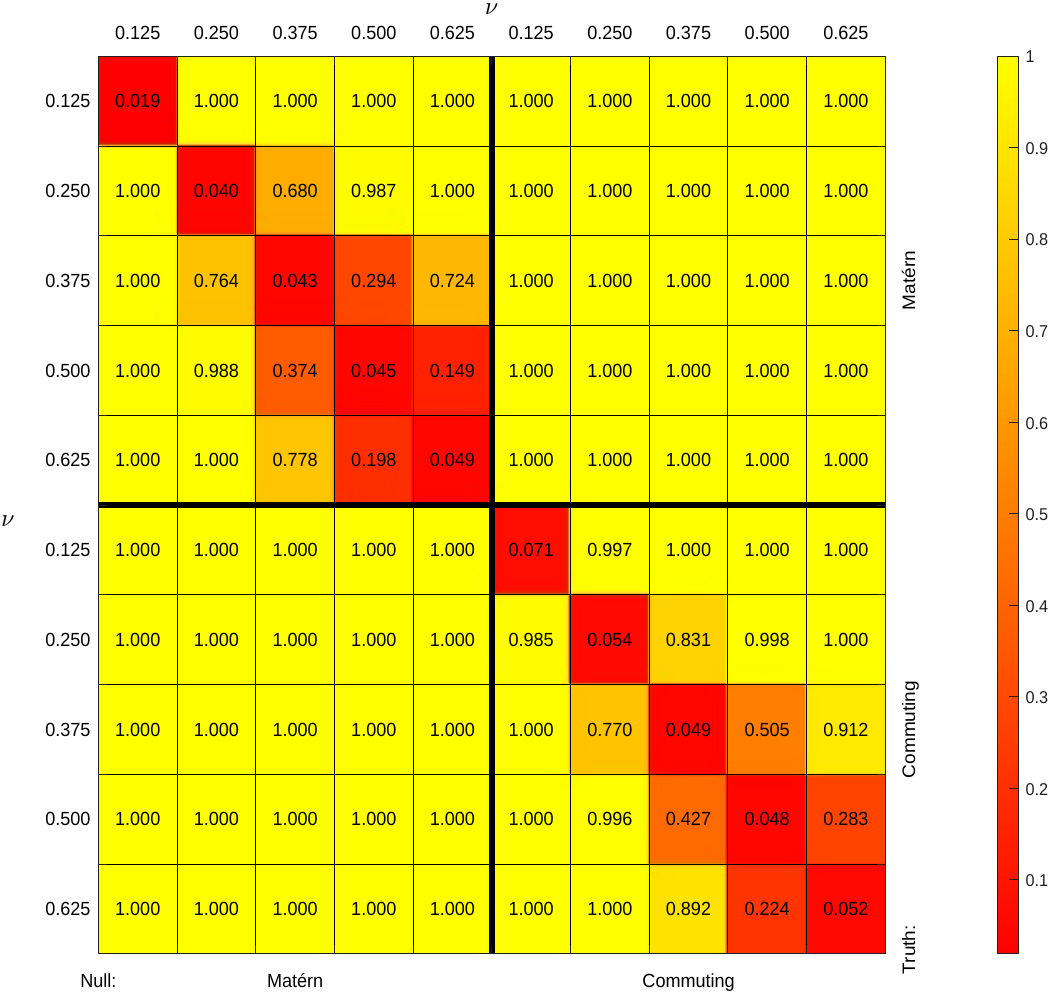}
	\caption{Rejection rates for the test of Theorem~\ref{thm: Universal Goodness of fit} ($\alpha = 0.05$), under $O = I_h$ with $\Delta = 2^{-8}, h = 2^{-4}$.}
	\label{fig:rejection_rates_coarse}
\end{figure}

First, we consider the test for a fixed null hypothesis. We test each of the ten kernels (two types, five smoothness parameters $\nu$) against every other. This setup serves to examine the test's ability to control its size (diagonal entries) and its power to distinguish between different covariance structures (off-diagonal entries). In an initial experiment, we compute the realized variation estimator using observation parameters $\Delta = 2^{-8}$ and $h = 2^{-4}$. The results are shown in Figure~\ref{fig:rejection_rates_coarse}. For smoothness parameters $\nu$ that satisfy the condition $\gamma \ge 1$ of Theorem~\ref{thm: CLT} (in our one-dimensional setting, this corresponds to $\nu > 1/4$), the rejection rates on the diagonal are close to the nominal level of $0.05$. The off-diagonal rates are high, indicating that the test has good power.

To further investigate the behavior of the test when the regularity conditions are not met, we conduct a second experiment with more observations, setting $\Delta = 2^{-12}$ and $h=2^{-6}$. As seen in the left panel of Figure~\ref{fig:rejection_rates_fine_and_parametric}, the empirical sizes for the low smoothness values $\nu=1/8$ and $\nu=1/4$ diverge from $0.05$. This is expected, as these choices violate the assumptions of Theorem~\ref{thm: CLT}, which illustrates the necessity of the regularity conditions for the test's validity.	

\begin{figure}[h]
	\centering
	\begin{minipage}{0.55\textwidth}
		\centering
		\includegraphics[width=\textwidth]{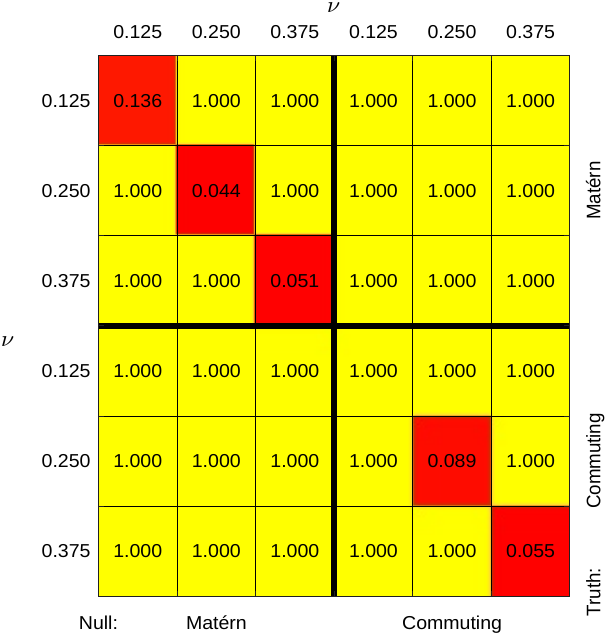}
	\end{minipage}\hfill
	\begin{minipage}{0.4\textwidth}
		\vspace{2.0em}
		\centering
		\includegraphics[width=\textwidth]{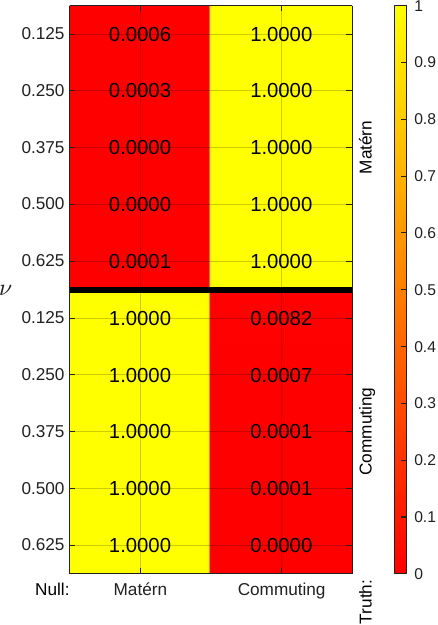}
	\end{minipage}
	\caption{Left: Rejection rates for the goodness-of-fit test of Theorem~\ref{thm: Universal Goodness of fit} ($\alpha = 0.05$) under $O = I_h$ with $\Delta = 2^{-12}, h = 2^{-6}$. Right: Rejection rates for the goodness-of-fit test of Theorem~\ref{thm: Universal Goodness of fit- estimated parameters} ($\alpha=0.05$) with $\Delta = 2^{-8}, h = 2^{-4}$.}
	\label{fig:rejection_rates_fine_and_parametric}
\end{figure}

Next, we study the test for a parametric family of covariances, considering both the Mat\'ern and the commutative kernel families as the null hypothesis. The realized variation is computed from observations with $\Delta = 2^{-8}$ and $h = 2^{-4}$. The results, shown in the right panel of Figure~\ref{fig:rejection_rates_fine_and_parametric}, demonstrate the test's effectiveness in distinguishing between these two qualitatively different kernel families. When the true kernel is from the Mat\'ern family, the test for a Mat\'ern null correctly fails to reject, with rejection rates close to zero, consistent with the conservative nature of the test. Conversely, when the data is generated from a commutative kernel, the same test consistently rejects the null hypothesis with a power close to one. The results are analogous when the roles of the kernel families are reversed.

\bibliographystyle{abbrv}
\bibliography{Bibliography}

\begin{thebibliography}{10}

\bibitem{altmeyer2023}
I.~Altmeyer, R.~Cialenko and G.~Pasemann.
\newblock Parameter estimation for semilinear spdes from local measurements.
\newblock {\em Bernoulli}, 29(3):2035--2061, 2023.

\bibitem{altmeyer2021}
R.~Altmeyer and M.~Rei{\ss}.
\newblock Nonparametric estimation for linear spdes from local measurements.
\newblock {\em Ann. Appl. Probab.}, 31(1):1--38, 2021.

\bibitem{altmeyer2024}
R.~Altmeyer, A.~Tiepner, and M.~Wahl.
\newblock Optimal parameter estimation for linear spdes from multiple
  measurements.
\newblock {\em Ann. Statist.}, 52(4):1307--1333, 2024.

\bibitem{BelgacemBrenner01}
F.~Ben~Belgacem and S.~C. Brenner.
\newblock Some nonstandard finite element estimates with applications to
  {{\(3D\)}} {Poisson} and {Signorini} problems.
\newblock {\em ETNA, Electron. Trans. Numer. Anal.}, 12:134--148, 2001.

\bibitem{Benth2022}
F.~E. Benth, D.~Schroers, and A.~E.~D. Veraart.
\newblock A weak law of large numbers for realised covariation in a {H}ilbert
  space setting.
\newblock {\em Stoch. Proc. Applic.}, 145:241--268, 2022.

\bibitem{BSV2022}
F.~E. Benth, D.~Schroers, and A.~E.~D. Veraart.
\newblock A feasible central limit theorem for realised covariation of spdes in
  the context of functional data.
\newblock {\em Ann.Appl.Prob.}, 2024.

\bibitem{BSV2022supplement}
F.~E. Benth, D.~Schroers, and A.~E.~D. Veraart.
\newblock Supplement to ``a feasible central limit theorem for realised
  covariation of spdes in the context of functional data".
\newblock {\em Ann.Appl.Prob.}, Supplementary Material, 2024.

\bibitem{Bibinger2019}
M.~Bibinger and M.~Trabs.
\newblock On central limit theorems for power variations of the solution to the
  stochastic heat equation.
\newblock In A.~Steland, E.~Rafaj{\l}owicz, and O.~Okhrin, editors, {\em
  Stochastic Models, Statistics and Their Applications}, pages 69--84, Cham,
  2019. Springer International Publishing.

\bibitem{Bibinger2020}
M.~Bibinger and M.~Trabs.
\newblock Volatility estimation for stochastic pdes using high-frequency
  observations.
\newblock {\em Stoch. Proc. Applic.}, 130(5):3005 -- 3052, 2020.

\bibitem{Bogachev1998}
V.~I. Bogachev.
\newblock {\em Gaussian Measures}, volume~62 of {\em Mathematical Surveys and
  Monographs}.
\newblock American Mathematical Society, Providence, RI, 1998.

\bibitem{BrennerScott08}
S.~C. Brenner and L.~R. Scott.
\newblock {\em The mathematical theory of finite element methods}, volume~15 of
  {\em Texts in Applied Mathematics}.
\newblock Springer, New York, third edition, 2008.

\bibitem{Chong2020}
C.~Chong.
\newblock High-frequency analysis of parabolic stochastic pdes.
\newblock {\em Ann. Statist.}, 48(2):1143--1167, 04 2020.

\bibitem{ChongDalang2020}
C.~Chong and R.~C. Dalang.
\newblock Power variations in fractional {S}obolev spaces for a class of
  parabolic stochastic {PDE}s.
\newblock {\em E-print arXiv:2006.15817}, 2020.

\bibitem{Cialenco2018}
I.~Cialenco.
\newblock Statistical inference for {SPDE}s: an overview.
\newblock {\em Statist. Inf. Stoch. Proc.}, 20(2):309--329, 12 2018.

\bibitem{Cialenco2020}
I.~Cialenco and Y.~Huang.
\newblock A note on parameter estimation for discretely sampled {SPDE}s.
\newblock {\em Stoch. Dynamics}, 20(03), 2020.

\bibitem{Cont2005}
R.~Cont.
\newblock Modeling term structure dynamics: An infinite dimensional approach.
\newblock {\em Int. J. Theor. Appl. Finance}, 08(03):357--380, 2005.

\bibitem{DPZ2014}
G.~Da~Prato and J.~Zabczyk.
\newblock {\em Stochastic Equations in Infinite Dimensions}, volume 152 of {\em
  Encyclopedia of Mathematics and its Applications}.
\newblock Cambridge University Press, Cambridge, second edition, 2014.

\bibitem{DupontScott80}
T.~Dupont and R.~Scott.
\newblock Polynomial approximation of functions in {S}obolev spaces.
\newblock {\em Math. Comp.}, 34(150):441--463, 1980.

\bibitem{Gapaillard1974}
C.~Gapaillard.
\newblock Un r\'{e}sultat de compacit\'{e} pour l'interpolation de couples
  hilbertiens.
\newblock {\em C. R. Acad. Sci. Paris S\'{e}r. A}, 278:681--684, 1974.

\bibitem{Gaudlitz2023}
S.~Gaudlitz and M.~Rei{\ss}.
\newblock Estimation for the reaction term in semi-linear spdes under small
  diffusivity.
\newblock {\em Bernoulli}, 29(4):3033--3058, 2023.

\bibitem{Hairer09}
M.~Hairer.
\newblock An introduction to stochastic {PDE}s, 2009.
\newblock Preprint at arXiv:0907.4178.

\bibitem{Hildebrandt2021}
F.~Hildebrandt and M.~Trabs.
\newblock {Parameter estimation for SPDEs based on discrete observations in
  time and space}.
\newblock {\em Electron. J. Stat.}, 15(1):2716 -- 2776, 2021.

\bibitem{hildebrandt2023}
F.~Hildebrandt and M.~Trabs.
\newblock Nonparametric calibration for stochastic reaction--diffusion
  equations based on discrete observations.
\newblock {\em Stoch. Proc. Applic.}, 162:171--217, 2023.

\bibitem{JacodProtter2012}
J.~Jacod and P.~Protter.
\newblock {\em Discretization of Processes}, volume~67 of {\em Stochastic
  Modelling and Applied Probability}.
\newblock Springer, Heidelberg, 2012.

\bibitem{Uchida2020}
Y.~Kaino and M.~Uchida.
\newblock Parametric estimation for a parabolic linear spde model based on
  discrete observations.
\newblock {\em J. Statist. Plann. Inference}, 2020.

\bibitem{KovacsLangPetersson2023}
M.~Kovács, A.~Lang, and A.~Petersson.
\newblock Hilbert--{S}chmidt regularity of symmetric integral operators on
  bounded domains with applications to {SPDE} approximations.
\newblock {\em Stochastic Analysis and Applications}, 41(3):564--590, 2023.

\bibitem{Kruse2014}
R.~Kruse.
\newblock {\em Strong and weak approximation of semilinear stochastic evolution
  equations}.
\newblock Springer, 2014.

\bibitem{lindgren2020}
F.~Lindgren, H.~Bakka, D.~Bolin, E.~Krainski, and H.~Rue.
\newblock A diffusion-based spatio-temporal extension of gaussian mat{\`e}rn
  fields.
\newblock {\em arXiv preprint arXiv:2006.04917}, 2020.

\bibitem{Lingren2011}
F.~Lindgren, H.~Rue, and J.~Lindström.
\newblock An explicit link between gaussian fields and gaussian markov random
  fields: the stochastic partial differential equation approach.
\newblock {\em J. R. Statist. Soc. B}, 73(4):423–--498, 2011.

\bibitem{LionsMagenes1972}
J.~L. Lions and E.~Magenes.
\newblock {\em Non-homogeneous boundary value problems and applications. {V}ol.
  {I}}.
\newblock Springer-Verlag, New York-Heidelberg, 1972.
\newblock Translated from the French by P. Kenneth, Die Grundlehren der
  mathematischen Wissenschaften, Band 181.

\bibitem{liu2022}
X.~Liu, K.~M. Yeo, and S.~Y. Lu.
\newblock Statistical modeling for spatio-temporal data from stochastic
  convection-diffusion processes.
\newblock {\em J. Amer. Statist. Assoc.}, 117:1482--1499, 2022.

\bibitem{LordPetersson25}
G.~J. Lord and A.~Petersson.
\newblock Piecewise linear interpolation of noise in finite element
  approximations of parabolic spdes.
\newblock {\em SIAM Journal on Numerical Analysis}, 63(2):542--563, 2025.

\bibitem{Reiss23}
M.~Rei{\ss}, C.~Strauch, and L.~Trottner.
\newblock Change point estimation for a stochastic heat equation.
\newblock {\em Ann. Statist. (to appear)}, 2023.

\bibitem{Riedle11}
M.~Riedle.
\newblock Cylindrical {Wiener} processes.
\newblock In {\em S\'eminaire de Probabilit\'es XLIII, Poitiers, France, Juin
  2009.}, pages 191--214. Berlin: Springer, 2011.

\bibitem{Schroers2024b}
D.~Schroers.
\newblock Dynamically consistent analysis of realized covariations in term
  structure models.
\newblock {\em arXiv preprint arXiv:2406.19412}, 2024.

\bibitem{Schroers2024}
D.~Schroers.
\newblock Robust functional data analysis for stochastic evolution equations in
  infinite dimensions.
\newblock {\em arXiv preprint arXiv:2401.16286}, 2024.

\bibitem{Sigrist2012}
F.~Sigrist, H.~K{\"u}nsch, and W.~Stahel.
\newblock A dynamic nonstationary spatio-temporal model for short term
  prediction of precipitation.
\newblock {\em Ann. Appl. Stat.}, 2012.

\bibitem{Sigrist2015}
F.~Sigrist, H.~K{\"u}nsch, and W.~Stahel.
\newblock Stochastic partial differential equation based modelling of large
  space-time data sets.
\newblock {\em J. R. Statist. Soc. B}, 77(1):3--33, 2015.

\bibitem{Stein1999}
M.~Stein.
\newblock {\em Interpolation of spatial data: some theory for kriging}.
\newblock Springer Science \& Business Media, 1999.

\bibitem{vanderVaart1996}
A.~W. van~der Vaart and J.~Wellner.
\newblock {\em Weak Convergence and Empirical Processes: with Applications to
  Statistics}.
\newblock Springer Science \& Business Media, New York, 1996.

\bibitem{Wendland04}
H.~Wendland.
\newblock {\em Scattered data approximation}, volume~17 of {\em Cambridge
  Monographs on Applied and Computational Mathematics}.
\newblock Cambridge University Press, Cambridge, 2005.

\bibitem{Yagi2010}
A.~Yagi.
\newblock {\em Abstract parabolic evolution equations and their applications}.
\newblock Springer monographs in mathematics. Springer-Verlag, Berlin, 2010.

\end{thebibliography}

\begin{appendix}
	
	\section{Pointwise evaluation of the SPDE}\label{Sec: Pointwise Sampling}	
	
	This section provides the technical framework for the pointwise sampling scheme presented in Section~\ref{Sec: Sampling setting}, which requires the solution $X_t$ of the SPDE \eqref{eq:spde} to be well-defined at specific points $x \in \cD$. Recall that the solution is given by the mild formulation \eqref{eq:mild}, $X_t = S(t)X_0 + Y_t$, where $Y_t := \int_0^t S(t-s) \dd W_s$ is the stochastic convolution.
	
	For the point evaluation $X_t(x)$ to be meaningful, $X_t$ must possess sufficient spatial regularity. By the Sobolev embedding theorem, this is the case if $X_t$ takes values in a Sobolev space $H^r$ for some $r > d/2$. For such $r$, $H^r$ is a reproducing kernel Hilbert space (see \cite[p.\ 133 and Corollary~10.48]{Wendland04}), which implies that the evaluation functional $\delta_x: u \mapsto u(x)$ is continuous. 
	
	For $Y_t$ to be a square-integrable $\dot{H}^r$-valued random variable, strong conditions on $Q$ (i.e., $\| A^{(r-1)/2}Q^{1/2} \|_{\cL_2(H)} < \infty$) are required.  Our framework relies on the weaker notion of cylindrical Gaussian random variables. This approach, for which we refer to \cite{Riedle11} for an introduction, is sufficient under the operator norm conditions of Assumption~\ref{ass:Ih}.
	
	\begin{proposition}
		\label{prop:cylindrical-pointwise}
		Suppose that $\| A^{(r-1)/2}Q^{1/2} \|_{\cL(H)} < \infty$ for some $r\ge 1$, and that $\Gamma \in \cL_2(\dot{H}^r,U)$ for some Hilbert space $U$. Then $Y_t$ defines a cylindrical random variable $Y_t$ in $\dot{H}^r$ such that the cylindrical random variable $\Gamma Y_t$ in $U$ is induced by $$\int^t_0 \Gamma S(t-s) \dd W(s).$$
	\end{proposition}
	
	Since $\Gamma = \delta_x$ has finite rank and is continuous, it satisfies $\delta_x \in \cL_2(H^r, \R)$ when $r > d/2$. Under Assumption~\ref{ass:Ih}, the conditions of Proposition~\ref{prop:cylindrical-pointwise} are met and we obtain the following corollary.
	
	\begin{corollary}
		\label{cor:pointwise-well-def}
		Under Assumption~\ref{ass:Ih},
		$$Y_t(x) := \int^t_0 \delta_x S(t-s) \dd W(s)$$
		is a well-defined real-valued random variable. If also $\iota > d/2$ in Assumption~\ref{ass:rate}\ref{ass:rate:initial}, then pointwise values $X_t(x) := (S(t)X_0)(x) + Y_t(x)$ are well-defined, and for any $p \ge 1$,
		$$ \sup_{t \ge 0, x \in \cD} \E[| X_t(x) |^p] < \infty. $$
	\end{corollary}
	
	The uniformity in $x$ follows from the fact that the kernel of $H^r, r > d/2$ is in turn uniformly bounded. The bound in time is then straightforward to deduce using standard estimates for analytic semigroups (see \eqref{eq:semigroup-time-bound}-\eqref{eq:semigroup-integral}) and the Burkholder-Davis-Gundy inequality.
	
	\section{Proofs}\label{Sec. Proofs}
	This section contains the formal proofs Sections   \ref{Sec: Sampling setting} and \ref{Sec: Asymptotic Theory}. In these proofs, we will make frequent use of the properties of $S$, extended to an analytic semigroup on $\dot{H}^r$ for all $r \in \R$.  Specifically, we need the following estimates, which hold for some constant $C<\infty$, and where $E(t) := I - S(t)$,
	\begin{equation}
		\label{eq:semigroup-time-bound}
		\| E(t)^{\frac s 2} A^{-\frac{r}{2}} v \| \le C t^{r/2} \| v \|_{H}, \quad s \in \{1,2\}, r \in [0,s], v \in H, t \ge 0,
	\end{equation}	
	\begin{equation}
		\label{eq:semigroup-analyticity}
		\| A^\frac{r}{2} S(t) v \| \le C t^{-r/2} \| v \|, \quad r \ge 0, v \in H, t > 0,
	\end{equation}
	\begin{equation}
		\label{eq:semigroup-exp-decay}
		\| S(t) v \| \le e^{-\lambda_1 t} \| v \|, \quad v \in H, t \ge 0, \text{ and }
	\end{equation}
	\begin{equation}
		\label{eq:semigroup-integral}
		\int^t_s \| A^\frac{r}{2} S(\tau) v \|^2 \dd \tau \le C (t-s)^{(1-r)} \| v \|^2 \quad r \in [0,1], v \in H, t \ge s \ge 0.
	\end{equation}
	For a proof of these results, as well as a rigorous introduction to the semigroup setting we consider, we refer to \cite[Appendix~B]{Kruse2014}.
	The last property immediately yields that for any additional Hilbert space $U$ and operator $\Gamma \in \cL_2(U,H)$, 
	\begin{equation}
		\label{eq:semigroup-hs}
		\int^t_s \| A^\frac{r}{2} S(\tau) \Gamma \|_{\cL_2(U,H)}^2 \dd \tau \le C (t-s)^{(1-r)} \| \Gamma \|_{\cL_2(U,H)}^2 \quad r \in [0,1], t \ge s \ge 0.
	\end{equation}
	
	Throughout the proofs, we make use of the fact that $A$ and $S$ commute, usually without saying so explicitly.

	For brevity, we also write
	\begin{align*}
		& \tilde \Delta W_O^i := \int^{i\Delta}_{(i-1)\Delta} OS(i\Delta - r) \dd W(r), \\
		&E^{i,j}_{k,O} := \int^{i\Delta}_{j\Delta} O E(\Delta) S((i+k)\Delta-r) \dd W(r) \text{ and }\\
		&E^{i,j}_{O} := E^{i,j}_{0,O},
	\end{align*}
	for $i,j = 1, \ldots, \ulT$.  The next lemma establishes bounds for second moments of these increments.
	\begin{lemma}\label{lem: second moments of increments} Let Assumption \ref{ass: minimal Assumption on A} hold. In this case, we can find a constant $C$ as well as a sequence $\zeta_{\Delta}$ both independent of $\Delta$ and $T$ such that  
		\begin{align}
			&\sup_{i=1,...,\ulT}\| \tilde \Delta W_I^i\|_{L^2(\Omega,H)}^2 \leq  C \Delta,\label{eq: bound on adjusted increment}\\
			&\sup_{j=1,...,\ulT} \| E^{j-1,0}_{I}\|_{L^2(\Omega,H)}^2 \leq C \Delta \zeta_{\Delta},\label{eq: bound on adjusted increment remainder}\\
			& \sum^{\ulT-1}_{j=1} \sum^{\ulT}_{i = j+1} \|E^{j,j-1}_{i-j,I}\|^2_{L^2(\Omega,H)} \leq T C  \zeta_{\Delta} \quad\text{and }\label{eq: bound on mixed term sum}\\
			&	\sum^{\ulT-1}_{j=1} \sum^{\ulT}_{i = j+1} \|E^{j-1,0}_{i-j,I}\|^2_{L^2(\Omega,H)} 
			\leq  T C .\label{eq: mixed sum bound II} 
		\end{align}
		Furthermore, if Assumption~\ref{ass:Ih} also holds, we find that 
		\begin{align}
			&\sup_{i=1,...,\ulT}\| \tilde \Delta W_{I_h}^i\|_{L^2(\Omega,H)}^2 \leq  C \Delta^{\min(1,(1+\beta)/2)},\label{eq: bound on adjusted increment-I_h case}\\
			&\sup_{j=1,...,\ulT} \| E^{j-1,0}_{I_h}\|_{L^2(\Omega,H)}^2 \leq C \Delta^{\min(2,(1+\beta+\epsilon/2)/2)},\label{eq: bound on adjusted increment remainder-I_h case}\\
			&	\sum^{\ulT-1}_{j=1} \sum^{\ulT}_{i = j+1} \|E^{j-1,0}_{i-j,I_h}\|^2_{L^2(\Omega,H)} 
			\leq   T C \Delta^{\frac{\min(0,\beta-1)}{2}}\quad\text{and }\label{eq: mixed sum bound I-I_h case} \\
			&	\sum^{\ulT-1}_{j=1} \sum^{\ulT}_{i = j+1} \|E^{j,j-1}_{i-j,I_h}\|^2_{L^2(\Omega,H)} 
			\leq  T C \Delta^{\frac{\min(1,\beta-1)+\epsilon/2}{2}}. \label{eq: mixed sum bound II-I_h case} 
		\end{align}
	\end{lemma}
	\begin{proof}
		The bound for $\| \tilde \Delta W_O^i\|_{L^2(\Omega,H)}^2$  follows from It{\^o}'s isometry and the boundedness of the semigroup in the operator norm.  
		
		For the derivation of \eqref{eq: bound on adjusted increment remainder} observe that  when $P_N= \sum_{i=1}^{N} e_i^{\otimes 2}$ is the projection onto the first $N$ eigenvalues $e_1,...,e_N$ of $A$,  we can split
		\begin{align*}
			& \sup_{j=1,...,\ulT} \| E^{j-1,0}_{I}\|_{L^2(\Omega,H)}^2\\
			= & 	\sup_{j=1,...,\ulT}\int_0^{(j-1)\Delta} \| E(\Delta) S((j-1)\Delta-u) Q^{\frac 12}\|_{\cL_2(H)}^2du\\
			= &	\sup_{j=1,...,\ulT} \int_0^{(j-1)\Delta} \| P_N E(\Delta) S((j-1)\Delta-u) Q^{\frac 12}\|_{\cL_2(H)}^2du\\
			&+	\sup_{j=1,...,\ulT}\int_0^{(j-1)\Delta} \| (I-P_N) E(\Delta) S((j-1)\Delta-u) Q^{\frac 12}\|_{\cL_2(H)}^2du\\
			=&	\sup_{j=1,...,\ulT}	 (1)^{j}_{\Delta,N}+(2)^{j}_{\Delta,N}.
		\end{align*}
		Since $\|P_N A\|_{\cL(H)}<\infty$ we have 
		\begin{align*}
			&\sup_{j=1,...,\ulT}	 (1)^{j}_{\Delta,N}\\ 
			\leq &\|P_N A\|^2_{\cL(H)} \| E(\Delta) A^{-1}\|_{\cL(H)}^2 \int_0^{(j-1)\Delta} \| S((j-1)\Delta-u) Q^{\frac 12}\|_{\cL_2(H)}^2du\\
			\lesssim &\Delta^2 \|P_N A\|^2_{\cL(H)}, 
		\end{align*}
		which  proves that $\Delta^{-1} i_{\Delta,N}\to 0$ for $\Delta\to 0$ and $N$ fixed. Moreover,   using \eqref{eq:semigroup-integral} we find
		\begin{align*}
			\sup_{i=1,...,\ulT}	 (2)^{j}_{\Delta,N}\leq &\|  E(\Delta) A^{-\frac 12}\|_{\cL(H)}^2 \int_0^{(j-1)\Delta}   \|  A^{\frac 12}S(u) (I-P_N) Q^{\frac 12}\|_{\cL_2(H)}^2 du \\
			\lesssim &  \Delta   \|(I-P_N) Q^{\frac 12}\|_{\cL_2(H)}^2.
		\end{align*}
		Since $Q^{\frac 12}$ is Hilbert-Schmidt,  $  \|(I-P_N) Q^{\frac 12}\|_{\cL_2(H)}^2\to 0$ as $N\to \infty$. This yields that $\lim_{n\to \infty}\sup_{\Delta\in (0,1]}\Delta^{-1} \sup_{i=1,...,\ulT}	 (2)^{j}_{\Delta,N}=0$. 
		In this case we could find an $N_{\epsilon}$ independent of $\Delta$ such that $\sup_{0<\Delta\leq 1} \Delta^{-1} \sup_{i=1,...,\ulT}	 (2)^{j}_{\Delta,N}<\epsilon/2$ as well as an $n_{\epsilon}$ only depending on $N_{\epsilon}$ (and hence only on $\epsilon$) such that $$\Delta^{-1} \sup_{i=1,...,\ulT}	 (1)^{j}_{\Delta,N_{\epsilon}}<\epsilon/2$$ for all $\Delta<\Delta_{\epsilon}$.  Hence, for all $\epsilon> 0$ we have for $\Delta<\Delta_{\epsilon}$ that 
		\begin{align*}
			&\sup_{i=1,...,\ulT}    \Delta^{-1}  \| E^{j-1,0}_{I}\|_{L^2(\Omega,H)}^2 \\
			&\quad\leq\Delta^{-1}\left( \sup_{i=1,...,\ulT}	 (1)^{j}_{\Delta,N_{\epsilon}} +  \sup_{i=1,...,\ulT}	 (2)^{j}_{\Delta,N_{\epsilon}}\right)\\
			&\quad\leq \epsilon.
		\end{align*}
		Since this holds for all $\epsilon>0$,  we have shown \eqref{eq: bound on adjusted increment remainder}. 
		
		We now turn to the proof of \eqref{eq: bound on mixed term sum}.
		Using that
		\begin{align*}
			&\sum^{\ulT-1}_{j=1} \sum^{\ulT}_{i = j+1} \| E(\Delta)^\frac{1}{2} S(\Delta(i-j)) \Gamma \|^2_{\cL_2(H)} \\
			&\quad= \sum_{k=1}^\infty \sum^{\ulT-1}_{j=1} \sum^{\ulT}_{i = j+1} (1-e^{-\lambda_k \Delta}) e^{-2\lambda_k \Delta(i-j)}  \| \Gamma^* e_k \|^2 \\
			&\quad= \sum_{k=1}^\infty \frac{(1-e^{-\lambda_k \Delta})((\ulT-1)(e^{2\lambda_k \Delta}-1)-(1-e^{-\lambda_k \Delta (\ulT-1)}))}{(e^{2\lambda_k \Delta}-1)^2} \| \Gamma^* e_k \|^2 \\
			&\quad\le \sum_{k=1}^\infty \frac{(1-e^{-\lambda_k \Delta})(\ulT-1)}{e^{2\lambda_k \Delta}-1} \| \Gamma^* e_k \|^2,
		\end{align*}
		and that $x \mapsto (1-e^{-x})/(e^{2 x}-1)$ is bounded for any Hilbert Schmidt operator $\Gamma$, we find for $\Gamma= E(\Delta)^{\frac 12}S(r)Q^{\frac 12}$ that
		\begin{align*}
			\sum^{\ulT-1}_{j=1} \sum^{\ulT}_{i = j+1} \|E^{j,j-1}_{i-j,I}\|^2_{L^2(\Omega,H)} \lesssim & \frac T {\Delta} \int_0^{\Delta}\| E(\Delta)^{\frac 12} S(r) Q^{\frac 12}\|^2_{\cL_2(H)}dr\\
			\leq &  \frac T {\Delta}     \int_0^{\Delta}\| P_N E(\Delta)^{\frac 12} S(r) Q^{\frac 12}\|^2_{\cL_2(H)}dr\\
			&+   \frac T {\Delta} \int_0^{\Delta}\| (I-P_N) E(\Delta)^{\frac 12} S(r) Q^{\frac 12}\|^2_{\cL_2(H)}dr\\
			= & (A)_{\Delta,N}+(B)_{\Delta,N}.
		\end{align*}
		Now it is
		\begin{align*}
			& \int_0^{\Delta}\| P_N E(\Delta)^{\frac 12} S(r) Q^{\frac 12}\|^2_{\cL_2(H)}dr\\
			\leq & \| E(\Delta)^\frac{1}{2} A^{-\frac{1}{2}}\|^2_{\cL(H)} 	 \int_0^{\Delta}\| P_N A^{\frac 12} S(r) Q^{\frac 12}\|^2_{\cL_2(H)}dr\\
			\leq & \Delta 	 \int_0^{\Delta}\| P_N A^{\frac 12}\|_{\cL(H)}\| S(r) Q^{\frac 12}\|^2_{\cL_2(H)}dr\\
			\lesssim & \Delta^2 \| P_N A^{\frac 12}\|_{\cL(H)}.
		\end{align*}		
		This shows that $\lim_{\Delta \to 0}   (A)_{\Delta,N}=0$ for $N$ fixed. Moreover,  we can derive
		\begin{align*}
			(B)_{\Delta,N} \lesssim &  \| E(\Delta)^\frac{1}{2} A^{-\frac{1}{2}}\|^2_{\cL(H)}	 \int_0^{\Delta}\|  A^{\frac 12} S(r) (I-P_N)Q^{\frac 12}\|^2_{\cL_2(H)} dr\\
			\leq & \Delta   \|(I-P_N)Q^{\frac 12}\|^2_{\cL_2(H)}.
		\end{align*}
		Since $ \|(I-P_N)Q^{\frac 12}\|^2_{\cL_2(H)}\to 0$ as $N\to \infty$, we find that $\lim_{N\to \infty} (B)_{\Delta,N}=0$. The proof of \eqref{eq: bound on mixed term sum} now follows analogously as in the proof of \eqref{eq: bound on adjusted increment remainder}.
		
		Similarly, we obtain
		\begin{align*}
			\sum^{\ulT-1}_{j=1} \sum^{\ulT}_{i = j+1} \|E^{j-1,0}_{i-j,I}\|^2_{L^2(\Omega,H)} \lesssim & \frac T {\Delta} \int_0^{(j-1)\Delta}\| E(\Delta)^{\frac 12} S(r) Q^{\frac 12}\|^2_{\cL_2(H)}dr\\
			\leq  &   \frac T {\Delta}	 \| E(\Delta) A^{\frac 12}\|_{\cL(H)}\int_0^{(j-1)\Delta}\| A^{\frac 12}S(r) Q^{\frac 12}\|^2_{\cL_2(H)}dr\\
			\lesssim & T. 
		\end{align*}
		That shows \eqref{eq: mixed sum bound II}.
		
		Let us now assume that { Assumption~\ref{ass:Ih} is satisfied and we have by \eqref{eq:semigroup-hs},
			\begin{align}\label{eq: bound on W_I_h}
				\|\tilde \Delta W^1_O\|^2_{L^2(\Omega,H)} &= \int_{0}^{\Delta} \| O S(r) Q^{\frac 1 2} \|_{\cL_2(H)}^2 \dd r \notag\\
				&= \int_{0}^{\Delta} \| O A^{-\frac{d+\epsilon}{4}} A^{\frac{1-\beta}{4}} S(r) A^{\frac{\beta + d + \epsilon - 1}{4}} Q^{\frac 1 2} \|_{\cL_2(H)}^2 \dd r \notag\\
				&\lesssim \int_{0}^{\Delta} \| O A^{-\frac{d+\epsilon}{4}} A^{\frac{1-\beta}{4}} S(r) \|_{\cL_2(H)}^2 \dd r \notag\\
				&= \sum_{i, j = 1}^\infty \int_{0}^{\Delta} \langle A^{\frac{1-\beta}{4}}S(r) e_i, (O A^{-\frac{d+\epsilon}{4}})^*e_j \rangle^2 \dd r \notag \\
				&\lesssim \int_{0}^{\Delta} \|A^{\frac{1-\beta}{4}} S(r) (O A^{-\frac{d+\epsilon}{4}})^* \|_{\cL_2(H)}^2 \dd r \lesssim \Delta^{\min(1,(1+\beta)/2)}.
			\end{align}
			In the last step, we also used  Lemma~\ref{lem:obs-hs}, which here yields $\| (OA^{-(d+\epsilon)/4})^*\|_{\cL_2(H)} = \| O\|_{\cL_2(\dot{H}^{d/2+\epsilon/2},H)} < \infty$. Moreover, for $O=I_h$, we find
			\begin{align}\label{eq: bound on E_0_I_h                                                                                                                                                     }
				\|E^{j-1,0}_O\|^2_{L^2(\Omega,H)} &= \int_{0}^{j\Delta} \| O S(r) E(\Delta) Q^{\frac 1 2} \|_{\cL_2(H)}^2 \dd r \notag\\
				&\lesssim \int_{0}^{(j-1)\Delta} \| A^{\frac 1 2}S(r) A^{-\frac{\beta+1+\frac{\epsilon}2}4} E(\Delta) (O A^{-\frac{d+\frac{\epsilon}2}4})^*\|_{\cL_2(H)}^2 \dd r \notag\\
				&\lesssim \| A^{{-\frac{\beta+1+\frac{\epsilon}2}{4}}} E(\Delta ) \|_{\cL(H)}^2 \lesssim \Delta^{\min(2,(1+\beta+\frac{\epsilon}2)/2)},
			\end{align}
			
			Using that, for arbitrary $\Gamma \in \cL_2(H)$,
			\begin{align*}
				\sum^{\ulT-1}_{j=1} \sum^{\ulT}_{i = j+1} \| E(\Delta)^\frac{1}{2} S(\Delta(i-j)) \Gamma \|^2_{\cL_2(H)} \le \sum_{k=1}^\infty \frac{(1-e^{-\lambda_k \Delta})(\ulT-1)}{e^{2\lambda_k \Delta}-1} \| \Gamma^* e_k \|^2,
			\end{align*}
			and that $x \mapsto (1-e^{-x})/(e^{2 x}-1)$ is bounded we obtain
			\begin{align*}
				&\sum^{\ulT-1}_{j=1} \sum^{\ulT}_{i = j+1} \|E^{j,j-1}_{i-j,O}\|^2_{L^2(\Omega,H)} \\
				&\hspace{0.75em}= \int_{0}^{\Delta} \sum^{\ulT-1}_{j=1} \sum^{\ulT}_{i = j+1} \| (A^{\frac{\beta + d + \epsilon - 1}{4}}Q^{\frac 1 2})^* E(\Delta) S((i - j)\Delta) \\  &\hspace{11em}\times A^{\frac{1-\beta-\frac{\epsilon}2}{4}} S(r)  (O A^{-\frac{d+\frac{\epsilon}2}{4}})^*\|_{\cL_2(H)}^2 \dd r \\
				&\hspace{0.75em}\lesssim \int_{0}^{\Delta} \sum^{\ulT-1}_{j=1} \sum^{\ulT}_{i = j+1} \| E(\Delta)^{\frac 1 2} S((i - j)\Delta) \\  &\hspace{11em}\times 
				E(\Delta)^{\frac 1 2} A^{\frac{1-\frac{\epsilon}2-\beta}{4}} S(r)  (O A^{-\frac{d+\frac{\epsilon}2}{4}})^*\|_{\cL_2(H)}^2 \dd r \\
				&\hspace{0.75em} \lesssim \frac{T }{\Delta} \int_{0}^{\Delta} \| E(\Delta)^{\frac 1 2} A^{-\frac{1+\frac{\epsilon}2}{4}}\|^2_{\cL(H)} \|A^{\frac{2-\beta}{4}} S(r)  (O A^{-\frac{d+\frac{\epsilon}2}{4}})^*\|_{\cL_2(H)}^2 \dd r \\
				&\hspace{0.75em}\lesssim T \Delta^{\frac{\min(1,\beta-1)}{2}} \Delta^{\frac{\epsilon}4}.
			\end{align*}
			Similarly, 
			\begin{align*}
				&\sum^{\ulT-1}_{j=1} \sum^{\ulT}_{i = j+1} \|E^{j-1,0}_{i-j,O}\|^2_{L^2(\Omega,H)} \\
				&\quad= \sum^{\ulT-1}_{j=1} \sum^{\ulT}_{i = j+1} \int_{0}^{\Delta(j-1)} \|  O E(\Delta) S(\Delta(i-j)) S(r) Q^{\frac 1 2} \|_{\cL_2(H)}^2 \dd r \\
				&\quad\lesssim \int_{0}^{T} \sum^{\ulT-1}_{j=1} \sum^{\ulT}_{i = j+1} \|  O E(\Delta) S(\Delta(i-j)) S(r) Q^{\frac 1 2} \|_{\cL_2(H)}^2 \dd r \\
				&\quad\lesssim \ulT \int_{0}^{T} \| E(\Delta)^{\frac 1 2} A^{\frac{1+\beta}{4}} \|^2_{\cL(H)} \|A^{\frac{1}{2}} S(r)  (O A^{-\frac{d+\epsilon}{4}})^*\|_{\cL_2(H)}^2 \dd r \\\
				&\quad\lesssim T \Delta^{\frac{\min(0,\beta-1)}{2}}.
		\end{align*}}
		
		The proof is complete.

	\end{proof}
	
	\subsection{Proofs of Section \ref{Sec: Identifiability}}
	
	\begin{proof}[Proof of Theorem \ref{thm: Identifiability}]
		Let us first assume that $X_0=0$ almost surely. In this case,  Lemma \ref{lem:var-conv} below guarantees that the variance of realized covariation converges to $0$, so we only need to show convergence of the bias. For that, observe that
		\begin{align*}
			Bias\left(\frac{RV_T^{\Delta,I}}T\right)=& \frac 1T\sum_{i=1}^{\ulT} \left(\mathbb E\left[\left(\tilde \Delta W_I^i\right)^{\otimes 2}\right]-\Delta Q\right)+\mathbb E\left[\left(E_I^{i-1,0}\right)^{\otimes 2}\right].
			% = & \sum_{i=1}^{\ulT} \int_0^{\Delta} S(u)QS(u)du-Q+\sum_{i=1}^{\ulT} \int_0^{(i-1)\Delta} S(u)QS(u) du
		\end{align*}
		Due to to the boundedness of the semigroup and using the triangle inequality we find for the first term that
		\begin{align*}
			\left\| \sum_{i=1}^{\ulT} \left(\mathbb E\left[\left(\tilde \Delta W_I^i\right)^{\otimes 2}\right]-\Delta Q\right)\right\|_{\cL_2(H)}= & \left\|\sum_{i=1}^{\ulT} \int_0^{\Delta} \left(S(u)QS(u)-Q\right)du\right\|_{\cL_2(H)}\\
			\lesssim & T \sup_{u\leq \Delta} \|(I-S(u))Q^{\frac 12}\|_{\cL(H)} \|Q^{\frac 12}\|_{\cL_2(H)}.
		\end{align*}
		This converges to $0$ as $\Delta \to 0$ by Lemma 5.1(i) in \cite{Benth2022}. 
		
		The second summand also converges to $0$, since by \eqref{eq: bound on adjusted increment remainder} it is 
		\begin{align*}
			\left\|\sum_{i=1}^{\ulT} \mathbb E\left[\left(E_I^{i-1,0}\right)^{\otimes 2}\right]\right\|_{\cL_2(H)}\leq \sum_{i=1}^{\ulT} \left\|E_I^{i-1,0}\right\|_{L^2(\Omega,H)}^2 \lesssim T\zeta_{\Delta},
		\end{align*}
		for a real-valued sequence $\zeta_{\Delta}$ converging to $0$. 
		
		Let us now turn to the case that $X_0\neq 0$ for which we recall the notation $Y_t=X_t-S(t)X_0$. By the triangle inequality and the Cauchy--Schwartz inequality we have
		\begin{align*}
			&\left\|\frac{RV_T^{\Delta,I}}T-Q\right\|_{\cL_2(H)}\\
			%  \lesssim & \left\|T^{-1}\sum_{i=1}^{\ulT} (Y_{i\Delta}-Y_{(i-1)\Delta}){\otimes 2}-Q\right\|_{\cL_2(H)}\\
			% &+ T^{-1}\| E(\Delta)X_0\|\sum_{i=1}^{\ulT}2\left(\left\|Y_{i\Delta}-Y_{(i-1)\Delta}\right\|_{\cL_2(H)}+\| E(\Delta)X_0\|\right)\\
			\lesssim & \left\|T^{-1}\sum_{i=1}^{\ulT} (Y_{i\Delta}-Y_{(i-1)\Delta})^{\otimes 2}-Q\right\|_{\cL_2(H)}\\
			&+ 2T^{-1} \left(\sum_{i=1}^{\ulT} \| E(\Delta)S((i-1)\Delta) X_0\|^2\right)^{\frac 12}\left(\sum_{i=1}^{\ulT}\left\|Y_{i\Delta}-Y_{(i-1)\Delta}\right\|^2\right)^{\frac 12}\\
			&+T^{-1}\sum_{i=1}^{\ulT} \| E(\Delta)S((i-1)\Delta) X_0\|^2.
		\end{align*}
		Since the first summand on the right of the inequality converges to $0$, we only need to prove that the latter two terms converge to $0$ as $\Delta \to 0$ in $L^2(\Omega)$. For that observe that since $Y$ is independent of $\mathcal F_0$
		\begin{align*}
			& \mathbb E\left[\left(\sum_{i=1}^{\ulT} \| E(\Delta)S((i-1)\Delta) X_0\|^2\right)\left(\sum_{i=1}^{\ulT}\left\|Y_{i\Delta}-Y_{(i-1)\Delta}\right\|^2\right)\right]\\
			%   = & \mathbb E\left[ \| E(\Delta)X_0\|^2\mathbb E\left[\sum_{i=1}^{\ulT}\left\|Y_{i\Delta}-Y_{(i-1)\Delta}\right\|_{\cL_2(H)}^2|\mathcal F_0\right]\right]\\
			= & \mathbb E\left[\sum_{i=1}^{\ulT} \| E(\Delta)S((i-1)\Delta) X_0\|^2\right]\mathbb E\left[\sum_{i=1}^{\ulT}\left\|Y_{i\Delta}-Y_{(i-1)\Delta}\right\|^2\right]\\ 
			\leq & \mathbb E\left[\sum_{i=1}^{\ulT} \| E(\Delta)S((i-1)\Delta) X_0\|^2\right]\mathbb E\left[2\sum_{i=1}^{\ulT}\left\|E(\Delta)Y_{(i-1)\Delta}\right\|^2+\|\tilde \Delta W_I^i\|^2\right], 
		\end{align*}
		where the second factor on the right is bounded by \eqref{eq: bound on adjusted increment} and \eqref{eq: bound on adjusted increment remainder} in Lemma \ref{lem: second moments of increments}. It is therefore left to show 
		\begin{equation}
			\lim_{\Delta \to 0}\mathbb E\left[\left(\sum_{i=1}^{\ulT} \| E(\Delta)S((i-1)\Delta) X_0\|^2\right)^2\right] =0.
		\end{equation}
		For that,
		observe that let $P_N= \sum_{i=1}^{N} e_i^{\otimes 2}$ is the projection onto the first $N$ eigenvalues $e_1,...,e_N$ of $A$. We then have
		\begin{align*}
			&\mathbb E\left[\left(\sum_{i=1}^{\ulT} \| E(\Delta)S((i-1)\Delta) X_0\|^2\right)^2\right]\\
			&\lesssim  \mathbb E\left[\left(\sum_{i=1}^{\ulT} \| E(\Delta)S((i-1)\Delta) P_NX_0\|^2\right)^2\right]\\
			&\qquad+ \mathbb E\left[\left(\sum_{i=1}^{\ulT} \| E(\Delta)S((i-1)\Delta) (I-P_N)X_0\|^2\right)^2\right].
		\end{align*}
		Since $\|AP_N\|_{\cL(H)}<\infty$ and using \eqref{eq:semigroup-time-bound} the first of these summands can be bounded from above by 
		\begin{align*}
			&\mathbb E\left[\left(\sum_{i=1}^{\ulT} \| E(\Delta)S((i-1)\Delta) P_NX_0\|^2\right)^2\right] \\
			&\quad\lesssim (\|E(\Delta) A^{-1}\|^2\ulT)^2 \mathbb E[\|X_0\|^4]\lesssim \Delta^2
		\end{align*}
		converging to $0$ for each fixed $N\in\mathbb N$.
		For the other summand, we find, using \eqref{eq:semigroup-time-bound} and  \eqref{eq:semigroup-analyticity} that 
		\begin{align*}
			&\mathbb E\left[\left(\sum_{i=1}^{\ulT} \| E(\Delta)S((i-1)\Delta) (I-P_N)X_0\|^2\right)^2\right]\\ &\lesssim \left(1+\|E(\Delta) A^{-1}\|^2\sum_{i=2}^{\ulT} \|A S((i-1)\Delta)\|_{\cL(H)}^2\right)^2 \mathbb E[\|(I-P_N)X_0\|^4]\\
			&\lesssim \mathbb E[\|(I-P_N)X_0\|^4]
		\end{align*}
		By dominated convergence, the latter converges to $0$ as $N\to \infty$, uniformly in $n$, so an analogous argument as in the proof of \eqref{eq: bound on adjusted increment remainder}
		yields the claim. 
	\end{proof}
	
	\subsection{Proofs of Appendix \ref{Sec: Pointwise Sampling}}
	\begin{proof}[Proof of Proposition \ref{prop:cylindrical-pointwise}]
		We identify $\dot{H}^r$ and $U$ with their duals and define a cylindrical random variable $Y_t \colon \dot{H}^r \to L^2(\Omega,\R)$ as the linear map given by
		\begin{equation}
			\label{eq:prop:cylindrical-pointwise:pf1}
			Y_t v := \sum_{j = 1}^\infty \langle v, e_j \rangle_{\dot{H}^r} \langle Y_t, e_j \rangle, \quad v \in \dot{H}^r.
		\end{equation}
		
		To see that \eqref{eq:prop:cylindrical-pointwise:pf1} is well-defined, we first recall that since $Q^{1/2}$ is symmetric, the condition $\| Q^{1/2} \|_{\cL(H,\dot{H}^{r-1})} < \infty$ yields that $Q^{1/2}$ extends to an operator in $\cL(\dot{H}^{1-r},H)$, so that $\| Q^{1/2} A^{(r-1)/2}\|_{\cL(H)} < \infty$. Thus, by~\eqref {eq:semigroup-integral}, with $v \in H$,
		\begin{align*}
			\| \Cov(Y_t)^{1/2} v \|^2 = \langle \Cov(Y_t) v, v \rangle &= \int^t_0 \langle S(s) Q S(s) v, v \rangle \dd s \\ &= \int^t_0 \| Q^{\frac 1 2} A^{\frac{r-1}{2}} A^{\frac{1}{2}} S(s) A^{-\frac{r}{2}} v \|^2 \dd s\\
			&\lesssim \int^t_0 \| A^{\frac{1}{2}} S(s) A^{-\frac{r}{2}} v \|^2 \dd s \lesssim \| A^{-\frac{r}{2}} v \|^2.
		\end{align*}
		By density, therefore, $\Cov(Y_t)^{1/2}$ extends to an operator in $\cL(\dot{H}^{-r},H)$, so that $\| \Cov(Y_t)^{1/2} A^{r/2}\|_{\cL(H)} < \infty$. Using this fact, we obtain for arbitrary $N \in \N$ that 
		\begin{align*}
			&\Big\| \sum_{j = 1}^N \langle v, e_j \rangle_{\dot{H}^r} \langle Y_t, e_j \rangle \Big\|^2_{L^2(\Omega,\R)} \\
			&\quad= \sum_{i,j = 1}^N \langle v, e_i \rangle_{\dot{H}^r} \langle v, e_j \rangle_{\dot{H}^r} \langle \Cov(Y_t) e_i, e_j \rangle \\
			&\quad= \sum_{i,j = 1}^N \langle A^{\frac r 2} v, e_i \rangle \langle A^{\frac r 2} v, e_j \rangle \langle \Cov(Y_t)^{\frac 1 2} A^{\frac r 2} e_i, \Cov(Y_t)^{\frac 1 2} A^{\frac r 2} e_j \rangle \\
			&\quad= \| \Cov(Y_t)^{\frac 1 2} A^{\frac r 2} P_N A^{\frac r 2} v \|^2 \lesssim \| P_N A^{\frac r 2} v \|^2 \lesssim \|v \|_{\dot{H}^r}, , \quad v \in \dot{H}^r, 
		\end{align*}
		where $P_N$ denotes the orthogonal projection onto the span of $\{e_1, \ldots, e_N\}$. This shows that that~\eqref{eq:prop:cylindrical-pointwise:pf1} is a well-defined cylindrical random variable.
		
		To see that the last claim of the proposition holds true, we first note that by definition of linear operators acting on cylindrical random variables, $\Gamma Y_t \colon \dot{H}^r \to L^2(\Omega,U)$ is given by 
		\begin{equation*}
			\Gamma Y_t u = \sum_{j = 1}^\infty \langle \Gamma^* u, e_j \rangle_{\dot{H}^r} \langle Y_t, e_j \rangle, \quad u \in U,
		\end{equation*}
		and we must show that 
		\begin{equation*}
			\Gamma Y_t u = \Big\langle \int^t_0 \Gamma S(t-s) \dd W(s), u  \Big\rangle_{U}, \quad \Prob\text{-a.s.}
		\end{equation*}
		The $U$-valued stochastic integral is well-defined, since by~\eqref{eq:semigroup-hs},
		\begin{align*}
			\int^t_0 \| \Gamma S(s) Q^{\frac 1 2} \|^2_{\cL_2(H,U)} \dd s &= \int^t_0 \| Q^{\frac 1 2} A^{\frac{r-1}{2}} A^{\frac 1 2} S(s) (\Gamma A^{-\frac{r}{2}})^* \|^2_{\cL_2(U,H)} \dd s \\
			&\lesssim \int^t_0 \| A^{\frac 1 2} S(s) (\Gamma A^{-\frac{r}{2}})^* \|^2_{\cL_2(U,H)} \dd s \\
			&\lesssim \| (\Gamma A^{-\frac{r}{2}})^* \|^2_{\cL_2(U,H)} = \|\Gamma \|^2_{\cL_2(\dot{H}^{r},U)}.
		\end{align*}
		By the same arguments, It\^o isometry and the dominated convergence theorem, 
		\begin{equation*}
			\Big\| \int^t_0 \Gamma S(t-s) (I-P_N) \dd W(s) \Big\|^2_{L^2(\Omega,U)} \lesssim \| (I-P_N) (\Gamma A^{\frac{-r}{2}})^* \|^2_{\cL_2(U,H)} \to 0
		\end{equation*}
		as $N \to 0$. Since also
		\begin{align*}
			\sum_{j = 1}^N \langle \Gamma^* u, e_j \rangle_{\dot{H}^r} \langle Y_t, e_j \rangle &= \int^t_0 \sum_{j = 1}^N \langle \Gamma^* u, e_j \rangle_{\dot{H}^r} \langle S(t-s) \cdot, e_j \rangle \dd W(s) \\
			&= \Big\langle \int^t_0 \Gamma S(t-s) P_N \dd W(s), u \Big\rangle_U,
		\end{align*}	
		the claim follows from the Cauchy--Schwarz inequality.	\qedhere			
	\end{proof}		
	
	\subsection{Proofs of Section \ref{Sec: Sampling setting}}\label{Sec: Proofs of Section 3}
	\begin{proof}[Proof of Lemma \ref{lem:obs-hs}]
		For $s>d/2$, the Sobolev space $H^s$ is a reproducing kernel Hilbert space. Let its kernel be $k_s$, which is continuous and bounded on $\bar{\cD} \times \bar{\cD}$ by a constant $C_s$. For any orthonormal basis $(f_j)_{j=1}^\infty$ of $H^s$, we may write 
		\begin{equation*}
			k_s(x,y) = \sum_{j = 1}^\infty f_j(x) f_j(y), \quad x,y \in \cD.
		\end{equation*}
		Using this and the definition of $I_h$, it follows that
		\begin{align*}
			\| I_h \|_{\cL_2(H^s,H)}^2 &= \sum_{j=1}^{\infty} \| I_h f_j \|^2_{H} = \sum_{j=1}^{\infty} \int_{\cD} \left( \sum_{x_k \in \ver(\cT_h)} f_j(x_k) \phi^k_h(x) \right)^2 \dd x \\
			&= \sum_{x_k, x_\ell \in \ver(\cT_h)} \left( \sum_{j=1}^\infty f_j(x_k) f_j(x_\ell) \right) \int_{\cD} \phi^k_h(x) \phi^\ell_h(x) \dd x \\
			&= \sum_{x_k, x_\ell \in \ver(\cT_h)} k_s(x_k, x_\ell) \int_{\cD} \phi^k_h(x) \phi^\ell_h(x) \dd x.
		\end{align*}
		Using the boundedness of the kernel, the non-negativity of $\phi^k_h$ and the partition of unity property $\sum_{x_k} \phi^k_h(x) = 1$, we have 
		\begin{align*}
			\| I_h \|_{\cL_2(H^s,H)}^2 \le C_s \sum_{x_k, x_\ell} \int_{\cD} \phi^k_h(x) \phi^\ell_h(x) \dd x = C_s |\cD|.
		\end{align*}
		
		The calculation above establishes that for any $s > d/2$, 
		\begin{equation} \label{eq:Ih_HS_H_revised}
			\sup_{h \in (0,1]} \| I_h \|_{\cL_2(H^s,H)} < \infty.
		\end{equation}		
		Moreover, by \cite[Proposition 4.1]{LordPetersson25}, for any $s > d/2$ and $r_0 \in (0, \min(3/2,s))$,
		\begin{equation} \label{eq:Ih_bounded_Hs_H_r0}
			\sup_{h \in (0,1]} \| I_h \|_{\cL(H^s, H^{r_0})} < \infty.
		\end{equation}
		
		Now, let $s > d/2$ be fixed and choose any $r_0 \in (0, \min(3/2,s))$. We apply interpolation of Schatten class operators (see \cite{Gapaillard1974}) to the bounds \eqref{eq:Ih_HS_H_revised} and \eqref{eq:Ih_bounded_Hs_H_r0} on $I_h$.
		For any $\theta \in (0,1)$, this produces an operator in the Schatten class $\cL_p$ with $1/p = (1-\theta)/\infty + \theta/2 = \theta/2$. The target space is the interpolation space $[H^{r_0}, H]_\theta = H^{r_0(1-\theta)}$. Thus, for any $\theta \in (0,1)$, $I_h \in \cL_{2/\theta}(H^s, H^{r_0(1-\theta)})$.
		
		Finally, let $r > d/2$ be given. We write $I_h$ as an operator from $H^r$ to $H^\epsilon$ as the composition $I_h \circ I_{H^r \hookrightarrow H^s}$. For this composition to be in $\cL_2$, the Hölder property for Schatten norms requires $I_{H^r \hookrightarrow H^s} \in \cL_q(H^r, H^s)$ where $1/2 = 1/(2/\theta) + 1/q$, which gives $1/q = 1/2 - \theta/2 = (1-\theta)/2$, or $q = 2/(1-\theta)$.
		
		The embedding $I_{H^r \hookrightarrow H^s}: H^r \to H^s$ is in $\cL_q(H^r, H^s)$ provided that $r-s > d/q$, which translates to the condition $r-s > {d(1-\theta)}/{2}$.
		Given $r > d/2$, we can always find parameters that satisfy this. First, choose any $s$ such that $d/2 < s < r$. Then choose $\theta \in (0,1)$ so that $\theta > 1 - {2(r-s)}/{d}$. With such a choice of $s$ and $\theta$, and our fixed $r_0$, the composition $I_h \circ I_{H^r \hookrightarrow H^s}$ is in $\cL_2(H^r, H^{r_0(1-\theta)})$. We set $\epsilon = r_0(1-\theta)$. Since $r_0 > 0$ and $\theta < 1$, we have found an $\epsilon > 0$ for which the statement is valid. This completes the proof.
	\end{proof}
	
	\begin{proof}[Proof of Lemma \ref{lem:obs-error}]	
		The result for $O=P_h$ is noted in \cite{BelgacemBrenner01}, the authors of which only consider triangular subdivisions. The proof is the same in our setting: a scaling argument combined with Friedrich's inequality (see \cite[Lemma~4.3.14]{BrennerScott08}).
		The result for $O=I_h$ is shown in \cite[Example~3]{DupontScott80} for $d=2$. The proof for general $d$ is essentially the same, replacing the infimum over linear functions in the application of the Bramble-Hilbert lemma with the infimum over constant functions when $r$ is small and $d=1$.
	\end{proof}
	
	\subsection{Proofs of Section \ref{Sec: Asymptotic Theory}}
	\subsection{Key discretization result }
	We first proof a key result that we will employ for the proofs of Theorem \ref{thm:RV-convergence} and \ref{thm: CLT}. For that, 
	we introduce the auxiliary estimator
	\begin{align}\label{eq: SARCV}
		SARCV_t^{\Delta,O}:=\sum_{i=1}^{\ul} \left(OX_{i\Delta}-O \mathcal S(\Delta)X_{(i-1)\Delta}\right)^{\otimes 2}, \quad t \ge 0,
	\end{align}
	which was examined for the case $O=I$  in \cite{Benth2022} and \cite{BSV2022}.  
	To make use of the asymptotic theory introduced in these two articles, we can establish an asymptotic equivalence between $SARCV^{\Delta,O}$ and $RV^{\Delta,O}$. For that, we introduce
	\begin{lemma}\label{lem: RV-SARCV asympt equiv}
		Let $X_0\equiv 0$ and Assumption \ref{ass: minimal Assumption on A}  be valid.
		If $O\neq I_h$,
		\begin{align*}
			\|(RV^{\Delta,O}_{T}-\mathbb E[RV^{\Delta,O}_T])-(SARCV^{\Delta,O}_{T}-\mathbb E[SARCV^{\Delta,O}_{T}])\|_{L^2(\Omega,\cL_2(H))}= o(T^{\frac 12} \Delta^{\frac 12}),
		\end{align*}
		whereas if $O= I_h$ and Assumption \ref{ass:Ih} holds, 
		\begin{equation}
			\label{eq: SARCV-RV equivalence}
			\begin{split}
				&\|(RV^{\Delta,O}_{T}-\mathbb E[RV^{\Delta,O}_{T}])-(SARCV^{\Delta,O}_{T}-\mathbb E[SARCV^{\Delta,O}_{T}])\|_{L^2(\Omega,\cL_2(H))} \\
				&\quad= o\left(T^{\frac 12} \Delta^{\frac {\min(1,\beta)}2}\right).
			\end{split}
		\end{equation}
	\end{lemma}
	
	\begin{proof}
		
		We assume without loss of generality that $\gamma \le 1$.
		First, observe that
		
		\begin{equation*}
			\begin{split}
				&  \|(RV^{\Delta,O}_{T}-\mathbb E[RV^{\Delta,O}_{T}])-(SARCV^{\Delta,O}_{T}-\mathbb E[SARCV^{\Delta,O}_{T}])\|_{L^2(\Omega,\cL_2(H))}^2\\
				&=\Big\| \sum_{i = 1}^{\ulT} (O X_{i\Delta} - O X_{(i-1)\Delta})^{\otimes 2} -(\tilde \Delta W^i_O)^{\otimes 2} \\ 
				&\hspace{5em}- \E[(O X_{i\Delta} - O X_{(i-1)\Delta})^{\otimes 2}-(\tilde \Delta W^i_O)^{\otimes 2}] \Big\|_{L^2(\Omega,\cL_2(H))}^2 \\
				&= \sum_{i,j=1}^{\ulT} \big(\E[\langle (O X_{i\Delta} - O X_{(i-1)\Delta})^{\otimes 2}-(\tilde \Delta W^i_O)^{\otimes 2}, \\
				&\hspace{5em}
				(O X_{j\Delta} - O X_{(j-1)\Delta})^{\otimes 2}-(\tilde \Delta W^j_O)^{\otimes 2} \rangle_{\cL_2(H)}] \\
				&- \langle \E[(O X_{i\Delta} - O X_{(i-1)\Delta})^{\otimes 2}-(\tilde \Delta W^i_O)^{\otimes 2}], \\
				&\hspace{5em}\E[(O X_{j\Delta} - O X_{(j-1)\Delta})^{\otimes 2}] -(\tilde \Delta W^j_O)^{\otimes 2}\rangle_{\cL_2(H)}\big).
			\end{split}
		\end{equation*}
		We treat separately the cases $i > j$ (by symmetry also $j > i$) and $i=j$. 
		First we consider the case that $i > j$. By stationarity of the terms $\tilde \Delta W^k_O$ and their independence of $\cF_{(k-1)\Delta}$, $k=1, \ldots, \ulT$, we find
		\begin{align*}
			&\E[\langle (O X_{i\Delta} - O X_{(i-1)\Delta})^{\otimes 2} -(\tilde \Delta W^i_O)^{\otimes 2},(O X_{j\Delta} - O X_{(j-1)\Delta})^{\otimes 2} -(\tilde \Delta W^j_O)^{\otimes 2} \rangle] \\
			&\quad=  2 \E[\langle \tilde \Delta W^j_O,E^{i-1,0}_O \rangle \langle E^{i-1,0}_O,E^{j-1,0}_O \rangle] 
			+ \E[\langle E^{i-1,0}_O,E^{j-1,0}_O \rangle^2].
		\end{align*}
		Using the  identity $\E[ E^{i-1,0}_O \otimes E^{j,j-1}_{i-j,O}] = \E[(E^{j,j-1}_{i-j,O})^{\otimes 2}]$ we find that 
		\begin{align*}
			\E[\langle \tilde \Delta W^j_O,E^{i-1,0}_O \rangle \langle E^{i-1,0}_O,E^{j-1,0}_O \rangle] &= \langle \E[\tilde \Delta W^j_O \otimes E^{j,j-1}_{i-j,O}] , \E[E^{j-1,0}_{i-j,O} \otimes E^{j-1,0}_{O}] \rangle_{\cL_2(H)} \\
			&\quad+ \E[\langle \tilde \Delta W^j_O, E^{j,j-1}_{i-j,O} \rangle] \E[\langle E^{j-1,0}_{i-j,O},E^{j-1,0}_{O} \rangle] \\
			&\le \langle \E[\tilde \Delta W^j_O \otimes E^{j,j-1}_{i-j,O}] , \E[E^{j-1,0}_{i-j,O} \otimes E^{j-1,0}_{O}] \rangle_{\cL_2(H)} \\
			&\quad+ \frac{\E[\langle \tilde \Delta W^j_O,E^{j,j-1}_{i-j,O} \rangle^2] + \E[\langle E^{j-1,0}_{i-j,O},E^{j-1,0}_O \rangle^2]}{2}.
		\end{align*}
		Moreover, we have
		\begin{align*}
			\E[\langle E^{i-1,0}_O,E^{j-1,0}_O \rangle^2] &= \E[(\langle E^{j-1,0}_{i-j,O},E^{j-1,0}_O \rangle + \langle E^{i-1,j-1}_O,E^{j-1,0}_O \rangle)^2] \\
			&= \E[\langle E^{j-1,0}_{i-j,O},E^{j-1,0}_O \rangle^2] \\
			&\quad+ \langle \E[(E^{i-1,j-1}_O)^{\otimes 2}] , \E[(E^{j-1,0}_O)^{\otimes 2}] \rangle_{\cL_2(H)}.
		\end{align*}
		Let us also note that since
		\begin{align*}
			&\E[(E_O^{i-1,0})^{\otimes 2}] - \E[(E_O^{i-1,j-1})^{\otimes 2}] \\
			&\quad= \int^{(i-1)\Delta}_{0} O S(r) Q (O S(r))^* \dd r - \int^{(i-1)\Delta}_{(j-1)\Delta} O S(r) Q (O S(r))^* \dd r \\
			&\quad= \int^{(j-1)\Delta}_{0} O S(r) Q (O S(r))^* \dd r = \E[(E_O^{j-1,0})^{\otimes 2}], 
		\end{align*}
		we obtain the inequality
		\begin{align*}
			\langle \E[(E^{i-1,j-1}_O)^{\otimes 2}] , \E[(E^{j-1,0}_O)^{\otimes 2}] \rangle_{\cL_2(H)} -\langle \E[(E^{i-1}_O)^{\otimes 2}] , \E[(E^{j-1,0}_O)^{\otimes 2}] \rangle_{\cL_2(H)} \le 0.
		\end{align*}
		Further, we obtain
		\begin{align*}
			&\langle \E[(O X_{t_i} - O X_{t_{i-1}})^{\otimes 2}-(\tilde \Delta W_O^i)^{\otimes 2}],\E[(O X_{t_j} - O X_{t_{j-1}})^{\otimes 2}-(\tilde \Delta W_O^j)^{\otimes 2}] \rangle_{\cL_2(H)} \\
			&\quad=  \langle \E[(E^{i-1,0}_{O})^{\otimes 2}] , \E[(E^{j-1,0}_{O})^{\otimes 2}] \rangle_{\cL_2(H)}.
		\end{align*}
		
		Summing up, we obtain 
		for $i > j$,
		\begin{align*}
			&\big(\E[\langle (O X_{t_i} - O X_{t_{i-1}})^{\otimes 2}-(\tilde \Delta W^i_O)^{\otimes 2},(O X_{t_j} - O X_{t_{j-1}})^{\otimes 2}-(\tilde \Delta W^j_O)^{\otimes 2} \rangle_{\cL_2(H)}] \\
			&- \langle \E[(O X_{t_i} - O X_{t_{i-1}})^{\otimes 2}-(\tilde \Delta W^i_O)^{\otimes 2}],\E[(O X_{t_j} - O X_{t_{j-1}})^{\otimes 2} -(\tilde \Delta W^j_O)^{\otimes 2}]\rangle_{\cL_2(H)}\big)\\
			&\quad\le 2\Big( \E[\langle \tilde \Delta W^j_O,E^{j,j-1}_{i-j,O} \rangle^2] + \langle \E[\tilde \Delta W^j_O \otimes E^{j,j-1}_{i-j,O}] , \E[E^{j-1,0}_{i-j,O} \otimes E^{j-1,0}_{O}] \rangle_{\cL_2(H)} \\
			&\hspace{4em}+ \E[\langle E^{j-1,0}_{i-j,O},E^{j-1,0}_O \rangle^2] \Big)  \\
			&\quad=: 2 \big(\mathrm{I}_{j}^i + \mathrm{II}_{j}^i + \mathrm{III}_{j}^i\big) .
		\end{align*}
		We bound the first and third terms by the Cauchy--Schwarz inequality and use the fact that there is a constant $C< \infty$ such that for any Gaussian $H$-valued random variable $Z$, $\|Z\|_{L^4(\Omega,H)} \le C \|Z\|_{L^2(\Omega,H)}$ \cite[Proposition~4.16]{Hairer09}. This then yields 
		\begin{equation*}
			\mathrm{I}_{j}^i \lesssim \|\tilde \Delta W^1_O\|_{L^2(\Omega,H)}^2 \|E^{j,j-1}_{i-j,O}\|_{L^2(\Omega,H)}^2 \text{ and } \mathrm{III}_{j}^i \lesssim \|E^{j-1,0}_{O}\|_{L^2(\Omega,H)}^2 \|E^{j,j-1}_{i-j,O}\|_{L^2(\Omega,H)}^2.
		\end{equation*}
		By definition of the Hilbert--Schmidt norm and the Cauchy--Schwarz inequality, 
		\begin{align*}
			\mathrm{II}_{j}^i &\le \E[\|\tilde \Delta W^j_O \otimes E^{j,j-1}_{i-j,O}\|_{\cL_2(H)}] \E[\|E^{j-1,0}_{i-j,O} \otimes E^{j-1,0}_{O}\|_{\cL_2(H)}] \\
			&\le \|\tilde \Delta W^1_O\|_{L^2(\Omega,H)} \|E^{j,j-1}_{i-j,O}\|_{L^2(\Omega,H)} \|E^{j-1,0}_{i-j,O}\|_{L^2(\Omega,H)} \|E^{j-1,0}_{O}\|_{L^2(\Omega,H)}.
		\end{align*}
		In the case that $i=j$, we 
		have by the Cauchy-Schwartz inequality
		\begin{align*}
			&\E[\| (O Y_{t_j}-O Y_{t_{j-1}})^{\otimes 2}-(\tilde \Delta W_O^i)^{\otimes 2} \|^2] -\| \E[(O Y_{t_j}-O Y_{t_{j-1}})^{\otimes 2}-(\tilde \Delta W_O^i)^{\otimes 2} ]\|^2 \\
			& \lesssim  \|\tilde \Delta W^1_O\|_{L^2(\Omega,H)}^2 \|E^{j-1,0}_{O}\|_{L^2(\Omega,H)}^2+ \|E^{j-1,0}_{O}\|_{L^2(\Omega,H)}^4.
		\end{align*}
		As an intermediate estimate we subsume
		\begin{align}\label{eq: intermediate estimate for key discretization result}
			&	\| \left(\mathrm{RV}^{\Delta,O}_{T} - \E[\mathrm{RV}^{\Delta,O}_{T}] \right)-\left(\mathrm{SARCV}^{\Delta,O}_{T} - \E[\mathrm{SARCV}^{\Delta,O}_{T}] \right)\|_{L^2(\Omega,\cL_2(H))}^2\notag\\
			&\lesssim \sum_{j = 1}^{\ulT}  \|\tilde \Delta W^1_O\|_{L^2(\Omega,H)}^2 \|E^{j-1,0}_{O}\|_{L^2(\Omega,H)}^2+ \|E^{j-1,0}_{O}\|_{L^2(\Omega,H)}^4 \\
			&\quad+ \sum^{\ulT-1}_{j=1} \sum^{\ulT}_{i = j+1}  \|\tilde \Delta W^1_O\|_{L^2(\Omega,H)}^2 \|E^{j,j-1}_{i-j,O}\|_{L^2(\Omega,H)}^2 \notag\\
			&\hspace{4em}+  \|E^{j-1,0}_{O}\|_{L^2(\Omega,H)}^2 \|E^{j,j-1}_{i-j,O}\|_{L^2(\Omega,H)}^2\notag\\
			&\hspace{4em}+  \|\tilde \Delta W^1_O\|_{L^2(\Omega,H)} \|E^{j,j-1}_{i-j,O}\|_{L^2(\Omega,H)} \|E^{j-1,0}_{i-j,O}\|_{L^2(\Omega,H)} \|E^{j-1,0}_{O}\|_{L^2(\Omega,H)}. \notag
		\end{align}
		
		We now first prove the result in the case $O=I$. In this case we use Lemma \ref{lem: second moments of increments} and the Cauchy-Schwartz inequality  to bound the components in \eqref{eq: intermediate estimate for key discretization result} to obtain
		\begin{align*}
			&	\| \left(\mathrm{RV}^{\Delta,O}_{T} - \E[\mathrm{RV}^{\Delta,O}_{T}] \right)-\left(\mathrm{SARCV}^{\Delta,O}_{T} - \E[\mathrm{SARCV}^{\Delta,O}_{T}] \right)\|_{L^2(\Omega,\cL_2(H))}^2\\
			&\lesssim \sum_{j = 1}^{\ulT} \Delta \|E^{j-1,0}_{O}\|_{L^2(\Omega,H)}^2+ \Delta^2 \zeta_{\Delta} \\
			&\quad+ \sum^{\ulT-1}_{j=1} \sum^{\ulT}_{i = j+1}  \Delta \|E^{j,j-1}_{i-j,O}\|_{L^2(\Omega,H)}^2 +\Delta \zeta_{\Delta} \|E^{j,j-1}_{i-j,O}\|_{L^2(\Omega,H)}^2\\
			&\hspace{4em}+ \Delta \zeta_{\Delta}  \|E^{j,j-1}_{i-j,O}\|_{L^2(\Omega,H)} \|E^{j-1,0}_{i-j,O}\|_{L^2(\Omega,H)}\\
			&\lesssim  \Delta T  \zeta_{\Delta}+ T\Delta \zeta_{\Delta} \\
			&\quad+ \Delta \left(\sum^{\ulT-1}_{j=1} \sum^{\ulT}_{i = j+1}   \|E^{j,j-1}_{i-j,O}\|_{L^2(\Omega,H)}^2\right) \\
			&\quad+\Delta \zeta_{\Delta} \left(\sum^{\ulT-1}_{j=1} \sum^{\ulT}_{i = j+1}   \|E^{j,j-1}_{i-j,O}\|_{L^2(\Omega,H)}^2\right)\\
			&\hspace{1em}+ \Delta \zeta_{\Delta}  \left(\sum^{\ulT-1}_{j=1} \sum^{\ulT}_{i = j+1}   \|E^{j,j-1}_{i-j,O}\|_{L^2(\Omega,H)}^2\right)^{\frac 12}\left(\sum^{\ulT-1}_{j=1} \sum^{\ulT}_{i = j+1}   \|E^{j-1,0}_{i-j,O}\|_{L^2(\Omega,H)}^2\right)^{\frac 12} \\
			&\lesssim   \Delta T  \zeta_{\Delta}.
		\end{align*}
		This proves the claim for the case $O= I$. We then immediately obtain the claim for the case that $O=P_h$, since this operator is uniformly bounded in $\cL(H)$ and therefore 
		\begin{align*}
			&	\| \left(\mathrm{RV}^{\Delta,P_h}_{T} - \E[\mathrm{RV}^{\Delta,P_h}_{T}] \right)-\left(\mathrm{SARCV}^{\Delta,P_h}_{T} - \E[\mathrm{SARCV}^{\Delta,P_h}_{T}] \right)\|_{L^2(\Omega,\cL_2(H))}^2\\
			\leq  &	\| \left(\mathrm{RV}^{\Delta,I}_{T} - \E[\mathrm{RV}^{\Delta,I}_{T}] \right)-\left(\mathrm{SARCV}^{\Delta,I}_{T} - \E[\mathrm{SARCV}^{\Delta,I}_{T}] \right)\|_{L^2(\Omega,\cL_2(H))}^2.
		\end{align*}
		It now remains to prove the result for $O=I_h$.
		Using the estimates from Lemma \ref{lem: second moments of increments} in the case $O=I_h$ under Assumption~\ref{ass:Ih}, we find
		\begin{align*}
			&	\| \left(\mathrm{RV}^{\Delta,I_h}_{T} - \E[\mathrm{RV}^{\Delta,O}_{T}] \right)-\left(\mathrm{SARCV}^{\Delta,I_h}_{T} - \E[\mathrm{SARCV}^{\Delta,I_h}_{T}] \right)\|_{L^2(\Omega,\cL_2(H))}^2\\
			&\lesssim T\Delta^{\min(1,\beta)} \zeta_{\Delta}.
		\end{align*}
		The proof is complete.
	\end{proof}
	
	\subsubsection{Proof of Theorem~\ref{thm:RV-convergence}}
	Theorem~\ref{thm:RV-convergence}, is a consequence of several lemmata. 
	We start with three lemmata on the stochastic convolution $Y$.
	
	First, we obtain an estimate for the variance of the realized covariation estimator.
	\begin{lemma}
		\label{lem:var-conv}
		Let Assumption \ref{ass: minimal Assumption on A} hold and $X_0 \equiv 0$. If $O \neq I_h$ there is a $C < \infty$ independent of $T$ and $\Delta$, such that 
		\begin{equation*}
			\Var[RV_T^{\Delta,O}]\le C T \Delta.
		\end{equation*}
		If instead $O = I_h$ and Assumption~\ref{ass:Ih} is fulfilled, the bound $C T \Delta$ in the statement is replaced by $C T\Delta^{\min(\beta,1)}$.
	\end{lemma}
	
	\begin{proof}
		We start with the case that $O\neq I_h$.    By Lemma \ref{lem: RV-SARCV asympt equiv} and  the triangle inequality, we only need to prove that
		$\mathbb E[\|SARCV_{T}^{\Delta,O}-\mathbb E[SARCV_{T}^{\Delta,O}]\|^2]\leq C T \Delta$.  For that, observe that $O(X_{i\Delta}-S(\Delta) X_{(i-1)\Delta}), i=1,...,\ulT$ are i.i.d. and we obtain
		\begin{align*}
			& \mathbb E[\|SARCV_{T}^{\Delta,O}-\mathbb E[SARCV_{T}^{\Delta,O}]\|^2]\\
			= &\sum_{i=1}^{\ulT} \mathbb E\left[\left\|(OX_{i\Delta}-OS(\Delta)X_{(i-1)\Delta})^{\otimes 2}-\mathbb E\left[(OX_{i\Delta}-OS(\Delta)X_{(i-1)\Delta})^{\otimes 2}\right]\right\|^2\right]\\
			\leq &2\sum_{i=1}^{\ulT} \mathbb E\left[\left\|(OX_{i\Delta}-OS(\Delta)X_{(i-1)\Delta})\right\|^4\right].
		\end{align*}
		We use the fact that there is a constant $C< \infty$ such that for any Gaussian $H$-valued random variable $Z\sim N(0,C)$, $\|Z\|_{L^4(\Omega,H)}\le 3^{\frac 14}\|Z\|_{L^2(\Omega,H)}$, such that
		\begin{align*}
			& \mathbb E[\|SARCV_{T}^{\Delta,O}-\mathbb E[SARCV_{T}^{\Delta,O}]\|^2]\\
			\lesssim &\sum_{i=1}^{\ulT} \mathbb E\left[\left\|(OX_{i\Delta}-OS(\Delta)X_{(i-1)\Delta})\right\|^2\right]^2\\
			= &\sum_{i=1}^{\ulT} \left(\int_{0}^{\Delta} \| O S(r) Q^{\frac 1 2} \|_{\cL_2(H)}^2 \dd r \right)^2.\\
			%\leq &  T\Delta tr(Q)^2
		\end{align*}
		If $I\neq I_h$ we can use the fact that $O\in \cL (H)$ to obtain that  $\mathbb E[\|SARCV_{T}^{\Delta,O}-\mathbb E[SARCV_{T}^{\Delta,O}]\|^2]\lesssim T \Delta$. For the case $O=I_h$ we find
		\begin{align*}
			\int_{0}^{\Delta} \| O S(r) Q^{\frac 1 2} \|_{\cL_2(H)}^2 \dd r
			&= \int_{0}^{\Delta} \| O A^{-\frac{d+\epsilon}{4}} A^{\frac{1-\beta}{4}} S(r) A^{\frac{\beta + d + \epsilon - 1}{4}} Q^{\frac 1 2} \|_{\cL_2(H)}^2 \dd r \\
			&\lesssim \int_{0}^{\Delta} \| O A^{-\frac{d+\epsilon}{4}} A^{\frac{1-\beta}{4}} S(r) \|_{\cL_2(H)}^2 \dd r \\
			&= \sum_{i, j = 1}^\infty \int_{0}^{\Delta} \langle A^{\frac{1-\beta}{4}}S(r) e_i, (O A^{-\frac{d+\epsilon}{4}})^*e_j \rangle^2 \dd r \\
			&\lesssim \int_{0}^{\Delta} \|A^{\frac{1-\beta}{4}} S(r) (O A^{-\frac{d+\epsilon}{4}})^* \|_{\cL_2(H)}^2 \dd r \lesssim \Delta^{\min(1,(1+\beta)/2)}.
		\end{align*}
		In the last step, we again made use of Lemma~\ref{lem:obs-hs}. This proves the claim
	\end{proof}
	
	For the proof of the next lemma, we recall that for any bounded operators $\Gamma_0,\Gamma_1$ and $\Gamma_2$ on Hilbert spaces,
	\begin{equation}
		\label{eq:operator-product-difference}
		\Gamma_1 \Gamma_0 \Gamma_1^* - \Gamma_2 \Gamma_0 \Gamma_2^* = \frac{(\Gamma_1 + \Gamma_2) \Gamma_0 (\Gamma_1 -
			\Gamma_2)^* + (\Gamma_1 - \Gamma_2) \Gamma_0 (\Gamma_1 +
			\Gamma_2)^*}{2}.
	\end{equation}
	\begin{lemma}
		\label{lem:time-conv}
		Suppose that Assumption~\ref{ass:rate} is satisfied. Then there is a $C<\infty$ and a real sequence $\zeta_{\Delta}$ converging to $0$ as $\Delta \to 0$ which does not depend on $T$ or $\Delta$ such that
		\begin{enumerate}[label=(\roman*)]
			\item \label{lem:time-conv:1}$\| \sum_{i = 1}^{\ulT} \int^{i\Delta}_{(i-1)\Delta} S(i\Delta-r) Q S(i\Delta-r) - Q \dd r \|_{\cL_2(H)} \le C T \Delta^{\frac{\min(\gamma,2)}{2}}\zeta_{\Delta}$ and 
			\item \label{lem:time-conv:2} $\| \sum_{i = 2}^{\ulT} \E[(E(\Delta)Y_{t_{i-1}})^{\otimes 2}] \|_{\cL_2(H)} \le C T \Delta^{\frac{\min(\gamma,2)}{2}}\zeta_{\Delta}$. 
		\end{enumerate}
	\end{lemma}
	
	\begin{proof}[Proof of Lemma~\ref{lem:time-conv}]
		Let $P_N$ denote the projection onto the first $N$ eigenfunctions $e_1,...,e_N$ of $A$. We will use that by Assumption \ref{ass:rate} we have that
		\begin{align}\label{aux-eq 1: Proof of remainder convergence with perfect observations}
			\|(I-P_N) A^{\frac{\gamma}2}Q\|_{\cL_2(H)}^2=\sum_{i=N+1}^{\infty} \|QA^{\frac{\gamma}2}  e_i\|^2 \longrightarrow 0 \quad \text{ as } N\to \infty
		\end{align}
		and for $s\in (0,1), N\in \mathbb N$,
		\begin{align}\label{aux-eq 2: Proof of remainder convergence with perfect observations}
			\Delta^{-\frac 12}\sup_{s\leq \Delta}\|P_N E(s)A^{-\frac{\gamma}2}\|_{\cL(H)}^2\leq  \Delta^{-\frac 12}N\sup_{s\leq \Delta}\sup_{i=1,...,N} (e^{-\lambda_i s}-1) \lambda_i^{\frac{\gamma}2}\longrightarrow 0 \, \text{ as }\Delta 
			\to 0 .
			%\infty
		\end{align} 
		To show the first claim, note first that by~\eqref{eq:operator-product-difference}, the triangle inequality, H\"older's inequality and  \eqref{eq:semigroup-analyticity} with $r=0$,
		\begin{align*}
			&\ulT^{2}\Big\| \int^{\Delta}_0 S(r) Q S(r) - Q \dd r \Big\|^2_{\cL_2(H)} \\ &\lesssim \ulT^{2}\Bigg(\Big\|P_N \int^{\Delta}_0 E(r) Q (S(r) + I) \dd r \Big\|^2_{\cL_2(H)}\\
			&\qquad\qquad\qquad\qquad+\Big\|(I-P_N) \int^{\Delta}_0 E(r) Q (S(r) + I) \dd r \Big\|^2_{\cL_2(H)}\Bigg)  \\
			&\le T^2\Delta^{-1}\Bigg(\int^{\Delta}_0 \| P_N E(r) Q (S(r) + I) \|^2_{\cL_2(H)} \dd r \\
			&\qquad\qquad\qquad\qquad+\int^{\Delta}_0 \Big\|(I-P_N) E(r) Q (S(r) + I) \Big\|^2_{\cL_2(H)}\dd r \Bigg)  \\
			&\lesssim T^2\Delta^{-1} \|A^{\frac \gamma 2} Q \|^2_{\cL_2(H)} \int^{\Delta}_0 \| P_N E(r) A^{-\frac{\gamma}{2}} \|^2_{\cL(H)}  \dd r\\
			&\qquad+T^2\Delta^{-1} \|(I-P_N)A^{\frac \gamma 2} Q \|^2_{\cL_2(H)} \int^{\Delta}_0 \| E(r) A^{-\frac{\gamma}{2}} \|^2_{\cL(H)}  \dd r  \\
			= & \mathrm{I}_{\Delta,N}+\mathrm{II}_{\Delta,N}.
		\end{align*}
		Using \eqref{eq:semigroup-time-bound} and \eqref{aux-eq 1: Proof of remainder convergence with perfect observations}, we find that the second summand converges to $0$ as $N\to \infty$ uniformly in $\Delta$ and by \eqref{aux-eq 2: Proof of remainder convergence with perfect observations} we obtain that the first summand converges to $0$ as $\Delta\to 0$  for each $N\in \mathbb N$. 
		This shows that for fixed $T>0$ and all positive $\epsilon>0$ we  can find an $N_{\epsilon}>0$ (independent of $\Delta$) such that $\sup_{\Delta >0}\mathrm{II}_{\Delta,N_{\epsilon}}\leq  \epsilon/2 $ and an $\Delta_{\epsilon}>0$
		such that  $\mathrm{I}_{\Delta_{\epsilon},N_{\epsilon}}\leq \epsilon/2$. Hence, we find that
		\begin{align*}
			\lfloor T/\Delta_{\epsilon}\rfloor\Big\| \int^{\Delta_{\epsilon}}_0 S(r) Q S(r) - Q \dd r \Big\|^2_{\cL_2(H)}\lesssim  \sup_{\Delta>0} \mathrm{I}_{\Delta_{\epsilon},N_{\epsilon}}+\mathrm{II}_{\Delta_{\epsilon},N_{\epsilon}}\leq \epsilon.
		\end{align*}
		This proves (i).
		
		For the claim~\ref{lem:time-conv:2},we first expand $\E[(E(\Delta)Y_{(i-1)\Delta})^{\otimes 2}]$ on the eigenbasis of $A$ to see that for $k \in \N$
		\begin{align*}
			&\| \E[(E(\Delta)Y_{(i-1)\Delta})^{\otimes 2}] e_k \|^2 \\
			&\quad= \Big\| \int^{(i-1)\Delta}_0 E(\Delta) S(r) Q S(r) E(\Delta) e_k \dd r \Big\|^2 \\ 
			&\quad= \sum_{j = 1}^{\infty} (1-e^{-\lambda_j (\Delta)})^2 (1-e^{-\lambda_k (\Delta)})^2 |\langle Q e_k , e_j \rangle|^2 \Big| \int^{(i-1)\Delta}_0 e^{-(\lambda_j+\lambda_k)r} \dd r\Big|^2 \\ 
			&\quad= \sum_{j = 1}^{\infty} \frac{(1-e^{-\lambda_j (\Delta)})^2 (1-e^{-\lambda_k (\Delta)})^2(1 - e^{-(\lambda_j + \lambda_k)(i-1)\Delta})^2}{\lambda_k^{\gamma}(\lambda_j + \lambda_k)^2} |\langle QA^{\frac \gamma 2}e_k , e_j \rangle|^2.
		\end{align*}
		Since $(1-e^{-x(\Delta)})/x \le \Delta$ for $x \in (0,\infty)$ and $(1 - e^{-(\lambda_j + \lambda_k)(i-1)\Delta}) \le 1$, we obtain 
		\begin{align*}
			&\big\| \E[(E(\Delta)Y_{(i-1)\Delta})^{\otimes 2}] \big\|_{\cL_2(H)}^2 \\
			= &  \big\| \E[P_N E(\Delta)Y_{(i-1)\Delta}^{\otimes 2}] \big\|_{\cL_2(H)}^2+\big\| \E[(I-P_N)(E(\Delta)Y_{(i-1)\Delta})^{\otimes 2}] \big\|_{\cL_2(H)}^2 \\
			&= \sum_{k=1}^{N} \| \E[(E(\Delta)Y_{(i-1)\Delta})^{\otimes 2}] e_k \|^2+\sum_{k=N+1}^\infty \| \E[(E(\Delta)Y_{(i-1)\Delta})^{\otimes 2}] e_k \|^2 \\
			&\le \sum_{k=1}^{N} \| \E[Y_{(i-1)\Delta}^{\otimes 2}]\|_{\cL(H)}^2\|(E(\Delta)e_k \|^2+ \Delta^{\min(2+\gamma,4)} \| (I-P_N)A^{\frac \gamma 2}Q\|_{\cL_2(H)}^2\\
			& \le T\| Q\|_{\cL_2(H)} \sum_{k=1}^{N}\|(E(\Delta)e_k \|^2+ \Delta^{\min(2+\gamma,4)} \| (I-P_N)A^{\frac \gamma 2}Q\|_{\cL_2(H)}^2.
		\end{align*}
		The first summand converges to $0$ for $N$ fixed and as $\Delta\to 0$, whereas the second summand converges to $0$ as $N\to \infty$ uniformly in $\Delta$.
		An appeal to the triangle inequality completes the proof. 
	\end{proof}
	\begin{remark}
		Modifying the condition in Assumption~\ref{ass:Ph} to $\| A^{\gamma/2} Q A^{\tilde \gamma/2}\|_{\cL_2(H)} < \infty$ for some $\tilde \gamma \ge 0$, the bound in claim~\ref{lem:time-conv:2} can be improved to $C T \Delta^{\frac{\min(\gamma+\tilde \gamma,2)}{2}}$. This would, however, not affect our main result. 
	\end{remark}
	
	Lemmata~\ref{lem:var-conv} and~\ref{lem:time-conv} suffice to derive convergence rates in the case that $O = I$. If this is not so, we need an additional lemma handling the convergence in space. 
	
	\begin{lemma}
		\label{lem:space-conv}
		Suppose that $O = P_h$ and that Assumption~\ref{ass:Ph} is satisfied. Then there is a $C<\infty$ that does not depend on $T, h, \Delta $ or $i = 1, \ldots, \ulT$ such that
		\begin{enumerate}[label=(\roman*)]
			\item \label{lem:space-conv:1} $\|\E[(\tilde \Delta W^i_O)^{\otimes 2}] - \E[(\tilde \Delta W^i_I)^{\otimes 2}] \|_{\cL_2(H)} \le C \Delta h^{\min(\gamma,1)}$ and
			\item \label{lem:space-conv:2} $\|\E[(OE(\Delta)Y_{(i-1)\Delta})^{\otimes 2} - (E(\Delta)Y_{(i-1)\Delta})^{\otimes 2}] \|_{\cL_2(H)} \le C \Delta h^{\min(\gamma,1)}$.
		\end{enumerate}
		
		Suppose instead that $O = I_h$ and that Assumption~\ref{ass:Ih} is satisfied. Let $\psi \in (d/2,1+\zeta)$ and write $\tilde \epsilon = \max(\psi-\zeta,0)+\max(d/2-\eta+\epsilon,0)$ for $ \epsilon  > 0$. Then, for sufficiently small $\epsilon>0$, claim \ref{lem:space-conv:1} holds with the bound $\Delta h^{\min(\gamma,1)}$ replaced by $\Delta^{\frac{2-\tilde \epsilon}{2}} h^{\min(\psi,2)}$ and claim \ref{lem:space-conv:2} holds with $\Delta h^{\min(\gamma,1)}$ replaced by $\Delta^{\frac{2-\tilde \epsilon}{2}} h^{\min(\psi,2)}$.
	\end{lemma}
	\begin{proof}
		We first derive claim \ref{lem:space-conv:1}, starting with the case that $O = P_h$. First note that by~\eqref{eq:operator-product-difference} and  \eqref{eq:semigroup-analyticity} with $r=0$,
		\begin{align*}
			&\|\E[(\tilde \Delta W^i_{P_h})^{\otimes 2}] - \E[(\tilde \Delta W^i_I)^{\otimes 2}] \|_{\cL_2(H)} \\
			&\hspace{0.3em}= \Big\| \int_{0}^{\Delta} P_h S(r) Q S(r) P_h - S(r) Q S(r) \dd r \Big\|_{\cL_2(H)}  \\
			&\hspace{0.3em}\le \int_{0}^{t_i-t_{i-1}} \| (P_h - I) S(r) Q S(r) (P_h + I) \|_{\cL_2(H)} \dd r  \\
			&\hspace{0.3em}\le \int_{0}^{t_i-t_{i-1}} \| (P_h - I) \|_{\cL(\dot{H}^\gamma,H)} \| S(r) \|_{\cL(H)}^2 \| Q \|_{\cL_2(H,\dot{H}^\gamma)} \| P_h + I \|_{\cL(H)} \dd r,
		\end{align*}
		which by Lemma~\ref{lem:obs-error} is bounded by a constant times $h^{\min(\gamma,1)} \Delta$.
		If instead $O = I_h$ and Assumption~\ref{ass:Ih} is satisfied, we note that
		\begin{align*}
			&\|\E[(\tilde \Delta W^i_{I_h})^{\otimes 2}] - \E[(\tilde \Delta W^i_I)^{\otimes 2}] \|_{\cL_2(H)} \\
			&\quad\le \int_{0}^{\Delta} \| (I_h - I) S(r) Q ((I_h + I)S(r))^* \|_{\cL_2(H)} \dd r \\
			&\quad= \int_{0}^{\Delta} \| (I_h - I) S(r) A^{-\frac \zeta 2} A^{\frac \zeta 2} Q A^{\frac \eta 2} ((I_h + I)S(r)A^{-\frac \eta 2})^* \|_{\cL_2(H)} \dd r \\
			&\quad\lesssim \int_{0}^{\Delta} \| (I_h - I) S(r) \|_{\cL(\dot{H}^\zeta,H)} \| Q \|_{\cL(\dot{H}^{-\eta},\dot{H}^\zeta)} \|(I_h + I)S(r)\|_{\cL_2(\dot{H}^\eta,H)} \dd r.
		\end{align*}
		Using \eqref{eq:semigroup-analyticity} and Lemma~\ref{lem:obs-hs}, we obtain  	
		\begin{equation}\label{eq:semigroup-observation operator regularity}
			\begin{split}
				\|(I_h + I)S(r)\|_{\cL_2(\dot{H}^\eta,H)} &= \|(I_h + I)A^{-\frac{d/2+\epsilon}{2}} S(r) A^{\frac{d/2+\epsilon-\eta}{2}} \|_{\cL_2(H)} \\ &\le \|I_h + I \|_{\cL_2(\dot{H}^{d/2+\epsilon,H})} \| S(r) A^{\frac{d/2+\epsilon-\eta}{2}} \|_{\cL(H)} \\
				&\lesssim r^{-\frac{\max(d/2-\eta+\epsilon,0)}{2}},
			\end{split}
		\end{equation}
		while Lemma~\ref{lem:obs-error} ensures that
		\begin{align*}
			\| (I_h - I) S(r) \|_{\cL(\dot{H}^r,H)} \le &\| (I_h - I) \|_{\cL(\dot{H}^\psi,H)} \| S(r) A^{\frac{\psi-\zeta}{2}} \|_{\cL(H)}\\
			\lesssim & h^{\min(\psi,2)} r^{-\frac{\max(0,\psi-\zeta)}{2}}.
		\end{align*}
		This then yields, since $\tilde \epsilon \le d/2-\eta+\epsilon+\psi-\zeta < 2$,
		\begin{equation*}
			\|(I_h + I)S(r)\|_{\cL_2(\dot{H}^s,H)} \lesssim h^{\min(\psi,2)} \int_{0}^{t_i-t_{i-1}} r^{-\frac{\tilde \epsilon}{ 2}} \dd r \lesssim \Delta^{\frac{2-\tilde \epsilon}{2}} h^{\min(\psi,2)},
		\end{equation*}
		which finishes the proof of claim~\ref{lem:space-conv:1}.
		
		To see that~\ref{lem:space-conv:2} holds true when $O =P_h$ we note, similar to the proof of Lemma~\ref{lem:time-conv}, that
		\begin{align*}
			&\Big\| \int^{(i-1)\Delta}_0 E(\Delta) S(r) Q S(r) E(\Delta) \dd r \Big\|_{\cL_2(H,\dot{H}^\gamma)}^2 \\
			&\quad= \sum_{j,k = 1}^{\infty} \frac{(1-e^{-\lambda_j (\Delta)}) (1-e^{-\lambda_k (\Delta)})(1 - e^{-(\lambda_j + \lambda_k)(i-1)\Delta})}{(\lambda_j + \lambda_k)^2} |\langle A^{\frac \gamma 2}Qe_k , e_j \rangle|^2,
		\end{align*}
		which can be bounded by $\Delta^2 \| Q \|_{\cL_2(H,\dot{H}^\gamma)}^2$. By this fact and \eqref{eq:operator-product-difference},
		\begin{align*}
			&\|\E[(P_hE(\Delta)Y_{(i-1)\Delta})^{\otimes 2} - (E(\Delta)Y_{(i-1)\Delta})^{\otimes 2}] \|_{\cL_2(H)} \\ 
			&\quad= \|P_h\E[(E(\Delta)Y_{(i-1)\Delta})^{\otimes 2}]P_h - \E[(E(\Delta)Y_{(i-1)\Delta})^{\otimes 2}] \|_{\cL_2(H)} \\ 
			&\quad\le \Big\| (P_h - I) \int^{(i-1)\Delta}_0 E(\Delta) S(r) Q S(r) E(\Delta) \dd r (P_h + I) \Big\|_{\cL_2(H)} \\ 
			&\quad\le \|P_h - I\|_{\cL(\dot{H}^\gamma,H)} \Big\|\int^{(i-1)\Delta}_0 E(\Delta) S(r) Q S(r) E(\Delta) \dd r \Big\|_{\cL_2(H,\dot{H}^\gamma)},
		\end{align*}
		which by Lemma~\ref{lem:obs-error} is bounded by a constant times $\Delta h^{\min(\gamma,1)}$. If instead $O = I_h$ and Assumption~\ref{ass:Ih} is satisfied, we first note that 
		\begin{equation}
			\label{eq:lem:space-conv:pfeq1}
			\begin{split}				
				&\|\E[(I_hE(\Delta)Y_{(i-1)\Delta})^{\otimes 2} - (E(\Delta)Y_{(i-1)\Delta})^{\otimes 2}] \|_{\cL_2(H)} \\ 
				&\quad\le \int^{(i-1)\Delta}_0 \big\| (I-I_h)E(\Delta) S(r) Q ((I + I_h) E(\Delta) S(r))^* \dd r \big\|_{\cL_2(H)} \dd r. 
			\end{split}
		\end{equation}
		By  \eqref{eq:semigroup-time-bound}, \eqref{eq:semigroup-analyticity} and Lemma~\ref{lem:obs-hs}, we obtain for $(i-1)\Delta \le 1$
		\begin{align*}
			&\|(I_h + I) E(\Delta) S(r)\|_{\cL_2(\dot{H}^s,H)} \\
			&\quad\le \|(I_h + I)A^{-\frac{d/2+2\epsilon}{2}} E(\Delta) A^{-\frac{1-(d/2+\epsilon-\eta)}{2}} S(r) A^{\frac{1-\epsilon}{2}} \|_{\cL_2(H)} \\
			&\quad\le \|I_h + I \|_{\cL_2(\dot{H}^{d/2+2\epsilon},H)} \| E(\Delta) A^{-\frac{1-(d/2+\epsilon-\eta)}{2}} \|_{\cL(H)} \| S(r) A^{\frac{1-\epsilon}{2}} \|_{\cL(H)} \\
			&\quad\lesssim (\Delta)^{\frac{1-\max(d/2-\eta+\epsilon,0)}{2}} r^{\frac{\epsilon-1}{2}},
		\end{align*}
		while using Lemma~\ref{lem:obs-error}  in a similar argument yields
		\begin{align*}
			&\|(I_h - I) E(\Delta) S(r)\|_{\cL(\dot{H}^\zeta,H)} \\
			&\quad\le \|I_h - I \|_{\cL(\dot{H}^{\psi},H)} \| E(\Delta) A^{-\frac{1-(\psi-\zeta)}{2}} \|_{\cL(H)} \| S(r) A^{\frac{1}{2}} \|_{\cL(H)} \\
			&\quad\lesssim h^{\min(\psi,2)}  (\Delta)^{\frac{1-\max(\psi-\zeta,0)}{2}} r^{-\frac{1}{2}}.
		\end{align*}
		For $(i-1)\Delta > 1$ we obtain similar estimates along with the factor $e^{-\lambda_1 r/2}$ stemming from~\eqref{eq:semigroup-exp-decay}. Plugging these estimates into~\eqref{eq:lem:space-conv:pfeq1}, using also the fact that 
		\begin{align*}
			&(I-I_h)E(\Delta) S(r) Q ((I + I_h) E(\Delta) S(r))^* \\
			&\quad= (I-I_h)E(\Delta) S(r) A^{-\frac r 2}A^{\frac r 2}Q A^{\frac s 2} ((I + I_h) E(\Delta) S(r)A^{-\frac s 2})^*,
		\end{align*}
		and Assumption~\ref{ass:Ih}, yields
		\begin{align*}
			&\|\E[(I_hE(\Delta)Y_{(i-1)\Delta})^{\otimes 2} - (E(\Delta)Y_{t_{i-1}})^{\otimes 2}] \|_{\cL_2(H)} \\
			&\quad\lesssim  (\Delta)^{\frac{2-\tilde \epsilon}{2}} h^{\min(\psi,2)}  \Big(\int^{1}_0 r^{\frac{\epsilon-2}{2}} \dd r + \int^{(i-1)\Delta}_1 e^{-\lambda_i r} \dd r\Big)\\
			&\quad\lesssim \Delta^{\frac{2-\tilde \epsilon}{2}} h^{\min(\psi,2)} ,
		\end{align*}
		which completes the proof.
	\end{proof}
	Next, we derive a lemma related to the initial value $X_0$.
	\begin{lemma}
		\label{lem:X0}
		Let Assumption \ref{ass:rate} hold.
		\begin{enumerate}[label=(\Alph*)]
			\item If $O \neq I_h$ and Assumption \ref{ass:Ph} are satisfied and $\iota < 1$, then there exists a bounded random univariate sequence $\zeta_{\Delta}$ converging to $0$ in mean square as $\Delta \to 0$ such that
			\[
			\sum_{i=1}^{\ulT} \|OE(\Delta)S((i-1)\Delta)X_0\|^2 \le \Delta^{\iota} \zeta_{\Delta}.
			\]
			If instead $\iota \ge 1$, the bound is replaced by $\Delta \| X_0 \|^2_{\dot{H}^{\iota}}$
			\item If $O=I_h$ and Assumption \ref{ass:Ih} is satisfied, then there is deterministic constant $C < \infty$ such that 
			\[
			\sum_{i=1}^{\ulT} \|OE(\Delta)S((i-1)\Delta)X_0\|^2 \le C (\Delta^{\min(\iota,1)}+h^{\min(2\iota,4)}) \| X_0 \|^2_{\dot{H}^{\iota}}
			\]
			almost surely.
		\end{enumerate}
	\end{lemma}
	
	\begin{proof}			
		Applying Lemma~\ref{lem:obs-error}, we have with $r=\min(\iota,2)$,
		\begin{align*}
			&\sum_{i=1}^{\ulT} \| (I - I_h) E(\Delta) S((i-1)\Delta) X_0\|^2\\
			&\lesssim h^{2 r} \sum_{i=1}^{\ulT} \| A^{\frac r 2} E(\Delta ) S((i-1)\Delta) X_0\|^2 \\
			&= h^{2 r}\sum_{k=1}^\infty \sum_{i=1}^{\ulT} e^{-2 \lambda_k (i-1)\Delta} (1-e^{-\Delta   \lambda_k })^2| \langle A^{\frac r 2} X_0, e_k \rangle|^2 \\
			&= h^{2 r}\sum_{k=1}^\infty (1-e^{-\lambda_k T })(1-e^{-\lambda_k \Delta }) |\langle A^{\frac r 2} X_0, e_k \rangle|^2 \\
			&\le h^{2 r}\| E(\Delta )^{1/2} A^{\frac r 2} X_0 \|^2,
		\end{align*}
		where we in the third step made use of the formula for geometric sums.
				
		In the same way, if $O \neq I_h$, the bound   
		\begin{equation*}
			\sum_{i=1}^{\ulT} \| O E(\Delta) S((i-1)\Delta) X_0\|^2 \le \| E(\Delta )^{1/2} X_0 \|^2
		\end{equation*}
		is obtained. Combining these results yields the second claim. Let $P_N:= \sum_{k=1}^N e_k^{\otimes 2}$ be the projection onto the subspace spanned by the first $N$ eigenfunctions of $A$. We set
		\begin{align*}
			\zeta_\Delta := \Delta^{-\iota} \sum_{k=1}^N (1-e^{-\lambda_k\Delta}) |\langle X_0, e_k \rangle|^2 + \Delta^{-\iota} \sum_{k=N+1}^\infty (1-e^{-\lambda_k\Delta}) |\langle X_0, e_k \rangle|^2.
		\end{align*}
		Since the first term tends to $0$ for fixed $N$ and the second term tends to $0$ as $N\to \infty$ uniformly in $\Delta$, the result follows.
	\end{proof}
	
	We are now ready to prove Theorem \ref{thm:RV-convergence}.
	
	\begin{proof}[Proof of Theorem \ref{thm:RV-convergence}]
		We start by noting that, since $Y_{(i-1)\Delta}$ and $\tilde \Delta W^i_O$ are independent,
		\begin{align*}
			&(O X_{i\Delta} - O X_{(i-1)\Delta})^{\otimes 2} - \Delta Q \\
			&\quad= (O E(\Delta) S((i-1)\Delta) X_0)^{\otimes 2}\\
			&\qquad+O E(\Delta) S((i-1)\Delta) X_0\otimes(O Y_{i\Delta} - O Y_{(i-1)\Delta})\\
			&\qquad+(O Y_{i\Delta} - O Y_{(i-1)\Delta})\otimes O E(\Delta ) S((i-1)\Delta) X_0\\
			&\qquad+ 
			(O Y_{i\Delta} - O Y_{(i-1)\Delta})^{\otimes 2} - \E[(O Y_{i\Delta} - O Y_{(i-1)\Delta})^{\otimes 2}] \\
			&\qquad+ \E[(OE(\Delta)Y_{(i-1)\Delta})^{\otimes 2} - (E(\Delta)Y_{(i-1)\Delta})^{\otimes 2}]\\
			&\qquad+ \E[(E(\Delta)Y_{(i-1)\Delta})^{\otimes 2}]\\
			&\qquad+\E[(\tilde \Delta W^i_O)^{\otimes 2}] - \E[(\tilde \Delta W^i_I)^{\otimes 2}]\\
			&\qquad+ \int^{i\Delta}_{(i-1)\Delta} S(i\Delta-r) Q S(i\Delta-r) - Q \dd r =\mathrm{I}_i + \mathrm{II}_i + \cdots + \mathrm{VIII}_i.
		\end{align*}
			These summands either contribute to the bias and the variance of the estimator. To be precise,
			\begin{align*}
				\left\|Bias\left(\frac{\mathrm{RV}^{\Delta,O}_T}{T}\right)\right\|_{\cL_2(H)}^2
				%= &  \left\|\mathbb E\left[\frac{\mathrm{RV}^{\Delta,O}_T}{T}\right]-Q\right\|_{\cL_2(H)}^2\\
				= &  \left\| T^{-1}\sum_{i=1}^{\ulT} \mathbb E[\mathrm{I}_i]+ \mathrm{V}_i+ \mathrm{VI}_i + \mathrm{VII}_i + \mathrm{VIII}_i\right\|_{\cL_2(H)}^2
			\end{align*}
			and
			\begin{align*}
				Var\left(\frac{\mathrm{RV}^{\Delta,O}_T}{T}\right)
				%= &  \mathbb E\left[\left\|T^{-1}\sum_{i=1}^{\ulT}(O X_{i\Delta} - O X_{(i-1)\Delta})^{\otimes 2}-\mathbb E\left[T^{-1}\sum_{i=1}^{\ulT}(O X_{i\Delta} - O X_{(i-1)\Delta})^{\otimes 2}\right]\right\|^2\right]\\
				= & T^{-2} \mathbb E\left[\left\|\sum_{i=1}^{\ulT} (\mathrm{I}_i-\mathbb E[\mathrm{I}_i])+\mathrm{II}_i+ \mathrm{III}_i + \mathrm{IV}_i\right\|_{\cL_2(H)}^2\right].
			\end{align*}
			
			Thus, by Lemma~\ref{lem:X0}, Lemma~\ref{lem:space-conv} and Lemma~\ref{lem:time-conv}, we find
			\begin{align*}
				& Bias\left(\frac{\mathrm{RV}^{\Delta,O}_T}{T}\right)^2\\
				&\leq \begin{cases}T^{-2}\Delta^{\min(2\iota,2)}+\Delta^{\min(\gamma,2)} & \text{ if }O = I,\\
					T^{-2}\Delta^{\min(2\iota,2)}+h^{\min(2\gamma,2)}+ \Delta^{\min(\gamma,2)} 
					& \text{ if }O = P_h,\\
					T^{-2}(\Delta^{\min(2\iota,2)}+h^{\min(4\iota,8)}) + \Delta^{-\bar \epsilon}h^{2\min(\psi,2)}+ \Delta^{\min(\gamma,2)}
					& \text{ if }O = I_h.
				\end{cases} 
			\end{align*}

			For upper bounds on the variance, we have
			\begin{align*}
				&\Big\|\sum^{\ulT}_{i=1} \mathrm{II}_i \Big\|^2_{L^2(\Omega,\cL_2(H))} \\
				&\quad\lesssim \sum^{\ulT}_{i=1} \E[\|O E(\Delta ) S((i-1)\Delta) X_0\|_H^2] \\
				&\hspace{5em}\times \big(\E[\|\tilde\Delta W^1_O\|_H^2]+\E[\| E^{i-1}_O\|_H^2]\big) \\
				&\qquad+\sum^{\ulT}_{j=2}\sum^{\ulT-1}_{i=j+1} \E[\langle  O E(\Delta) S((i-1)\Delta) X_0,O E(\Delta) S((j-1)\Delta) X_0  \rangle] \\ &\hspace{4em}\times\E[\langle E_{i-j,O}^{j-1,0} , E_{O}^{j-1,0} \rangle] \\
				&\quad\le \sum^{\ulT}_{i=1} \E[\|O E(\Delta ) S((i-1)\Delta) X_0\|_H^2] \big(\E[\|\tilde\Delta W^1_O\|_H^2]+\E[\| E^{i-1}_O\|_H^2]\big)\\
				&\qquad+\Big(\sum^{\ulT}_{i=1} \| O E(\Delta) S((i-1)\Delta) X_0\|^2_{L^2(\Omega,H)}\Big)^{2}\\
				&\qquad+\sum^{\ulT}_{j=2}\sum^{\ulT-1}_{i=j+1} \| E_{i-j,O}^{j-1,0} \|^2_{L^2(\Omega,H)} \|E_{O}^{j-1,0}\|^2_{L^2(\Omega,H)},
			\end{align*}
			with the corresponding bound for $\mathrm{III}_i$ being found in the same way. 
			
			If $O\neq I_h$,
			we apply Lemma \ref{lem: second moments of increments} and Lemma \ref{lem:X0} to obtain a bound proportional to
			\begin{equation*}
				\Delta^{\min(\iota,1)}(\Delta+\Delta \zeta_{\Delta})+\Delta^{\min(2\iota,2)}+ \Delta \zeta_{\Delta}T
				\lesssim \Delta \zeta_{\Delta} T+\Delta^{\min(2\iota,2)},%\Delta(\Delta^{\frac{1}{2}}h+\Delta) + (\Delta^{\frac{1}{2}}h+\Delta)^2 + T \Delta \lesssim T \Delta.
			\end{equation*}
			for $\zeta_\Delta$ converging to $0$ as $\Delta \to 0$.
			
			If $O= I_h$,
			we use Lemma \ref{lem: second moments of increments} and Assumption~\ref{ass:Ih} as well as Lemma \ref{lem:X0}. Since $\Delta^{\min(\iota,1)}\Delta^{\min(1,(1+\beta)/2)}<\Delta$ due to the assumption $\iota>\frac 12$, we obtain a bound proportional to
			\begin{align*}
				&(\Delta^{\min(\iota,1)}\Delta^{\min(1,(1+\beta)/2)}+ h^{\min(2\iota,4)}\Delta^{\min(1, (1+\beta)/2)}) +(\Delta^{\min(\iota,1)}+h^{\min(2\iota,4)})^2\\
				& \qquad+ T\Delta^{\min(2, \beta + \epsilon/4,(1+\beta+\epsilon/2)/2})\\
				\lesssim & \Delta +h^{\min(4\iota,8)}+ T\Delta^{\min(2, \beta+\epsilon/4,(1+\beta+\epsilon/2)/2}).
			\end{align*}
			
				Moreover, by Lemma \ref{lem:X0} 
				\begin{align*}
					\mathbb E\left[  \left\| \sum_{i=1}^{\ulT} I_i-\mathbb E[I_i]\right\|_{\cL_2(H)}^2\right]
					\lesssim & \mathbb E\left[   \left(\sum_{i=1}^{\ulT} \left\|I_i\right\|_{\cL_2(H)}\right)^2\right]\\  \lesssim &\begin{cases}
						\Delta^{\min(2,2\iota)} & O\neq I_h,\\
						\Delta^{\min(2,2\iota)}+ h^{4\min(\iota,2)} & O= I_h.
					\end{cases}
				\end{align*}
				
				Hence, using also Lemma~\ref{lem:var-conv} to bound $\sum^{\ulT}_{i=1} \mathrm{IV}_i$ and since in the case that $O=I_h$ we assumed $\iota>\frac 12$ we obtain that 
				\begin{align*}
					Var\left(\frac{\mathrm{RV}^{\Delta,O}_T}{T}\right)\lesssim \begin{cases}T^{-1}\Delta+T^{-2}\Delta^{\min(2\iota,2)}& \text{ if }O = I,\\
						T^{-1}\Delta+T^{-2}\Delta^{\min(2\iota,2)}
						& \text{ if }O = P_h,\\
						T^{-1}\Delta^{\min(1,\beta)}+T^{-2}h^{\min(4\iota,8)}
						& \text{ if }O = I_h.
					\end{cases} 
				\end{align*}
				This proves the claim
			\end{proof}
			
			We end this section with the proof of the lower bound of  Proposition~\ref{prop:lower-bound}.
			\begin{proof}
				Without loss of generality, let $T=1$ and $X_0=0$. Let $A=\sum_{j=1}^{\infty} \lambda_j e_j^{\otimes 2}$ and  $Q=\sum_{j=1}^{\infty} \mu_j e_j^{\otimes 2}$ with respect to an orthonormal basis $(e_j)_{j=1}^\infty$ of $H$. Then, the squared bias may, with $T=1$ and $n:=\ulT$, be expressed by
				\begin{align*}
					&\left\|\mathbb E\left[\sum_{i=1}^{n} (X_{i\Delta}-X_{(i-1)\Delta})^{\otimes 2}\right]-Q\right\|^2_{\mathcal L_2(H)}\\
					\quad &= \sum_{j=1}^{\infty} \mu_j^2\left(\sum_{i=1}^n\int_0^{\Delta}e^{-2\lambda_js}ds+\int_0^{(i-1)\Delta}e^{-2\lambda_js}(1-e^{-\lambda_j\Delta})^2ds-\Delta\right)^2\\
					\quad &= \sum_{j=1}^{\infty} \mu_j^2\left(n\frac{(1-e^{-2\lambda_j\Delta})}{2\lambda_j}+\frac{(n-\sum_{i=1}^n e^{-2\lambda_j(i-1)\Delta})}{2\lambda_j}(1-e^{-\lambda_j\Delta})^2-1\right)^2\\
					\quad &=  \sum_{j=1}^{\infty} \frac{\mu_j^2}{4\lambda_j^2}(1-e^{-\lambda_j\Delta})^2\left(2n-\frac{1-e^{-2\lambda_j}}{1+e^{-\lambda_j\Delta}}-2\frac{\lambda_j}{(1-e^{-\lambda_j\Delta})}\right)^2.
				\end{align*}
				Now, by the mean value theorem we have $n(1-e^{-\lambda_j\Delta})\leq \lambda_j$, so
				\begin{align*}
					\left\| \mathbb{E}\left[\frac{RV^{\Delta,I}_T}{T}\right] - Q \right\|_{\mathcal{L}_2(H)}
					\geq & \sum_{j=1}^{\infty} \frac{\mu_j^2}{\lambda_j^2}(1-e^{-\lambda_j\Delta})^2\left(n-\frac{\lambda_j}{(1-e^{-\lambda_j\Delta})}\right)^2\\
					= & \sum_{j=1}^{\infty} \frac{\mu_j^2}{\lambda_j^2}\left((1-e^{-\lambda_j\Delta})n-\lambda_j\right)^2\\
					\geq & \frac{\Delta^2}{4} \sum_{j=1}^{\infty} \mu_j^2\lambda_j^2  e^{-2\lambda_j \Delta },
				\end{align*}
				where we in the last step used the fact that the function $$[0,\infty) \ni  x \mapsto ( (1-e^{-x})/x - 1)^2 e^x/x^2$$ is bounded from below by $1/4$.
				
				We set $\lambda_j = j^2$ for $j \ge 1$ and let $Q$ be defined by $\mu_1 = 0$ and 
				\[
				\mu_j = \frac{1}{j^{\gamma+1/2} \log j} \quad \text{for } j \ge 2.
				\]
				This choice of $Q$ is trace-class since $\gamma > 1/2$ and satisfies Assumption~\ref{ass:rate}\ref{ass:rate:q} since
				\[
				\|A^{\gamma/2}Q\|_{\cL_2(H)}^2 = \sum_{j=2}^\infty \lambda_j^\gamma \mu_j^2 = \sum_{j=2}^\infty j^{2\gamma} \left( \frac{1}{j^{\gamma+1/2} \log j} \right)^2 = \sum_{j=2}^\infty \frac{1}{j (\log j)^2} < \infty.
				\]
				With these choices,
				\begin{equation*}
					\left\|\mathbb E\left[\sum_{i=1}^{n} (X_{i\Delta}-X_{(i-1)\Delta})^{\otimes 2}\right]-Q\right\|^2_{\mathcal L_2(H)}
					\geq \frac{\Delta^2}{4} \sum_{j=2}^\infty \frac{j^{3-2\gamma}}{(\log j)^2} e^{-2j^2\Delta}.
				\end{equation*}
				Let $N = \lfloor 1/\sqrt{2\Delta} \rfloor$. For $j \le N$, $e^{-2j^2\Delta} \ge e^{-1}$. Thus,
				\[
				\sum_{j=2}^\infty \frac{j^{3-2\gamma}}{(\log j)^2} e^{-2j^2\Delta} \ge e^{-1} \sum_{j=2}^{N} \frac{j^{3-2\gamma}}{(\log j)^2}.
				\]
				The function $f(x) = x^{3-2\gamma}/(\log x)^2$ is eventually increasing for $\gamma \in (1/2,1)$. Therefore, for some fixed $c > 2$ and all $N > 2^{1/(4-2\gamma)} c$,
				\begin{align*}
					\sum_{j=2}^{N} \frac{j^{3-2\gamma}}{(\log j)^2} \ge \int_c^{N} \frac{x^{3-2\gamma}}{(\log x)^2} \dd x 
					&\ge \frac{1}{(\log N)^2} \int_c^{N} x^{3-2\gamma} \dd x \\
					&= \frac{N^{4-2\gamma} - c^{4-2\gamma}}{(4-2\gamma)(\log N)^2} \ge \frac{N^{4-2\gamma}}{(8-4\gamma)(\log N)^2}.
				\end{align*}
				In light of the definition of $N$, this completes the proof.
			\end{proof}
			
			\subsubsection{Proofs of Section \ref{Sec: Asymptotic Normality}}
			
			To prove Theorem \ref{thm: CLT}, we also need, in the pointwise-sampling case, rates of convergence of the discretization error of the semigroup adjusted covariation in the weak operator topology.
			
			\begin{lemma}\label{prop:temp-SARCV-diff-proj}
				Let Assumption \ref{ass:Ih} hold. Let $u, v \in H$. Let $\psi \in (d/2,1+\zeta)$ and write $\tilde \epsilon = \max(\psi-\zeta,0)+\max(d/2-\eta+\epsilon,0)$ for $ \epsilon  > 0$. Then, for sufficiently small $\epsilon>0$,  there is a constant $C$, depending on $u, v$ but not on $\Delta$, $T$ or $h$, such that
				\begin{align*} &\left\| \langle \mathrm{SARCV}_T^{\Delta, I_h} - \mathrm{SARCV}_T^{\Delta, I}, u \otimes v \rangle_{\mathcal{L}_2(H)} \right\|_{L^2(\Omega, \R)} \\
					&\quad \le C \Big(T^{1/2} \min(h^{(d+\epsilon)/2} \Delta^{\min(1,\beta)/2}, h^{\min(2,(\beta+d+1)/2)} \Delta^{\min(\epsilon/2,(2\epsilon+\beta-1)/4)}) \\
					&\hspace{4em} +T\Delta^{\frac{2-\tilde \epsilon}{2}} h^{\min(\psi,2)} \Big). \end{align*}
			\end{lemma}
			
			\begin{proof}
				Let $\xi_i(O) = \langle \tilde \Delta W_{O}^i, u \rangle \langle \tilde \Delta W_{O}^i, v \rangle$. The term whose $L^2(\Omega, \R)$ norm we want to bound is
				\begin{align*}
					&\langle \mathrm{SARCV}_T^{\Delta, {I_h}} - \mathrm{SARCV}_T^{\Delta, I}, u \otimes v \rangle_{\mathcal{L}_2(H)} \\
					&= \sum_{i=1}^{\ulT} \langle (\tilde \Delta W_{I_h}^i)^{\otimes 2} - (\tilde \Delta W_I^i)^{\otimes 2}, u \otimes v \rangle_{\mathcal{L}_2(H)} \\
					&= \sum_{i=1}^{\ulT} (\xi_i({I_h}) - \xi_i(I)).
				\end{align*}
				Since $(\tilde \Delta W_{I_h}^i, \tilde \Delta W_I^i)_{i=1}^{\ulT}$ are independent and stationary pairs,
				\begin{align*}
					&\left\| \sum_{i=1}^{\ulT} (\xi_i({I_h}) - \xi_i(I)) \right\|_{L^2(\Omega, \R)}^2 \\
					&\quad= \ulT \left(\Var(\xi_i({I_h}) - \xi_i(I)) + \ulT \E[\xi_i({I_h}) - \xi_i(I)]^2 \right).
				\end{align*}
				Let $D_i = \tilde \Delta W_{I_h}^i - \tilde \Delta W_I^i$ and $S_i = \tilde \Delta W_{I_h}^i + \tilde \Delta W_I^i$. Then, as in \eqref{eq:operator-product-difference},
				\begin{align*}
					\xi_i({I_h}) - \xi_i(I) &= \frac{1}{2} \left( \langle D_i, u \rangle \langle S_i, v \rangle + \langle S_i, u \rangle \langle D_i, v \rangle \right).
				\end{align*}
				
				For the variance term, we use the facts that $D_i$ and $S_i$ are Gaussian random variables to obtain
				\begin{align*}
					&	\E[(\xi_i({I_h}) - \xi_i(I))^2] \\
					&= \frac{1}{4} \E\left[ \left( \langle D_i, u \rangle \langle S_i, v \rangle + \langle S_i, u \rangle \langle D_i, v \rangle \right)^2 \right] \\
					&\le \frac{1}{2} \E\left[ (\langle D_i, u \rangle \langle S_i, v \rangle)^2 + (\langle S_i, u \rangle \langle D_i, v \rangle)^2 \right]  \\
					&\le \frac{1}{2} \Big( \E[\langle D_i, u \rangle^4]^{1/2} \E[\langle S_i, v \rangle^4]^{1/2} + \E[\langle S_i, u \rangle^4]^{1/2} \E[\langle D_i, v \rangle^4]^{1/2} \Big) \\
					&\lesssim \E[\langle D_i, u \rangle^2] \E[\langle S_i, v \rangle^2] + \E[\langle S_i, u \rangle^2] \E[\langle D_i, v \rangle^2].
				\end{align*}
				Here, $$\E[\langle D_i, u \rangle^2] = \int_0^\Delta \langle ({I_h}-I) S(r) Q (({I_h}-I) S(r))^* u,u \rangle \dd r,$$ and similarly $$\E[\langle S_i, u \rangle^2] = \int_0^\Delta \langle ({I_h}+I) S(r) Q (({I_h}+I) S(r))^* u,u \rangle \dd r.$$ 
				
				By \eqref{eq:operator-product-difference}, $\E[\xi_1({I_h}) - \xi_1(I)]$ is given by 
				\begin{align*}
					\frac{1}{2} \left\langle \int_0^\Delta [ ({I_h}+I) S(r) Q  (({I_h}-I)S(r))^* + ({I_h}-I) S(r) Q (({I_h}+I)S(r))^* ] u \dd r, v \right\rangle .
				\end{align*}
				We first consider the variance term. We have
				\begin{align*}
					\E[\langle D_i, u \rangle^2] 
					&\le \int_0^\Delta \| (I_h - I) S(r) Q^{1/2} u\|^2 \dd r.
				\end{align*}
				Applying Lemma~\ref{lem:obs-error}, Assumption~\ref{ass:Ih} and~\eqref{eq:semigroup-analyticity}, we obtain on the one hand, 
				\begin{align*}
					&\| (I_h - I) S(r) Q^{1/2} \|_{\mathcal{L}(H)} \\
					&\le \| I_h - I \|_{\cL(\dot{H}^{(d+\epsilon)/2}, H)} \| S(r) A^{\frac{1-\beta}{4}} \|_{\mathcal{L}(H)} \| A^{\frac{\beta+d+\epsilon-1}{4}} Q^{1/2} \|_{\mathcal{L}(H)} \\
					&\lesssim h^{(d+\epsilon)/2} r^{-\frac{1-\beta}{4}}.
				\end{align*}
				On the other hand, 
				\begin{align*}
					&	\| (I_h - I) S(r) Q^{1/2} \|_{\mathcal{L}(H)} \\
					&\le \| I_h - I \|_{\cL(\dot{H}^{(\beta+d+1)/2}, H)} \| S(r) A^{\frac{2-\epsilon}{4}} \|_{\mathcal{L}(H)} \| A^{\frac{\beta+d+\epsilon-1}{4}} Q^{1/2} \|_{\mathcal{L}(H)} \\
					&\lesssim h^{\min(2,(\beta+d+1)/2)} r^{-\frac{1-\epsilon/2}{2}}.
				\end{align*}
				Therefore
				\begin{align*}
					\E[\langle D_i, u \rangle^2] \lesssim \min(h^{d+\epsilon} \Delta^{\min(1,(1+\beta)/2)}, h^{\min(4,\beta+d+1)} \Delta^{\epsilon}).
				\end{align*}
				For the $S_i$ term, we obtain in the same way,
				\begin{align*}
					\E[\langle S_i, u \rangle^2] 
					&\lesssim \int_0^\Delta \| I_h + I \|^2_{\cL(\dot{H}^{(d+\epsilon)/2}, H)} \| S(r) A^{\frac{1-\beta}{4}} \|^2_{\mathcal{L}(H)} \| A^{\frac{\beta+d+\epsilon-1}{4}} Q^{1/2} \|^2_{\mathcal{L}(H)} \dd r \\
					&\lesssim \Delta^{\min(1,(1+\beta)/2)},
				\end{align*}
				and thus
				\begin{align*}
					\ulT \Var(\xi_i({I_h}) - \xi_i(I)) \lesssim T \min(h^{d+\epsilon} \Delta^{\min(1,\beta)}, h^{\min(4,\beta+d+1)} \Delta^{\min(\epsilon,(2\epsilon+\beta-1)/2)}).
				\end{align*}
				
				For the the bias term, we simply apply the bound
				\begin{align*}
					\left| \E[\xi_1({I_h}) - \xi_1(I)] \right| 
					&\quad\lesssim \int_0^\Delta \| (I_h - I) S(r) Q^{1/2} ((I_h + I) S(r) Q^{1/2})^* \|_{\cL(H)} \dd r \\
					&\quad\lesssim \int_0^\Delta \| (I_h - I) S(r) Q^{1/2} ((I_h + I) S(r) Q^{1/2})^* \|_{\cL_2(H)} \dd r,
				\end{align*}
				which is treated in the same way as in Lemma~\ref{lem:space-conv}\ref{lem:space-conv:1}. Combining these two results completes the proof.
			\end{proof}
			
			\begin{lemma}\label{lem: discret of SARCV}
				Let one of the subsequent conditions be valid:
				\begin{itemize}
					\item $O=P_h$, Assumption \ref{ass: minimal Assumption on A} holds and the subdivisions $(\cT_h)_{h \in (0,1]}$ are nested.
					\item $O=I_h$, Assumption \ref{ass:Ih} holds and  $h=o(\Delta^{\frac{\tilde \epsilon-1}{2\min(\psi,2)}})$ where $\tilde \epsilon$ and $\psi$ are defined as in Theorem \ref{thm: CLT}(C).
				\end{itemize}
				Then the following limit holds in probability.
				\begin{align}\label{eq: sarcv discretization}
					\lim_{\Delta\to 0} \Delta^{-\frac 12}\left( \left(\mathrm{SARCV}^{\Delta,O}_{T} - \E[\mathrm{SARCV}^{\Delta,O}_{T}] \right)-\left(\mathrm{SARCV}^{\Delta,I}_{T} - \E[\mathrm{SARCV}^{\Delta,I}_{T}] \right)\right)
					= & 0.
				\end{align}
		\end{lemma}
		\begin{proof}
			We start with the proof in the case $O=P_h$.  For that, we let $Z_i := P_h(\tilde\Delta W_I^i)^{\otimes 2}P_h - (\tilde\Delta W_I^i)^{\otimes 2}$. 
			The term we want to bound is the $L^2(\Omega,\cL_2(H))$ norm of
			\begin{align*}
				&\left(\mathrm{SARCV}^{\Delta,P_h}_{T} - \E[\mathrm{SARCV}^{\Delta,P_h}_{T}] \right)-\left(\mathrm{SARCV}^{\Delta,I}_{T} - \E[\mathrm{SARCV}^{\Delta,I}_{T}] \right) \\
				&\quad= \sum_{i=1}^{\ulT} (Z_i - \E[Z_i]).
			\end{align*}
			The sequence $(Z_i - \E[Z_i])_{i=1}^{\ulT}$ is a martingale difference sequence. Therefore
			\begin{align*}
				&\left\| \sum_{i=1}^{\ulT} (Z_i - \E[Z_i]) \right\|_{L^2(\Omega,\cL_2(H))}^2 \\
				&\quad= \sum_{i=1}^{\ulT} \E\left[ \| Z_i - \E[Z_i] \|_{\cL_2(H)}^2 \right] \\
				&\quad\le \sum_{i=1}^{\ulT} \E\left[ \| Z_i \|_{\cL_2(H)}^2 \right] = \sum_{i=1}^{\ulT} \E\left[ \left\| P_h(\tilde\Delta W_I^i)^{\otimes 2}P_h - (\tilde\Delta W_I^i)^{\otimes 2} \right\|_{\cL_2(H)}^2 \right].
			\end{align*}
			Applying \eqref{eq:operator-product-difference}, the fact that $P_h$ is bounded  and that $\tilde \Delta W^i_I$ is Gaussian,
			\begin{align*}
				\E\left[ \| Z_i \|_{\cL_2(H)}^2 \right] &\lesssim \E\left[ \left\| (P_h+I)(\tilde\Delta W_I^i)^{\otimes 2}(P_h-I) \right\|_{\cL_2(H)}^2 \right] \\ 
				&\lesssim \E\left[ \left\| (P_h-I) (\tilde\Delta W_I^i)^{\otimes 2} \right\|_{\cL_2(H)}^2 \right] \\
				&= \E\left[ \left\| (P_h-I) \tilde\Delta W_I^i \right\|^2 \left\| \tilde\Delta W_I^i \right\|^2 \right] \\
				&\le \left( \E\left[ \left\| (P_h-I) \tilde\Delta W_I^i \right\|^4 \right] \right)^{1/2} \left( \E\left[ \left\| \tilde\Delta W_I^i \right\|^4 \right] \right)^{1/2} \\
				&\lesssim \E\left[ \left\| (P_h-I) \tilde\Delta W_I^i \right\|^2 \right] \E\left[ \left\| \tilde\Delta W_I^i \right\|^2 \right] \lesssim \Delta \E\left[ \left\| (P_h-I) \tilde\Delta W_I^i \right\|^2 \right].
			\end{align*}
			By It\^o isometry,
			\begin{align*}
				\E\left[ \left\| (P_h-I) \tilde\Delta W_I^i \right\|^2 \right] &= \int_0^\Delta \left\| (P_h-I) S(r) Q^{1/2} \right\|_{\cL_2(H)}^2 \dd r. %\\
				%&= \int_0^\Delta \left\| (P_h-I) S(r) Q S(r) (P_h-I) \right\|_{\cL_1(H)} \dd r. 
			\end{align*} 
			This yields
			\begin{align*}
				&\left(\mathrm{SARCV}^{\Delta,P_h}_{T} - \E[\mathrm{SARCV}^{\Delta,P_h}_{T}] \right)-\left(\mathrm{SARCV}^{\Delta,I}_{T} - \E[\mathrm{SARCV}^{\Delta,I}_{T}] \right) \\
				&\quad\leq \sum_{i=1}^{\ulT} \Delta  \E\left[ \left\| (P_h-I) \tilde\Delta W_I^i \right\|^2 \right]
				\lesssim  T\Delta \sup_{r\leq 1}\left\| (P_h-I) S(r) Q^{1/2} \right\|_{\cL_2(H)}^2.
			\end{align*}
			To see that $\sup_{r\leq 1}\| (P_h-I) S(r) Q^{1/2} \|_{\cL_2(H)}^2$ converges to $0$ as $h\to 0$, first consider $v \in H$ and let $\epsilon > 0$ be arbitrary. By density, we may choose $u \in H^1$ such that $\|v-u\|_H < \epsilon/2$ and by Lemma~\ref{lem:obs-error} we may find an $h_0 > 0$ such that $\|(I-P_h)u\| \le \epsilon/2$ for all $h < h_0$. Since the subdivisions are nested, this shows that as $h \to 0$, $\|(I-P_h)v\| \to 0$ monotonously. By Dini's theorem, using also that $Q^{\frac 12}$ is Hilbert-Schmidt,  the factor $\sup_{r\leq 1}\left\| (P_h-I) S(r) Q^{1/2} \right\|_{\cL_2(H)}^2$ therefore converges to $0$ as $h\to 0$.
			
			An application of Markov's inequality proves the claim in the case that $O=P_h$.
			Let us now turn to the case that $O=I_h$. We begin by proving tightness of the term on the right of \eqref{eq: sarcv discretization} and then prove that its finite dimensional distributions converge to $0$. 
			
			To prove tightness, let $P_N$ be the projection onto $span\{e_1, e_2, \ldots, e_N\}$, where $(e_i)_{i=1}^\infty$ is the orthonormal eigenbasis of $A$ and note that  
			\begin{align*}
				& \E \left[ \left\| P_N^\perp \left( \mathrm{SARCV}_T^{\Delta,I} - \E[\mathrm{SARCV}_T^{\Delta,I}] \right) P_N^\perp \right\|_{\cL_2(H)}^2 \right] \\
				&= \sum_{i=1}^{\ulT} \E \left[ \left\| P_N^\perp \left( (\tilde{\Delta}W_I^i)^{\otimes 2} - \E[(\tilde{\Delta}W_I^i)^{\otimes 2}] \right) P_N^\perp \right\|_{\cL_2(H)}^2 \right] \\
				&\le \sum_{i=1}^{\ulT} \E \left[ \left\| P_N^\perp \tilde{\Delta}W_I^i \right\|_{H}^4  \right] \\
				&\lesssim \sum_{i=1}^{\ulT} \left( \int_0^\Delta \left\| P_N^\perp S(r) Q^{1/2} \right\|_{\cL_2(H)}^2 \dd r \right)^2 \lesssim T \Delta \sup_{r \in [0,1]} \left\| P_N^\perp S(r) Q^{1/2} \right\|_{\cL_2(H)}^2.
			\end{align*}
			Combining the fact that $S$ is a strongly continuous semigroup with Dini's theorem, the sequence $\Delta^{-1/2} (SARCV_{T}^{\Delta,I}-\mathbb E[SARCV_{T}^{\Delta,I}])$ is tight by Corollary D.3 from \cite{BSV2022supplement}. Similarly, 
			\begin{align*}
				&\E \left[ \left\| P_N^\perp \left( \mathrm{SARCV}_T^{\Delta,I_h} - \E[\mathrm{SARCV}_T^{\Delta,I_h}] \right) P_N^\perp \right\|_{\cL_2(H)}^2 \right] \\
				&\quad\lesssim \sum_{i=1}^{\ulT} \left( \int_0^\Delta \left\| P_N^\perp A^{-\frac{\delta}{4}} A^{\frac{\delta}{4}} I_h A^{-\frac{d +\epsilon}{4}} A^{\frac{d +\epsilon}{4}} S(r) Q^{1/2} \right\|_{\cL_2(H)}^2\dd r\right)^2
				\\
				&\quad\lesssim \| P_N^\perp A^{-\frac{\delta}{4}} \|^2_{\cL(H)} \|A^{\frac{\delta}{4}} I_h A^{-\frac{d +\epsilon}{4}} \|^2_{\cL_2(H)} \sum_{i=1}^{\ulT} \Delta^2 \lesssim \| P_N^\perp A^{-\frac{\delta}{4}} \|^2_{\cL(H)} \Delta,
			\end{align*}
			where we, in the last step, followed the approach to analyzing $\tilde{\Delta}W_{I_h}$ given in Lemma~\ref{lem: RV-SARCV asympt equiv}, taking $\gamma = 1$ into account. We also chose $\delta > 0$ so small that $\|A^{\frac{\delta}{4}} I_h A^{-\frac{d +\epsilon}{4}} \|^2_{\cL_2(H)} < \infty$ by Lemma~\ref{lem:obs-hs}. Since $A^{-\delta/4}$ is compact, the sequence $\Delta^{-1/2} (SARCV_{T}^{\Delta,I_h}-\mathbb E[SARCV_{T}^{\Delta,I_h}])$ is also tight by Corollary D.3 from \cite{BSV2022supplement}.
			This proves tightness of the term on the left of \eqref{eq: sarcv discretization}.

			It is now left to prove that finite-dimensional distributions converge to $0$. This follows immediately by 
			Lemma~\ref{prop:temp-SARCV-diff-proj}, 	since for 	any $u, v \in H$,
			\begin{align*} 
				&\Big\| \langle \Delta^{-\frac 12}\big( \big(\mathrm{SARCV}^{\Delta,O}_{T} - \E[\mathrm{SARCV}^{\Delta,O}_{T}] \big) \\
				&\quad-\big(\mathrm{SARCV}^{\Delta,I}_{T} - \E[\mathrm{SARCV}^{\Delta,I}_{T}] \big)\big), u \otimes v \rangle_{\mathcal{L}_2(H)} \Big\|_{L^2(\Omega, \R)} \\
				&\le \Delta^{-\frac 12}\left\| \langle \mathrm{SARCV}_T^{\Delta, I_h} - \mathrm{SARCV}_T^{\Delta, I}, u \otimes v \rangle_{\mathcal{L}_2(H)} \right\|_{L^2(\Omega, \R)} \notag\\
				&\lesssim  h^{(d+\epsilon)/2}  + \Delta^{\frac{1-\tilde \epsilon}{2}} h^{\min(\psi,2)}.
			\end{align*}
			This proves the claim.
		\end{proof}
		
		We can now prove the central limit theorem \ref{thm: CLT}.
		\begin{proof}[Proof of Theorem \ref{thm: CLT}]
			We first show that it is enough to prove the claim when $X_0\equiv 0$, and recall the notation
			$Y_t= X_t-S(t)X_0$. We have
			\begin{align*}
				RV_{T}^{\Delta,O}
				=  &\sum_{i=1}^{\ulT} (OY_{i\Delta}-OY_{(i-1)\Delta})^{\otimes 2}\\
				&+ \sum_{i=1}^{\ulT} (OY_{i\Delta}-OY_{(i-1)\Delta})\otimes (OE(\Delta)S((i-1)\Delta)X_0)\\
				&+\sum_{i=1}^{\ulT} (OE(\Delta)S((i-1)\Delta)X_0)\otimes(OY_{i\Delta}-OY_{(i-1)\Delta})\\
				&+ \sum_{i=1}^{\ulT}(OE(\Delta)S((i-1)\Delta)X_0)^{\otimes 2}.
			\end{align*}
			We need to prove that the  last three summands on the right are $o_p(\Delta^{\frac 12})$. A simple application of the triangle inequality and  Hölder's inequality shows that 
			\begin{align*}
				& \left\|\sum_{i=1}^{\ulT} (OY_{i\Delta}-OY_{(i-1)\Delta})\otimes (OE(\Delta)S((i-1)\Delta)X_0)\right.\\
				&\left. +\sum_{i=1}^{\ulT} (OE(\Delta)S((i-1)\Delta)X_0)\otimes(OY_{i\Delta}-OY_{(i-1)\Delta})\right.\\
				&\left. + \sum_{i=1}^{\ulT}(OE(\Delta)S((i-1)\Delta)X_0)^{\otimes 2}\right\|_{\cL_2(H)}\\
				\leq &\left(\sum_{i=1}^{\ulT} \|OY_{i\Delta}-OY_{(i-1)\Delta}\|^2\right)^{\frac 12}\left(\sum_{i=1}^{\ulT} \|OE(\Delta)S((i-1)\Delta)X_0\|^2\right)^{\frac 12}\\
				&+\sum_{i=1}^{\ulT} \|OE(\Delta)S((i-1)\Delta)X_0\|^2.
			\end{align*}
			It is, therefore, enough to show
			\begin{equation}\label{eq: negligibility of the initial condition-i}
				\Delta^{-\frac 12} \sum_{i=1}^{\ulT} \|OE(\Delta)S((i-1)\Delta)X_0\|^2=o_p(1),
			\end{equation}
			and
			\begin{equation}\label{eq: negligibility of the initial condition-ii}
				\Delta^{-\frac 12} \sum_{i=1}^{\ulT} \|OY_{i\Delta}-OY_{(i-1)\Delta}\|^2=\mathcal O_p(1).
			\end{equation}
			We first prove \eqref{eq: negligibility of the initial condition-i}.
			For that, we apply Lemma \ref{lem:X0} to obtain 
			\begin{align*}
				&	\Delta^{-\frac 12} \sum_{i=1}^{\ulT}\left\|O E(\Delta)S((i-1)\Delta)X_0\right\|_{H}^2\\
				&\lesssim \begin{cases}
					\zeta_\Delta \Delta^{\min(\iota-1/2,1/2)}, & O\neq  I_h,\\
					\Delta^{\min(\iota-1/2,1/2)}+\Delta^{-\frac 12} h^{2\min(\iota,2)}, & O=I_h.
				\end{cases}
			\end{align*}
			
			When $O\neq I_h$, the assumption that $\iota \ge 1/2$ yields convergence to $0$. When $O=I_h$, by the assumption that $\iota >\psi$ and $h=\mathcal O(\Delta^{\frac {1+\tilde \epsilon}{2\min(\psi,2)}})$, we find $\Delta^{-\frac 12} h^{2\min(\iota,2)}=\mathcal O(\Delta^{1/2+\tilde \epsilon})$ proving convergence to $0$.
			
			The proof of \eqref{eq: negligibility of the initial condition-ii} follows by Lemma \ref{lem: second moments of increments} since  $OY_{i\Delta}-OY_{(i-1)\Delta}= \tilde \Delta W_O^i+ E_O^{i-1,0}$ and by Assumption $\beta \geq 1$ in the case that $O=I_h$.  
			
			From here on, we assume that $X_0\equiv 0$.
			Observe that $$\mathbb E[RV_{T}^{\Delta,O}]= \mathbb E[SARCV_{T}^{\Delta,O}]+\sum_{i=1}^{\ulT} \mathbb E \left[(OE(\Delta)Y_{(i-1)\Delta})^{\otimes 2}\right].$$
			We decompose the error by
			\begin{align*}
				&RV_{T}^{\Delta,O}-Q\\
				&\quad= (RV_{T}^{\Delta,O}-\mathbb E[RV_{T}^{\Delta,O}])-(SARCV_{T}^{\Delta,O}-\mathbb E[SARCV_{T}^{\Delta,O}])\\
				&\qquad +\sum_{i=1}^{\ulT} \mathbb E \left[(OE(\Delta)X_{(i-1)\Delta})^{\otimes 2}\right]\\
				&\qquad +  (SARCV_{T}^{\Delta,O}-Q)\\
				&\quad= (RV_{T}^{\Delta,O}-\mathbb E[RV_{T}^{\Delta,O}])-(SARCV_{T}^{\Delta,O}-\mathbb E[SARCV_{T}^{\Delta,O}])\\
				&\qquad +\left(\sum_{i=1}^{\ulT} \mathbb E \left[(O E(\Delta)X_{(i-1)\Delta})^{\otimes 2}\right]- \mathbb E \left[(E(\Delta)X_{(i-1)\Delta})^{\otimes 2}\right]\right)\\
				&\qquad +\left(\sum_{i=1}^{\ulT} \mathbb E \left[(E(\Delta)X_{(i-1)\Delta})^{\otimes 2}\right]\right)\\
				&\qquad +  \left(SARCV_{T}^{\Delta,O}-\mathbb E[SARCV_{T}^{\Delta,O}])-(SARCV_{T}^{\Delta,I}-\mathbb E[SARCV_{T}^{\Delta,I}])\right)\\
				&\qquad+ (SARCV_{T}^{\Delta,I}-\mathbb E[SARCV_{T}^{\Delta,I}])\\
				&\qquad+(\mathbb E[SARCV_{T}^{\Delta,O}]-\mathbb E[SARCV_{T}^{\Delta,I}])\\
				&\quad +(\mathbb E[SARCV_{T}^{\Delta,I}]-Q)\\
				&\quad=  \mathrm{I} + \mathrm{II} + \mathrm{III} + \mathrm{IV} + \mathrm{V} + \mathrm{VI} + \mathrm{VII}.
			\end{align*}
			Terms I, II, III, IV, VI and VII are asymptotically negligible with respect to the Hilbert-Schmidt topology. 
			
			Specifically, in the case $O\neq I_h$, we have $\mathrm{I}= o_p(\Delta^{\frac 12})$ by Lemma~\ref{lem: RV-SARCV asympt equiv}, $\mathrm{II}= \mathcal O_p(h^{\min(\gamma,1)})=o_p(\Delta^{\frac 12})$ by Lemma~\ref{lem:space-conv}\ref{lem:space-conv:2} since $h=o(\Delta^{\frac 12})$ if $O=P_h$, $\mathrm{III}=o(\Delta^{1/2})$ by  Lemma~\ref{lem:time-conv}\ref{lem:time-conv:2} and Assumption \ref{ass:Ph} with $\gamma \geq 1$, $\mathrm{IV}=o_p(\Delta^{\frac 12})$ by Lemma \ref{lem: discret of SARCV}, $\mathrm{VI}=\mathcal O_p(h^{\min(\gamma,1)})=o_p(\Delta^{\frac 12})$ by Lemma~\ref{lem:space-conv}\ref{lem:space-conv:1} since $h=o(\Delta^{\frac 12})$ if $O=P_h$ and   $\mathrm{VII} = o(\Delta^{1/2})$ by Lemma~\ref{lem:time-conv}\ref{lem:time-conv:1}. 
			
			In the case $O=I_h$,
			Lemma~\ref{lem: RV-SARCV asympt equiv} yields $\mathrm{I} = o_p(\Delta^{1/2})$. By Lemma~\ref{lem:space-conv}\ref{lem:space-conv:2}, $\mathrm{II} =\mathcal O_p(\Delta^{-\frac{\tilde \epsilon}{2}} h^{\min(\psi,2)})=o_p(\Delta^{-\frac 12})$ by coupling $h=o(\Delta^{\frac{1+\tilde \epsilon}{2\min(\phi,2)}})$.  By Lemma~\ref{lem:time-conv}\ref{lem:time-conv:2}, $\mathrm{III} = o(\Delta^{1/2})$. By Lemma~\ref{lem: discret of SARCV}, $\mathrm{IV}=o_p(\Delta^{-\frac 12})$. By Lemma~\ref{lem:space-conv}\ref{lem:space-conv:1}, we have $\mathrm{VI} \le \Delta^{-\frac{\tilde \epsilon}{2}} h^{\min(\psi,2)}$. By Lemma~\ref{lem:time-conv}\ref{lem:time-conv:1},  $\mathrm{VII} = o(\Delta^{1/2})$.

			This yields 
			\begin{align*}
				\Delta^{-\frac 12} (RV_{T}^{\Delta,O}-Q)=   \Delta^{-\frac 12} (SARCV_{T}^{\Delta,I}-\mathbb E[SARCV_{T}^{\Delta,I}])+ o_p(1),
			\end{align*}
			such that the asymptotic law of $\Delta^{-\frac 12} (RV_{T}^{\Delta,O}-Q)$ is the same as the one of $ \Delta^{-\frac 12} (SARCV_{T}^{\Delta,I}-\mathbb E[SARCV_{T}^{\Delta,I}])$, which has the desired form due to Theorem D.10 in \cite{BSV2022supplement}. 
		\end{proof}
		
		\subsection{Proofs of Section \ref{Sec: Universal Goodness of fit test}}
		
		\begin{proof}[Proof of Theorem \ref{thm: Universal Goodness of fit- estimated parameters}]
			Assume that the $H_0$ in \eqref{eq: universal test estimated} is true  let $\theta_0\in \Theta$ denote the true parameter. 
			We write $T_n(\theta)=n\|RV_O^{T}/T-Q_{\theta}\|_{\cL_2(H)}^2$ and $q_{\alpha}(\theta)$ for the asymptotic quantile of $T_n(\theta)$ under $H_0$ for all $\theta \in \Theta$.  We need to prove that
			\begin{equation}
				\lim_{n\to\infty} \mathbb P[T_n^* \geq q_{1-\alpha}^*]\leq \alpha.
			\end{equation}
			
			To prove that, we first observe that  defining $T(\theta)=\left\|Q_{\theta_0}-Q_{\theta}\right\|_{\cL_2(H)}$ we have that 
			
			\begin{align*}
				\sup_{\theta \in \Theta}|(\Delta T_n(\theta))^{\frac 12}-T(\theta)|%\leq & \sup_{\theta \in \Theta}\left(\left\|\frac{RV_O^{t_n}}{t_n}-Q_{\theta}\right\|_{\cL_2(H)}-\left\|Q_{\theta_0}-Q_{\theta}\right\|_{\cL_2(H)}\right)\\
				\leq &\left\|\frac{RV_O^{T}}{T}-Q_{\theta_0}\right\|_{\cL_2(H)} ,
			\end{align*}
			which converges to $0$ in mean square as $\Delta \to 0$.  Moreover, $T(\theta_0)=0< T(\theta_0)$ for all $\theta \in \Theta$ with $\theta \neq \theta_0$ since $\theta \mapsto \Theta$ is injective.  Then,   we find that $$\theta^*_n\overset{p}{\longrightarrow}\theta_0\quad \text{ as } n\to \infty$$ by Theorem 5.7 in \cite{vanderVaart1996}.  
			
			Now observe that since $\theta \mapsto Q_{\theta}$ is continuous with respect to the trace norm,  also $\theta\mapsto \Gamma_{\theta}$ is continuous since 
			\begin{align*}
				\|\Gamma_{\theta}-\Gamma_{\rho}\|_{\cL_1(\cL_2(H))}\lesssim \|Q_{\theta}-Q_{\rho}\|_{\cL_1(H)}(\|Q_{\theta}\|_{\cL_1(H)}+\|Q_{\rho}\|_{\cL_1(H)}).
			\end{align*}
			This implies that $\theta\mapsto N(0, \Gamma_{\theta})$ weakly (see e.g.  3.8.15  in \cite{Bogachev1998}) and hence, by the continuous mapping theorem we obtain that $T_n(\theta_n^*)\overset{d}{\longrightarrow}V$ as $n\to \infty$.  
			
			Now fix $\alpha \in (0,1)$ and let $\epsilon > 0$ be arbitrary. We first note that since $q_{1-\alpha}^* \to q_{1-\alpha}$ in probability as $n \to \infty$, and since the limiting distribution of $T_n(\theta_0)$ is continuous, the quantiles $q_{1-\alpha} = F^{-1}_{V}(1-\alpha)$ are strictly decreasing and continuous. Therefore, for any $\epsilon > 0$, we have:
			$$
			\lim_{n \to \infty} \mathbb{P} \left( q_{1-\alpha}^* < q_{1-\alpha-\epsilon} \right) 
			\leq 
			\lim_{n \to \infty} \mathbb{P} \left( \left| q_{1-\alpha}^* - q_{1-\alpha} \right| > q_{1-\alpha} - q_{1-\alpha - \epsilon} \right) 
			= 0.
			$$
			
			Now observe that for all $n$,
			$$
			\mathbb{P} \left( T_n^* \geq q_{1-\alpha}^* \right) 
			\leq 
			\mathbb{P} \left( T_n(\theta_0) \geq q_{1-\alpha}^* \right).
			$$
			We split the right-hand side using the event $\{ q_{1-\alpha}^* < q_{1-\alpha-\epsilon} \}$ to get
			$$
			\mathbb{P} \left( T_n(\theta_0) \geq q_{1-\alpha}^* \right) 
			\leq 
			\mathbb{P} \left( q_{1-\alpha}^* < q_{1-\alpha - \epsilon} \right) 
			+ 
			\mathbb{P} \left( T_n(\theta_0) \geq q_{1-\alpha - \epsilon} \right).
			$$
			
			Taking the limit superior as $n \to \infty$, we use that 
			$\mathbb{P} \left( q_{1-\alpha}^* < q_{1-\alpha - \epsilon} \right) \to 0$,
			and by weak convergence of $T_n(\theta_0)$ to $V$, 
			$$
			\lim_{n \to \infty} \mathbb{P} \left( T_n(\theta_0) \geq q_{1-\alpha - \epsilon} \right) = \alpha - \epsilon.
			$$
			Hence,
			$$
			\limsup_{n \to \infty} \mathbb{P} \left( T_n^* \geq q_{1-\alpha}^* \right) 
			\leq 
			\alpha - \epsilon.
			$$
			
			Since this holds for arbitrary $\epsilon > 0$, we conclude that
			$$
			\limsup_{n \to \infty} \mathbb{P} \left( T_n^* \geq q_{1-\alpha}^* \right) \leq \alpha,
			$$
			which proves that the test has asymptotic level at most $\alpha$.
		\end{proof}
		
	\end{appendix}
	
\end{document}